%% file: ams_final_submission.tex
\documentclass{amsart}
\usepackage{color}
\usepackage[pdftex]{graphicx}
\usepackage [cmtip,arrow]{xy}
\usepackage {pb-diagram,pb-xy} 
\usepackage {amsfonts} 
\usepackage{amsmath}

\newtheorem{theorem}{Theorem}[section]
\newtheorem{lemma}[theorem]{Lemma}
\newtheorem{corollary}[theorem]{Corollary}
\newtheorem{proposition}[theorem]{Proposition}
\theoremstyle{definition}
\newtheorem{definition}[theorem]{Definition}
\newtheorem{example}[theorem]{Example}

\theoremstyle{remark}
\newtheorem{remark}[theorem]{Remark}
\numberwithin{equation}{section}

\newtheorem{algorithm}{\sc Algorithm}
\newtheorem{notation}{\sc Notation}

\newcommand {\R} {\mbox{\rm R}}

\newcommand {\junk}[1]{}
\newcommand {\hide}[1]{}

\newcommand {\s}        {\mbox{\rm sign}}
\newcommand {\D}     {\mbox{\rm D}}

\newcommand {\C}     {\mbox{\rm C}}

\newcommand {\Real}[1]   {\mbox{${\mathbb R}^{#1}$}}
\newcommand {\Complex}[1]   {\mbox{${\mathbb C}^{#1}$}}

 \newcommand {\re}         {\Real{}}
 \newcommand {\complex}   {\Complex{}}

\newcommand {\level}{\mbox{\rm level}}
\newcommand {\Z}  {{\mathbb Z}}
\newcommand {\Q}         {{\mathbb Q}}

\newcommand {\ZZ} {{\rm Z}}
\newcommand {\RR} {{\mathcal R}}

\newcommand {\la}   {{\langle}}
\newcommand {\ra}   {{\rangle}}
\newcommand {\eps} {{\varepsilon}}
\newcommand {\E} {{\rm Ext}}
\newcommand {\dist} {{\rm dist}}

\newcommand {\I} {\mbox{\rm I}}

\newcommand {\Ker}      {\mbox{\rm Ker}}

\def\addots{\mathinner{\mkern1mu
\raise1pt\vbox{\kern7pt\hbox{.}}
\mkern2mu\raise4pt\hbox{.}\mkern2mu
\raise7pt\hbox{.}\mkern1mu}}

\newcommand{\RM}  {\mbox{\rm RM}}

\renewcommand {\Im} {\mbox {\rm Im}}

\newcommand {\spanof} {{\rm span}}
\newcommand {\Tot} {{\rm Tot}}
\newcommand {\ancestor} {{\rm an}}
\newcommand {\F} {{\mathbb F}}
\newcommand {\M} {{\mathcal M}}
\newcommand {\N} {{\mathcal N}}
\newcommand {\HH} {{\rm H}}
\newcommand {\LL} {{\rm L}}
\newcommand {\Ch} {{\rm C}}
\newcommand {\A} {{\mathbb A}}
\newcommand {\Sphere}{\mbox{${\bf S}$}}     
\newcommand {\hocolimit}{{\rm hocolim}} 
\newcommand {\colimit}{{\rm colim}}

\newcommand {\U} {{\mathcal U}}
\newcommand{\dd}{\tilde{d}}

\newcommand{\bigcupdot}
{\mathop{\makebox[0pt]{\hskip 1.4em $\boldsymbol\cdot$}\bigcup}}

\newcommand{\coucou}[1]{\ifvmode\else\marginpar[\hfill$\rhd$]{$\lhd$}\fi
                       $\langle$\textsc{#1}$\rangle$}

\newcommand{\basu}[1]{}
\newcommand {\rp}[1]{}

\begin{document}
\title[Algorithmic Semi-algebraic Geometry and Topology]{
Algorithmic Semi-algebraic Geometry and Topology
-- Recent Progress and Open Problems
\footnote{2000 Mathematics Subject Classification Primary 14P10, 14P25;
Secondary 68W30}
}
\author{Saugata Basu}
\address{School of Mathematics,
Georgia Institute of Technology, Atlanta, GA 30332, U.S.A.}
\email{saugata.basu@math.gatech.edu}
\thanks{The author was supported in part by NSF grant CCF-0634907. 
Part of this work was done while the author was visiting the 
Institute of Mathematics and its Applications, Minneapolis.}


\begin{abstract}
We give a survey of algorithms for computing topological invariants
of semi-algebraic sets with special emphasis on the more 
recent developments in designing algorithms for computing the Betti
numbers of semi-algebraic sets. 
Aside from describing these results, we discuss briefly 
the background as well as the importance of these problems, and also 
describe the main tools from algorithmic semi-algebraic geometry,
as well as algebraic topology,  which make these advances possible.
We end with a list of open problems.
\end{abstract}
\maketitle

\tableofcontents
\section{Introduction}
\label{sec:intro}
This article has several goals. The primary goal is 
to provide the reader with a thorough survey of the current state of 
knowledge on efficient
algorithms for computing topological invariants of semi-algebraic sets --
and in particular their Betti numbers. 
At the same time we want to provide graduate students who intend to pursue
research in the area of algorithmic semi-algebraic geometry, a 
primer on the main technical tools used in the recent developments 
in this area, so that they can start to use these themselves in their own work.
Lastly, for experts in closely related areas who might want to use the results
described in the paper, we want to present self-contained descriptions
of these results in a usable form.

With this in mind  we first give a short introduction 
to the main algorithmic problems in semi-algebraic geometry, 
their history, as well as brief
descriptions of the main mathematical and algorithmic tools used in the
design of efficient algorithms for solving these problems. 
We then provide a more detailed description
of the more recent advances in the area of designing efficient algorithms for
computing the Betti numbers of semi-algebraic sets. Since the design of these
algorithms draw on several new ideas from diverse areas, we describe 
some of the most important ones in some detail for the reader's benefit. 
The goal is to provide the reader with a short but comprehensive 
introduction to the mathematical tools that have proved to be useful in
the area.
The reader who is interested in being
up-to-date with the recent developments in this area, but not interested
in pursuing research in the area, can safely skip the more technical
sections.  
Throughout the survey
we omit most proofs referring the reader to the appropriate references where
such  proofs appear.  

The rest of the paper is organized as follows. In Section \ref{sec:background}
we discuss the background, significance, and history of  algorithmic
problems in semi-algebraic geometry and topology. 
In Section \ref{sec:results} we state some of the recent results in the
field.
In Section 
\ref{sec:algoprelim} we outline a few of the basic algorithmic tools used
in the design of algorithms for dealing with  semi-algebraic sets. These
include the cylindrical algebraic decomposition, as well as the critical
point method exemplified by the roadmap algorithm.
In Section \ref{sec:topbackground} we provide the reader 
some relevant facts and definitions from algebraic topology which 
are used in the more modern algorithms, including definitions
of cohomology of simplicial complexes as well as semi-algebraic sets,
the Nerve Lemma and its generalizations for non-Leray covers,
the descent spectral sequence and the basic properties of homotopy
colimits.
In Section \ref{sec:bettifew} we describe  recent progress in the
design of algorithms for computing the higher Betti numbers of
semi-algebraic sets. In Section \ref{sec:quadratic} we restrict
our attention to sets defined by quadratic inequalities, and describe 
recent progress in the design of efficient algorithms for computing the
Betti numbers of such sets. In Section \ref{sec:arrangements}
we describe a simplified version of an older algorithm for efficiently
computing the Betti numbers of an arrangement -- where the emphasis is
on obtaining tight bounds on the combinatorial complexity only 
(the algebraic part of the complexity being assumed to be bounded by a 
constant). We end by listing some open problems in Section \ref{sec:open}.

\paragraph{{\em Prerequisites}} 
In this survey we are aiming at a wide audience.
We expect that the reader has a basic background in algebra, 
has some familiarity 
with  simplicial complexes and their homology, 
and the theory of NP and \#P-completeness. 
Beyond these we make no additional 
assumption of any prior advanced knowledge of 
semi-algebraic geometry, algebraic topology, or the theory of
computational complexity.

\section{Semi-algebraic Geometry: Background}
\label{sec:background}
\subsection{Notation}
We first fix some notation.
Let $\R$ be a real closed field (for example, the field $\re$ of real numbers
or $\re_{{\rm alg}}$ of real algebraic numbers).
A semi-algebraic subset of $\R^k$ is a set defined by a
finite system of polynomial equalities and inequalities,
or more generally by a Boolean
formula whose atoms are polynomial equalities and inequalities.
Given a finite set ${\mathcal P}$ of polynomials in ${\R}[X_1,\ldots,X_k]$, a
subset $S$ of ${\R}^k$ is ${\mathcal P}$-semi-algebraic
if $S$ is the realization of a Boolean formula with
atoms $P=0$, $P > 0$ or $P<0$ with $P \in {\mathcal P}$.
It is clear that for every semi-algebraic subset $S$ of $\R^k$
there exists
a finite set ${\mathcal P}$ of polynomials in ${\R}[X_1,\ldots,X_k]$
such that $S$ is ${\mathcal P}$-semi-algebraic.
We call a semi-algebraic set  a ${\mathcal P}$-closed semi-algebraic
set if it is defined by a Boolean formula with no negations with atoms
$P=0$, $P \geq 0$, or $P \leq 0$ with $P \in {\mathcal P}$.

For an element $a \in\R$ we let
$$
\s(a) =
\begin{cases}
0 & \mbox{ if }  a=0,\cr
1 & \mbox{ if } a> 0,\cr
-1& \mbox{ if } a< 0.
\end{cases}
$$

A  {\em sign condition}  on
${\mathcal P}$ is an element of $\{0,1,- 1\}^{\mathcal P}$.
For any semi-algebraic set $Z \subset \R^k$ 
the {\em realization of the sign condition
$\sigma$ over $Z$}, $\RR(\sigma,Z)$, is the semi-algebraic set
\label{def:R(Z)}
$$
        \{x\in Z\;\mid\; \bigwedge_{P\in{\mathcal P}} \s({P}(x))=\sigma(P) \},
$$
and in case $Z= \R^k$ we will denote $\RR(\sigma,Z)$ by just $\RR(\sigma)$.
   
If ${\mathcal P}$ is a finite subset of
$\R [X_1, \ldots , X_k]$, we write the set of zeros
of ${\mathcal P}$ in $\R^k$ as
$$
\ZZ({\mathcal P},\R^k)=\{x\in \R^k\mid\bigwedge_{P\in{\mathcal P}}P(x)= 0\}.
$$

We will denote by 
$B_k(0,r)$ the open ball with center 0 and radius $r$ in $\R^k$. 
We will also denote by $\Sphere^k$ the unit sphere in $\R^{k+1}$ centered
at the origin. Notice that these sets are semi-algebraic.

For any 
semi-algebraic set $X$, we denote by $\overline{X}$ the closure of $X$,
which is also a  semi-algebraic set by the Tarksi-Seidenberg principle
\cite{Tarski51,Seidenberg54} (see  \cite{BPRbook06} for a modern treatment).  
The Tarksi-Seidenberg principle
states that the class of semi-algebraic sets is closed under
linear projections or equivalently that the first order theory of the
reals admits quantifier elimination. It is an easy exercise
to verify that the closure of a semi-algebraic set 
admits a description by a quantified first order formula.

For any semi-algebraic set $S$, we will denote by $b_i(S)$ its $i$-th
Betti number, which is the dimension of the
$i$-th cohomology group, $\HH^i(S,\Q)$, taken with rational coefficients,
which in our setting is also isomorphic to the $i$-th homology group,
$\HH_i(S,\Q)$
(see Section \ref{subsec:sahomology} below  for precise definitions of these
groups).
In particular,
$b_0(S)$ is the number of semi-algebraically connected components of $S$.
We will sometimes refer to the sum 
$\displaystyle{b(S) = \sum_{i \geq 0} b_i(S)}$ as the
{\em topological complexity} of a semi-algebraic set $S$.

\begin{remark}
\label{rem:homologyvscohomology}
Departing from usual practice, 
in the description of the algorithms occurring later in this paper
we will mostly refer to the 
cohomology groups instead of the homology groups.
Even though the geometric interpretation of the cohomology groups
is a bit more obscure than that for homology groups
(see Section \ref{subsubsec:homologyvscohomology} below), 
it turns out that from
the point of view of designing algorithms for computing Betti numbers of
semi-algebraic sets (at least for those discussed in this survey)
the usual geometric interpretation of 
homology as measuring the number of ``holes'' or ``tunnels'' etc.
is of little use,
and the main concepts behind these algorithms are better understood
from the cohomological point of view. This is the reason why we emphasize
cohomology over homology in what follows.
\end{remark}
 
\subsection{Main Algorithmic Problems}
\label{subsec:problems}
Algorithmic problems in semi-algebraic geometry typically consist of 
the following. We are given as input a finite family, ${\mathcal P} \subset
\R[X_1,\ldots,X_k]$, as well as a
formula defining a ${\mathcal P}$-semi-algebraic set $S$.
The task is to
decide whether certain geometric and topological properties hold for $S$,
and in some cases also computing certain topological invariants of $S$.
Some of the most basic problems include the following.

Given a ${\mathcal P}$-semi-algebraic set $S \subset \R^k$:
\begin{enumerate}
\item 
decide whether it is empty or not,
\item
given two points $x,y \in S,$ decide if they  are in the
same connected component of $S$ and if so output a
semi-algebraic path in $S$ joining them,
\item
compute semi-algebraic descriptions of the connected components
of $S$,
\item
compute semi-algebraic descriptions of the projection of
$S$ onto some linear subspace of $\R^k$
(this problem is also known as the quantifier elimination
problem for the first order theory of the reals and many
other problems can be posed as special cases of this
very general problem).
\end{enumerate}

At a deeper level we have  problems of more topological flavor, such as:
\begin{itemize}
\item[(5)]
compute  the cohomology groups of $S$,
its Betti numbers, its Euler-Poincar\'e characteristic  etc.,
\item[(6)] 
compute a semi-algebraic triangulation of $S$ 
(cf.  Definition \ref{def:triangulation} below), as well as 
\item[(7)]
compute   a decomposition of $S$ into semi-algebraic smooth pieces of various
dimensions which fit together nicely (a Whitney-regular
stratification).
\end{itemize}

The complexity of an algorithm for solving any of the
above problems is measured in terms of the following
three parameters:
\begin{itemize}
\item
the number of polynomials, $s = \#{\mathcal P}$,
\item
the maximum degree, $d = \max_{P \in {\mathcal P}} \deg(P)$, and 
\item
the number of variables, $k$.
\end{itemize}

\begin{definition}[Complexity]
A typical  input to the algorithms considered in this survey 
will be a set  of polynomials with coefficients in an ordered 
ring $\D$ (which can be taken to be the ring generated by the
coeffcients of the input polynomials).
By complexity of an algorithm we will mean the number of arithmetic
operations (including comparisons) performed by the algorithm
in the ring $\D$.
In case the input polynomials have integer coefficients with bounded bit-size,
then we will often give the bit-complexity, which is the number of bit
operations performed by the algorithm.
We refer the reader to \cite{BPRbook06}[Chapter 8] for a full discussion
about the various measures of complexity.
\end{definition}

Even though the goal is always to design algorithms with 
the best possible complexity in terms of all the parameters $s,d,k$,
the relative importance of the parameters is
very much application dependent. For instance, in applications
in {\em computational geometry} it is the {\em combinatorial} complexity
(that is the dependence on $s$) that is of paramount importance,
the {\em algebraic} part depending on $d$, as well as the dimension $k$, 
are assumed to be bounded by constants. On the other hand in
algorithmic real algebraic geometry, and in applications in complexity
theory, the  algebraic part
depending on $d$ is considered to be equally important.

\subsection{Brief History}
Even though there exist algorithms  for
solving all the above problems, the main research
problem is to design {\em efficient} algorithms
for solving them. 
The  complexity of the first decision procedure 
given by Tarski \cite{Tarski51}  to solve Problems
1 and 4 listed in Section \ref{subsec:problems}
is not elementary recursive, which implies  that the
running time cannot be bounded by a function of
the size of the input which is a fixed tower of exponents.
The first algorithm with a significantly  better worst-case time bound was 
given by Collins \cite{Collins} in 1976. 
His algorithm had a worst case running time 
doubly  exponential in the number of variables.
Collins' method is to obtain a cylindrical algebraic decomposition
of the given semi-algebraic set (see Section \ref{subsec:cad} below for
definition). 
Once this decomposition is computed
most topological questions about semi-algebraic sets such as those
listed in Section \ref{subsec:problems} can be answered.
However, this method involves cascading projections
which involve squaring  of the degrees at each step resulting
in a complexity which is
doubly exponential in the number of variables.

Most of the recent work in algorithmic  semi-algebraic
geometry has focused on obtaining {\em single exponential time}
algorithms -- that is algorithms with complexity of the
order of 
$(s d)^{k^{O(1)}}$ 
rather than $(s d)^{2^k}$. 
An important motivating reason behind the search for such algorithms,
is the following theorem  due to Gabrielov and Vorobjov \cite{GV05} 
(see  \cite{OP,Thom,Milnor,Basu1}, as well as the 
survey article \cite{BPR10}, for work leading up to this result).

\begin{theorem} \cite{GV05}
\label{the:GV}
For a ${\mathcal P}$-semi-algebraic set $S \subset \R^k$,
the sum of the Betti numbers of $S$ 
(refer to Section \ref{sec:topbackground} below for definition) 
is bounded by $(O(s^2d))^k$,
where 
$s = \#{\mathcal P}$, and $d = \max_{P \in {\mathcal P}} \deg(P)$.
\end{theorem}

For the special case of ${\mathcal P}$-closed semi-algebraic sets the
following slightly better bound was known before \cite{Basu1}
(and this bound is used in an essential way in the proof of
Theorem \ref{the:GV}).
Using the same notation as in Theorem \ref{the:GV} above we have

\begin{theorem} \cite{Basu1}
\label{the:B99}
For a ${\mathcal P}$-closed semi-algebraic set $S \subset \R^k$,
the sum of the Betti numbers of $S$ is bounded by $(O(sd))^k$.
\end{theorem}

\begin{remark}
\label{rem:lowerbound}
These bounds are asymptotically tight, as can be already seen from the 
example where each $P \in {\mathcal P}$ is a product of 
$d$ generic polynomials of degree one. The number of connected
components of the ${\mathcal P}$-semi-algebraic
set defined as the subset of $\R^k$
where all polynomials in ${\mathcal P}$ are non-zero is clearly
bounded from below by $(\Omega(sd))^k$.
\end{remark}

Notice also that the above bound has single exponential rather than double 
exponential dependence on $k$.
Algorithms with single exponential complexity have now been given
for several of the problems listed above and there have been 
a sequence of improvements in the complexities of such algorithms. 
We now have  single exponential algorithms for 
deciding emptiness of semi-algebraic sets \cite{Gri88,GV,R92,BPR4},
quantifier elimination \cite{R92,BPR4,Basu2},
deciding connectivity
\cite{Canny,GV92,CGV,GR92,BPR5},
computing descriptions of the connected components \cite{HRS94,BPR9},
computing the Euler-Poincar\'e characteristic 
(see Section \ref{subsubsec:EP} below for definition)
\cite{Basu1,BPR7}, as well 
as the first few (that is, any constant number of) 
Betti numbers of semi-algebraic sets \cite{BPR9,Basu7}. 
These algorithms answer 
questions about the semi-algebraic set $S$ without
obtaining a full cylindrical algebraic decomposition
(see Section \ref{subsec:cad} below for definition), which makes it possible
to avoid having double exponential complexity.
Moreover, polynomial time algorithms are now known for
computing some of these invariants for special classes of 
semi-algebraic sets \cite{Barvinok93,GP,Basu6,Basu8,BZ}. 
We describe some of these new results in greater detail in 
Section \ref{sec:results}.

\subsection{Certain Restricted Classes of Semi-algebraic Sets}
Since general semi-algebraic sets can have exponential topological 
complexity (cf. Remark \ref{rem:lowerbound}), it is natural to consider certain
restricted classes of semi-algebraic sets. One natural class consists of
semi-algebraic sets defined by a conjunction of 
quadratic inequalities.

\subsubsection{Quantitative Bounds for Sets Defined by Quadratic Inequalities}
Since sets defined
by linear inequalities have no interesting topology, sets  defined
by quadratic inequalities can be considered to be the simplest 
class of semi-algebraic sets which can have non-trivial topology. 
Such sets are in fact quite general, since every semi-algebraic set 
can be
defined by a (quantified) formula involving only quadratic polynomials
(at the cost of increasing the number of variables and the size of the
formula). Moreover, as in the case of general semi-algebraic sets,
the Betti numbers of such sets can be exponentially large.
For example, the set  $S \subset \R^k$ defined by 
\[
X_1(1 - X_1) \leq 0, \ldots, X_k(1 - X_k) \leq 0,
\]
has $b_0(S) = 2^k$.

Hence, it is somewhat surprising that for any constant $\ell \geq 0$, 
the Betti numbers
$b_{k-1}(S),\ldots,b_{k-\ell}(S)$, of a basic closed semi-algebraic set
$S \subset \R^k$ defined by quadratic inequalities, are  polynomially 
bounded. 
The following theorem which appears in \cite{Basu4} is derived 
using a bound proved by Barvinok \cite{Barvinok97}
on the Betti numbers of sets defined by few quadratic equations.

\begin{theorem}\cite{Basu4}
\label{the:quadratic}
Let  $\R$ a real closed field and $S \subset \R^k$ be defined by 
\[
P_1 \leq 0,\ldots, P_s \leq 0, \deg(P_i) \leq 2, 1 \leq i \leq s.
\]
Then, for any $\ell \geq 0$,
\[
b_{k-\ell}(S) \leq {s \choose {\ell}} k^{O(\ell)}.
\]
\end{theorem}

Notice that for fixed $\ell$ this gives a polynomial bound on the
highest $\ell$ Betti numbers of $S$ (which could possibly be non-zero). 
Observe also that similar bounds do not hold for sets defined
by polynomials of degree greater than two. For instance, the
set $V \subset \R^k$ defined by the single quartic inequality,
\[
\sum_{i=1}^k X_i^2(X_i - 1)^2 -\varepsilon \geq 0,
\]
will have $b_{k-1}(V) = 2^k$, for all small enough $\varepsilon > 0$.

To see this observe that for all sufficiently small $\varepsilon > 0$,
$\R^k \setminus V $ is defined by 
\[ 
\sum_{i=1}^k X_i^2(X_i - 1)^2 <  \varepsilon
\]
and has $2^k$ connected components
since it retracts onto the set $\{0,1\}^k$. 
It now follows that 
\[
b_{k-1}(V) = b_0(\R^k \setminus V) = 2^k,
\]
where the first equality is a consequence of the well-known
Alexander duality theorem (see \cite[pp. 296]{Spanier}).

\subsubsection{Relevance to Computational Complexity Theory}
Semi-algebraic sets defined by a system of quadratic inequalities 
have a special significance in the theory of computational complexity.
Even though such sets might seem to be the next simplest class of 
semi-algebraic sets after sets defined by linear inequalities, 
from the point of view of 
computational complexity they represent a quantum leap.
Whereas there exist (weakly) polynomial time algorithms for
solving linear programming, solving quadratic
feasibility problem is provably hard.
For instance, it follows from an easy reduction from the problem of 
testing feasibility of a real quartic equation in many variables, that 
the problem of testing whether a system of quadratic inequalities 
is feasible is  ${\rm NP}_{\rm R}$-complete in the Blum-Shub-Smale
model of computation (see \cite{BCSS}). 
Assuming the input polynomials to have integer
coefficients, the same problem is NP-hard in the classical Turing machine
model, since it is also not difficult to see that the Boolean satisfiability 
problem can be posed as the problem of deciding whether a certain 
semi-algebraic set  defined by  quadratic inequalities is empty or not.

Counting the number of connected components of such sets is even harder.
In fact, it is shown in \cite{Basu8} that 
for $\ell \leq \log k$,
computing the $\ell$-th Betti number of 
a basic semi-algebraic set defined by quadratic inequalities in $\R^k$
is $\#$P-hard. 
In contrast to these hardness results, the polynomial bound on the top 
Betti numbers of sets defined by quadratic inequalities gives rise to
the possibility that these might in fact be computable in polynomial time. 

\subsubsection{Projections of Sets Defined by Few Quadratic Inequalities}
A case of intermediate complexity between semi-algebraic sets defined by 
polynomials of higher degrees and sets defined by a fixed number of
quadratic inequalities  is obtained by considering  linear
projections of such sets.
The operation of linear projection of semi-algebraic sets plays a very
significant role in algorithmic semi-algebraic geometry.  It is a
consequence of the Tarski-Seidenberg principle (see for instance
\cite[Theorem 2.80]{BPRbook06}) 
that the image of a semi-algebraic set under a
linear projection is semi-algebraic, and designing efficient
algorithms for computing properties of projections of semi-algebraic
sets (such as its description by a quantifier-free formula) is a
central problem of the area and is a very well-studied topic (see for
example \cite{R92,BPR4,Basu2} or \cite[Chapter 14]{BPRbook06}).  
However, the
complexities of the best algorithms for computing descriptions of
projections of general semi-algebraic sets is single exponential in
the dimension and do not significantly improve when restricted to the
class of semi-algebraic sets defined by a constant number of quadratic
inequalities.  
Indeed, any semi-algebraic set can be realized as the projection of a
set defined by quadratic inequalities, and it is not known whether
quantifier elimination can be performed efficiently when the number of
quadratic inequalities is kept constant.  However, it is
shown in \cite{BZ} that, 
with a fixed number of inequalities, the projections of
such sets are topologically simpler than projections of general
semi-algebraic sets.

More precisely, let $S \subset
\re^{k+m}$ be a closed and bounded  semi-algebraic set defined by
$P_1 \geq 0, \ldots, P_\ell \geq 0,$ with $P_i \in
\re[X_1,\ldots,X_k,Y_1,\ldots,Y_m],
\deg(P_i) \leq 2, \; 1 \leq i \leq \ell$. (For technical reasons,
which we do not delve into, it is necessary in this case to restrict
ourselves to the case where $\R = \re$.)
Let $\pi:\re^{k+m} \rightarrow \re^m$ be the projection onto the
last $m$ coordinates.  In what follows, the number of inequalities,
$\ell$, used in the definition of $S$ will be considered as fixed.
Since, $\pi(S)$ is not necessarily describable using only
quadratic inequalities, the bound in 
Theorem \ref{the:quadratic} does
not hold for $\pi(S)$ and $\pi(S)$ can in principle be quite
complicated.  Using the best known complexity estimates for quantifier
elimination algorithms over the reals (see~\cite[Chapter 14]{BPRbook06}), 
one gets
single exponential (in $k$ and $m$) bounds on the degrees and the
number of polynomials necessary to obtain a semi-algebraic description
of $\pi(S)$.  In fact, there is no known algorithm for computing a
semi-algebraic description of $\pi(S)$ in time polynomial in $k$ and
$m$. Nevertheless, we know
that for any constant $q > 0$, the sum of the first $q$
Betti numbers of $\pi(S)$ is bounded by a polynomial in $k$ and $m$.

\begin{theorem}\cite{BZ}
\label{the:projquad}
Let $S \subset \re^{k+m}$ be a closed and bounded 
semi-algebraic set defined by 
\[
P_1 \geq  0, \ldots, P_\ell \geq 0, 
P_i \in \re[X_1,\ldots,X_k,Y_1,\ldots,Y_m], 
\deg(P_i) \leq 2, \; 1 \leq i \leq \ell.
\]
Let $\pi:\re^{k+m} \rightarrow \re^m$ be the projection onto the last
$m$ coordinates. For any $q >0,\; 0 \leq q \leq k$,
\begin{equation}
\label{eqn:projquad}
\sum_{i=0}^q b_i(\pi(S)) \leq (k + m)^{O(q\ell)}.
\end{equation}
\end{theorem}

This suggests, from the point of view of designing efficient
(polynomial time) algorithms in semi-algebraic geometry, that
images under linear projections of semi-algebraic sets
defined by a constant number of quadratic inequalities, 
are simpler than general semi-algebraic sets. So they
should be the next natural class of sets to consider,
after sets defined by linear and quadratic inequalities.

\subsection{Some Remarks About the Cohomology Groups}
Since in this survey we focus mainly on the algorithmic problem
of computing the Betti numbers of semi-algebraic sets,
which are the dimensions of the various cohomology (also homology) groups
of such sets,
it is perhaps worthwhile to say a few words about 
our motivations behind computing them, and also their connections with
other parts of mathematics, especially with computational complexity theory.

\subsubsection{Motivation behind computing the zero-th Betti number}
The algorithmic problems of deciding whether a given semi-algebraic set is 
empty or if it is connected,
have obvious applications in many different areas of science and engineering.
(Recall that the number of connected components of a semi-algebraic
set $S$ is equal to its zero-th Betti number, $b_0(S)$.)
For instance, in robotics,
the configuration space of a robot can be modeled as a semi-algebraic set.
Similarly,
in molecular chemistry the 
conformation space of a molecule with  constraints on bond lengths and angles
is a semi-algebraic set.
In both these cases understanding connectivity information is important:
for solving motion planning problem in robotics,
or for determining possible molecular conformations in molecular
chemistry. 

\subsubsection{The higher Betti numbers}
The higher cohomology groups of semi-algebraic sets,
which measure higher dimensional connectivity,
do not appear to have such obvious applications.
Nevertheless, 
there exist several reasons why the problem of
computing the higher homology groups of semi-algebraic sets is an
important problem and we mention a few of these below.
  
Firstly, the algorithmic problem of pinning down the exact topology 
of any given topological space, such as a semi-algebraic set 
in $\re^k$, is an exceedingly difficult problem. In fact, the general problem
of determining if two given spaces are homeomorphic is undecidable
\cite{Markov}. 
In order to get around this difficulty, mathematicians since  the time of 
Poincar\'e 
have devised more easily computable (albeit weaker) 
invariants of topological spaces. 
One reason that cohomology groups are so important, is that 
unlike other topological invariants, they are readily computable --
they allow one to
discard a large amount of information regarding the topology of a given
space, while retaining just enough to derive important 
{\em qualitative} information about the space in question.
For instance, in the case of semi-algebraic sets, 
the dimensions of the cohomology groups also known as the Betti numbers,
determine qualitative information about the set, such as
connectivity (in the usual sense), number of holes and/or tunnels (i.e. higher
dimensional connectivity), its Euler-Poincar\'e characteristic 
(a discrete valuation with properties analogous to those of 
volume) etc.

Secondly, the reach of cohomology theory 
is not restricted to the continuous domain (such as the study of algebraic
varieties in  ${\mathbb C}^k$ or semi-algebraic sets in ${\mathbb R}^k$).
As a consequence of a series of astonishing theorems (conjectured by Andre Weil
\cite{Weil}
and proved by Deligne \cite{Deligne1,Deligne2}, Dwork \cite{Dwork} et al.), 
it turns out that the number of solutions
of systems of polynomial equations over a finite field, ${\mathbb F}_q$,
in algebraic extensions of ${\mathbb F}_q$, is governed by the dimensions
of certain (appropriately defined) cohomology groups of the 
associated variety (see below). In this way, 
cohomology theory plays analogous roles in the discrete 
and continuous settings.

Finally, the algorithmic problem of computing the cohomology
groups of semi-algebraic sets is important from the viewpoint of 
computational complexity theory because of the following.
It is easily seen that the classical NP-complete problem in
discrete complexity theory, the Boolean satisfiability problem, is polynomial
time equivalent
to the problem of deciding whether a given system of polynomial equations in 
many variables over a finite field (say $\Z/2\Z$) has a solution.  
The real (as well as the complex) analogue of this problem has been 
proved to be
NP-complete in the real (resp. complex) version of Turing machines, namely
the Blum-Shub-Smale machine (see \cite{BCSS}). 
The algebraic variety defined by
a system of polynomial equations clearly has further structure apart from
being merely empty or non-empty as a set. In the discrete case, 
we might want to count the number of solutions -- and this turns out to 
be a $\#$P-complete problem.
Recently, a $\#P$-completeness theory has been proposed for the BSS model
as well \cite{BC, BC2} -- 
and the natural $\#$P complete problem in this context
is  computing the Euler-Poincar\'e characteristic of a given variety
(the Euler-Poincar\'e characteristic
being a discrete valuation is the ``right'' notion of cardinality 
for  infinite sets  in this context). 

If one is interested in more information about the variety, 
then in the discrete case
one could ask to count the number of solutions of the given system of
polynomials not just over the ground field ${\mathbb F}_q$, 
but in every algebraic extension,
${\mathbb F}_{q^n}$ of the ground field. 
Even though this appears to be an infinite 
sequence of numbers, its exponential generating function (the so called
zeta-function of the variety) turns out to be a rational function 
(conjectured by Weil \cite{Weil}, and proved by  Dwork \cite{Dwork}) 
of the form,
\[
Z(t) = \frac{P_1(t)P_3(t)\ldots P_{2m-1}(t)}{P_2(t)P_4(t)\ldots P_{2m}(t)},
\]
where each $P_i$ is a polynomial with coefficients in a field of
characteristic $0$,
and the degrees of the polynomials $P_i(t)$ are the dimensions of 
(appropriately defined) cohomology groups associated to the variety $V$
defined by the given system of equations. In the real and complex setting, 
the ordinary topological Betti numbers are considered some of the most 
important computable invariants of varieties and carry important 
topological information. Thus, the algorithmic problem of computing Betti
numbers of constructible sets or varieties, is a 
natural extension of some of the basic problems appearing in computational 
complexity theory 
-- namely deciding whether a given system of polynomial equation is 
satisfiable, and counting the number of solutions. This is true in both the
discrete and continuous settings. Even though, in this survey we concentrate
on the latter, some of the techniques developed in this context
conceivably have applications in the discrete case as well.

Also note that, by considering a complex variety $V \subset \complex^k$ as a
real semi-algebraic set in $\re^{2k}$, all results discussed in this survey
extend directly (with the same asymptotic complexity bounds) to the 
corresponding problems (of computing the Betti numbers) for  
complex algebraic varieties, and more generally for constructible subsets
of $\complex^k$.

\section{Recent Algorithmic Results}
\label{sec:results}
In this section we list 
some of the recent progress on the algorithmic problem
of determining the Betti numbers of semi-algebraic sets.

\begin{itemize}
\item
In \cite{BPR9}, an algorithm with single exponential complexity 
is given  for computing the first Betti number of semi-algebraic sets
(see Section \ref{sec:closedcase} below). 
Previously, only the zero-th Betti number
(i.e. the number of connected components) could be computed in single 
exponential time. Another important result contained in this paper is 
the homotopy equivalence between an arbitrary semi-algebraic set,
and a closed and bounded one 
(which is defined using infinitesimal perturbations of the polynomials
defining the original set)
obtained by a construction due to Gabrielov and Vorobjov \cite{GV05}. 
It was conjectured in \cite{GV05} that
these sets are homotopy equivalent. 
This result is important by itself since it allows, for instance,
a single exponential time reduction of the problem of computing 
Betti numbers of arbitrary 
semi-algebraic sets to the same problem for closed and bounded ones.

\item
The above result is generalized in \cite{Basu7}, 
where a single exponential time algorithm is given 
for computing the first $\ell$ Betti
numbers of semi-algebraic sets, where $\ell$ is allowed to be any
constant (see Section \ref{sec:covering} below). 
More precisely, an algorithm is described that takes as input a description
of a ${\mathcal P}$-semi-algebraic set  $S \subset \R^k$,
and outputs the first $\ell+1$ Betti numbers of $S$,
$b_0(S),\ldots,b_\ell(S).$ The complexity of the algorithm is 
$(s d)^{k^{O(\ell)}},$ 
where $s  = \#({\mathcal P})$ and $d = \max_{P\in {\mathcal P}}{\rm deg}(P),$
which is single  exponential in $k$ for $\ell$ any constant.

\item
In \cite{Basu8}, a {\em polynomial}  time algorithm is given for computing a
constant number of the top Betti numbers of semi-algebraic sets defined
by quadratic inequalities. If the number of inequalities is fixed then
the algorithm computes all the Betti numbers in polynomial time
(see Section \ref{subsec:comp} below).
More precisely, an algorithm is described which takes as input a 
semi-algebraic set, $S$, defined by 
$P_1 \geq 0,\ldots,P_s \geq 0$, where each $P_i \in \R[X_1,\ldots,X_k]$
has degree $\leq 2,$
and computes the top $\ell$ Betti numbers of $S$,
$b_{k-1}(S), \ldots, b_{k-\ell}(S),$ in polynomial time. 
The complexity of the algorithm is
$ 
\sum_{i=0}^{\ell+2} {s \choose i} k^{2^{O(\min(\ell,s))}}.
$
For fixed $\ell$, the complexity of the algorithm can be expressed as
$s^{\ell+2} k^{2^{O(\ell)}},$
which is polynomial in the input parameters $s$ and $k$.
For fixed $s$, we obtain by letting $\ell = k$, 
an algorithm for computing all
the Betti numbers of $S$ whose complexity is $k^{2^{O(s)}}$.

\item
In \cite{Basu6},
an algorithm is described which takes as input a closed semi-algebraic set,
$S \subset \R^k$, defined by 
\[
P_1 \geq 0, \ldots, P_\ell \geq 0,
P_i \in \R[X_1,\ldots,X_k], \deg(P_i) \leq 2,
\]
and computes the Euler-Poincar\'e characteristic of $S$
(see Section \ref{subsec:epquadratic} below). The complexity of the
algorithm is $k^{O(\ell)}$.
Previously, algorithms with the same complexity bound were known only
for testing emptiness (as well as computing sample points) of such sets
\cite{Barvinok93, GP}.

\item
In \cite{BZ}, a polynomial time algorithm is obtained for computing a
constant number of the lowest Betti numbers of semi-algebraic sets defined
as the projection of semi-algebraic sets defined by few 
by quadratic inequalities (see Section \ref{subsec:projquad} below). 
More precisely,
let $S \subset \re^{k + m}$ be a closed and bounded  semi-algebraic set
defined by 
$
P_1 \geq 0, \ldots, P_\ell \geq 0,
$
where 
$
P_i \in \re[X_1,\ldots,X_k,Y_1,\ldots,Y_m],$ and $\deg(P_i) \leq 2, 
1 \leq i \leq \ell.
$
Let $\pi$ denote the standard projection from $\re^{k + m}$ onto $\re^m$.
An algorithm is described for computing the
the first $q$ Betti numbers of $\pi(S)$,
whose complexity is
$\displaystyle{
(k+m)^{2^{O((q+1)\ell)}}.
}$ 
For fixed $q$ and $\ell$, the bound is polynomial in $k+m$.

\item
The complexity estimates for all the algorithms mentioned above included
both the combinatorial and algebraic parameters. As mentioned in Section
\ref{sec:background}, in applications in computational geometry the
algebraic part of the complexity is treated as a constant. In this context,
an interesting question is how efficiently can one compute the Betti 
numbers 
of an arrangement of $n$ closed and bounded  semi-algebraic sets,
$S_1,\ldots,S_n \subset \R^k$, where
each $S_i$ is described using a constant number of polynomials with degrees
bounded by a constant. Such arrangements are ubiquitous in 
computational geometry (see \cite{Agarwal}). 
A naive approach using triangulations would entail a complexity of
$O(n^{2^k})$ (see Theorem \ref{the:triangulation} below).
This problem is considered in \cite{Basu5} where
an algorithm is described for computing
$\ell$-th Betti number, $\displaystyle{b_\ell(\bigcup_{i=1}^{n} S_i), \;
0 \leq \ell \leq k-1}$, 
using $O(n^{\ell+2})$ algebraic operations.
Additionally, one has to perform linear algebra on integer matrices of 
size bounded by $O(n^{\ell+2})$ (see Section \ref{sec:arrangements} below). 
All previous algorithms for computing the Betti numbers of arrangements
triangulated the whole arrangement giving rise to a complex of size 
$O(n^{2^k})$
in the worst case. Thus, the complexity of computing
the  Betti numbers (other than the zero-th one)
for these algorithms was $O(n^{2^k})$.
This is the first algorithm for
computing $\displaystyle{b_\ell(\bigcup_{i=1}^{n} S_i)}$
that does not rely on such a global triangulation,
and has a graded complexity which depends on $\ell$.
\item
We should also mention at least one other approach towards computation 
of Betti numbers (of complex varieties) that we do not describe in  
detail in this survey. Using the theory of {\em 
local cohomology} and {\em D-modules}, 
Oaku and Takayama \cite{OT} and Walther \cite{Walther1,Walther2},
have given explicit algorithms for computing a sub-complex of the
algebraic
de Rham complex of the complements of complex affine varieties
(quasi-isomorphic to the full complex but of much smaller size)
from which the Betti numbers of such varieties as well as their
complements can be computed easily using linear algebra.
For readers familiar with de Rham cohomology theory for differentiable
manifolds, the algebraic de Rham complex is an algebraic analogue of the
usual de Rham complex consisting of vector spaces of differential forms. 
The computational complexities of these
procedures are not analyzed very precisely in the papers cited above.
However, these algorithms use Gr\"{o}bner basis computations over non-commutative
rings (of differential operators),
and as such are unlikely to have complexity better than double exponential
(see \cite[Section 2.4]{Walther2}).
Also, these 
techniques are applicable only over algebraically closed fields, and
not immediately useful in the semi-algebraic context  which is our main
interest in this paper, and as such we do not discuss these algorithms
any further. 
\end{itemize}
   
\section{Algorithmic Preliminaries}
\label{sec:algoprelim}
In this section we give a brief overview of the basic 
algorithmic constructions from semi-algebraic geometry that play a
role in the design of more sophisticated algorithms. These include
cylindrical algebraic decomposition (Section \ref{subsec:cad}), 
the critical point method (Section \ref{sec:critical}), and the 
construction of roadmaps of semi-algebraic sets (Section 
\ref{sec:roadmap}).

\subsection{Cylindrical Algebraic Decomposition}
\label{subsec:cad}
As mentioned earlier one fundamental technique for computing topological
invariants of semi-algebraic sets is through {\em Cylindrical Algebraic
Decomposition}. 
Even though the mathematical ideas behind cylindrical algebraic decomposition
were  known before (see for example \cite{Loj2}),
Collins \cite{Collins} was the first to
apply cylindrical algebraic decomposition in the setting of 
algorithmic semi-algebraic geometry. Schwartz and Sharir \cite{SS}
realized its importance in trying to solve the motion planning problem in 
robotics, as well as computing topological properties of semi-algebraic
sets. Variants of the basic 
cylindrical algebraic decomposition have also been used in several papers
in computational geometry. 
For instance in 
the paper by Chazelle  et al. \cite{CEGS91}, 
a truncated version of cylindrical decomposition 
is described whose combinatorial (though not the algebraic) complexity is
single exponential. This result has found several applications in discrete and
computational geometry (see for instance \cite{CEGSW}).

\begin{definition}[Cylindrical Algebraic Decomposition]
\label{5:def:cad}
A cylindrical algebraic decomposition of ${\R}^k$
is a sequence ${\mathcal S}_1,\ldots,{\mathcal S}_k$
where, for each $1\leq i\leq k$, ${\mathcal S}_i$ is a finite
partition of ${\R}^i$ into  semi-algebraic
subsets, called the  cells of level $i$,
which satisfy the following properties:
\begin{itemize}
\item Each cell $S\in {\mathcal S}_1$ is either a point or
 an open interval.
\item For every $1\leq i<k$ and every $S\in {\mathcal S}_i$,
there are finitely
 many continuous
semi-algebraic functions
 $$ \xi_{S,1}<\ldots<\xi_{S,\ell_S}: S\longrightarrow {\R}$$ such that
the cylinder $S\times {\R} \subset {\R}^{i+1}$ is the disjoint union
 of cells of ${\mathcal S}_{i+1}$
which are:
\begin{itemize}
\item either the graph
of one of the functions $\xi_{S,j}$, for $j =
1,\ldots,\ell_S$: $$
\{ (x', x_{j+1}) \in S \times {\R}\mid x_{j+1} = \xi_{S, j}(x') \}\;,
$$
\item or a  band
 of the cylinder bounded from below and from above by the graphs
 of the functions
$\xi_{S,j}$ and $\xi_{S,j+1}$, for $j = 0,\ldots, \ell_S$,
 where we take $\xi_{S,0} =-\infty$
and $\xi_{i,\ell_S+1} = +\infty$:
$$
\{ (x', x_{j+1}) \in S \times {\R}\mid \xi_{S,j}(x')
< x_{j+1} < \xi_{S, j+1}(x')\}\;.
$$
\end{itemize}
\end{itemize}
\end{definition}

\begin{figure}[hbt] 
\centerline{\scalebox{1.0}{
\begin{picture}(500,250)(-140,-10)
\includegraphics{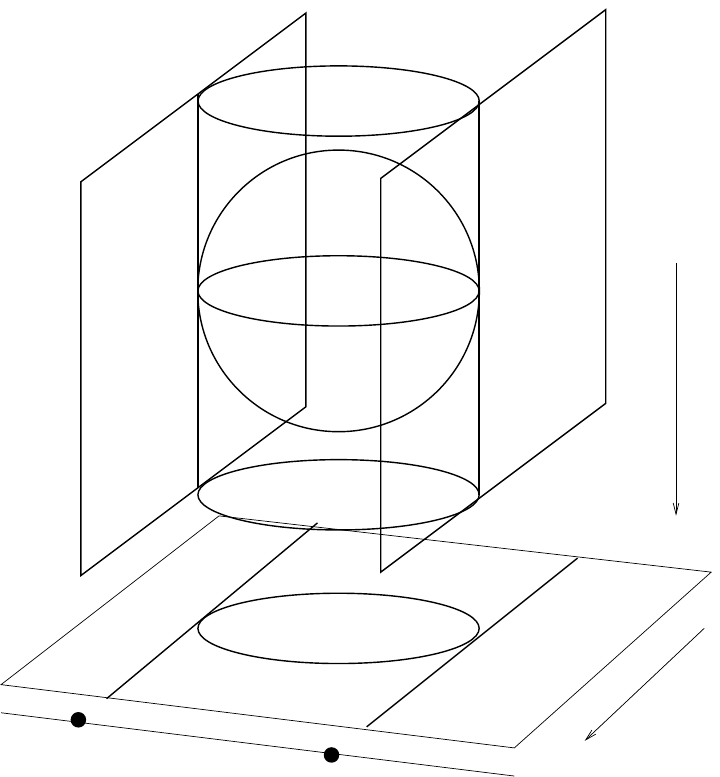}%
\end{picture}
}
}
\caption{Example of cylindrical algebraic decomposition of $\R^3$ adapted
to a sphere.}
\label{fig:cadsphere}
\end{figure}

We note that every cell of a cylindrical algebraic decomposition 
is semi-algebraical-\\
ly homeomorphic
to an open $i$-cube $(0,1)^{i}$ (by
convention, $(0,1)^0$ is a point).

A cylindrical algebraic decomposition adapted to a finite family of
semi-algebraic sets
$T_1,\ldots,T_\ell$
is a cylindrical algebraic 
decomposition of $\R^k$ such that every $T_i$ is a union of cells.
(see Figure \ref{fig:cadsphere}).

\begin{definition}
Given a finite set ${\mathcal P} \subset {\R}[X_1,\ldots,X_k]$, a
subset $S$ of ${\R}^k$ is 
is ${\mathcal P}$-invariant 
if every polynomial $P \in {\mathcal P}$
has a constant sign ($>0$, $<0$, or $=0$) on $S$.
A cylindrical algebraic 
 decomposition of ${\R}^k$ adapted to
\index{Cylindrical decomposition!adapted to ${\mathcal P}$}
\label{5:de:cylindricaladap}
${\mathcal P}$ is a cylindrical algebraic decomposition for which each cell
 $C\in {\mathcal S}_k$ is ${\mathcal P}$-invariant.  It is clear that if
$S$ is
${\mathcal P}$-semi-algebraic, a cylindrical algebraic 
 decomposition adapted to
${\mathcal P}$
is a cylindrical algebraic 
 decomposition adapted to
$S$.
\end{definition}

One important result which underlies most algorithmic applications of 
cylindrical algebraic decomposition is the following
(see \cite[Chapter 11]{BPRbook06} for an easily accessible exposition).

\begin{theorem}
\label{5:the:cad} 
For every finite set ${\mathcal P}$
of polynomials in ${\R}[X_1,\ldots,X_k]$, there is a cylindrical
decomposition of ${\R}^k$ adapted to ${\mathcal P}$.
Moreover, such a decomposition can be computed in time
$(sd)^{2^{O(k)}}$, where $s = \#{\mathcal P}$ and 
$d = \max_{P \in {\mathcal P}} \deg(P).$
\end{theorem}

The cylindrical algebraic decomposition obtained in Theorem \ref{5:the:cad} can
in fact be refined to give a semi-algebraic triangulation of any given
semi-algebraic set within the same complexity bound.

Recall that
\begin{definition}[Semi-algebraic Triangulation]
\label{def:triangulation}
A semi-algebraic triangulation of a semi-algebraic set $S$ is a
simplicial complex $K$ together with a
semi-algebraic homeomorphism from $\vert K\vert$ to $S$.
\end{definition}

The following theorem states that such triangulations can
be computed  for any closed and bounded
semi-algebraic set with double exponential complexity.

\begin{theorem}
\label{the:triangulation}
 Let $S \subset \R^k$ be a closed and bounded semi-algebraic set, and
let $S_1,\ldots, S_q$ be semi-algebraic subsets of $S$.
 There exists a simplicial complex $K$ in
$\R^k$ and a semi-algebraic homeomorphism
$h: \vert K\vert \to S$ such that each $S_j$ is the
union of images by $h$ of open simplices of $K$. Moreover,
 the vertices of $K$ can be chosen
with rational coordinates. 

Moreover, if $S$ and each $S_i$ are ${\mathcal P}$-semi-algebraic sets,
then the semi-algebraic triangulation $(K,h)$ can be computed in time
$(sd)^{2^{O(k)}}$, where $s = \#{\mathcal P}$ and 
$d = \max_{P \in {\mathcal P}} \deg(P).$
\end{theorem}

\subsection{The Critical Point Method}
\label{sec:critical}
As mentioned earlier, all algorithms using cylindrical algebraic decomposition
have double exponential complexity. Algorithms with single exponential
complexity for solving problems in semi-algebraic geometry are 
mostly based on the {\em critical point method}. This method was
pioneered by several researchers including Grigoriev and Vorobjov 
\cite{GV,GV92},
Renegar \cite{R92}, Canny \cite{Canny}, 
Heintz, Roy and Solern\`o \cite{HRS94},
Basu, Pollack and Roy 
\cite{BPR4} amongst others.
In simple terms, the critical point method is nothing but
a method for finding at least one point
in every semi-algebraically  connected
component of an algebraic set.
It can be shown that for a bounded nonsingular algebraic hyper-surface, 
it is possible to change coordinates
so that its projection to the $X_1$-axis has
a finite number of non-degenerate critical points.
These points
provide at least one point in every semi-algebraically
 connected component of the
bounded nonsingular algebraic
hyper-surface.
Unfortunately this is not very useful
in algorithms since it provides no method
for performing this linear change of variables. Moreover
when we deal with the case of a
general algebraic set, which may be unbounded or singular,
this method no longer works.

In order to reduce the general case to the case of bounded nonsingular
algebraic sets, we use  
an important technique in algorithmic semi-algebraic geometry --
namely, perturbation of a given real algebraic set in $\R^k$ using 
one or more infinitesimals. 
The perturbed variety is then defined over a non-archimedean real
closed extension of the ground field -- namely the field of algebraic
Puiseux series in the infinitesimal elements with coefficients in $\R$.

Since the theory behind such extensions might be unfamiliar to some
readers, we introduce here the necessary algebraic background 
referring the reader to \cite[Section 2.6]{BPRbook06} 
for full detail and proofs.

\subsubsection{Infinitesimals and the Field of Algebraic Puiseux Series}
\label{subsec:Puiseux}

\begin{definition}[Puiseux series]
A {\em Puiseux series} in $\eps$ with coefficients in $\R$ 
is a  series of the form
\begin{equation}
\label{2:not:puiseux1}
		{\overline a}=\sum_{i \ge k} a_i \eps^{i/q},
\end{equation}
with $k \in\Z$, $i\in\Z$, $a_i\in \R$, $q$ a positive integer.
\end{definition}

It is a straightforward exercise to verify that
the field of all Puiseux series in $\eps$ with coefficients in $\R$ 
is an ordered field. The order extends the
order of $\R$, and $\eps$ is an infinitesimally small and positive,
i.e. is positive and smaller than any positive $r\in \R$.

\begin{notation}
\label{not:puiseux}
The field of Pusisex series in $\eps$ with coefficients in $\R$ contains
as a subfield, the field of Puiseux series which are algebraic over
$\R[\eps]$.
We denote by $\R\langle \eps\rangle$  the field of algebraic
Puiseux series in $\zeta$ with coefficients in $\R$.
\end{notation}

The following theorem is classical (see for example
\cite[Section 2.6]{BPRbook06} for a proof).

\begin{theorem}
\label{the:Puiseux}
The field $\R\langle \eps\rangle$ is real closed.
\end{theorem}

\begin{definition}[The $\lim_\eps$ map]
\label{def:lim}
When $a \in \R\la \eps \ra$ is bounded by an element of $\R$,
$\lim_\eps(a)$ is the constant term of $a$, obtained by
substituting 0 for $\eps$ in $a$.
\end{definition}

\begin{example}
\label{ex:lim}
A typical example of the application of the $\lim$ map can be
seen in Figures \ref{fig:homotopy1} and  \ref{fig:homotopy2} below. 
The first picture depicts the algebraic set $\ZZ(Q,\R^3)$,
while the second depicts 
the algebraic set $\ZZ(\bar Q,\R\la\zeta\ra^3)$ (where we substituted
a very small positive number for $\zeta$ in order to able display
this set),
where $Q$ and $\bar Q$ are defined by Eqn. (\ref{eqn:examplehomotopy}) and
Eqn. (\ref{11:equation:deform1}) resp.
The algebraic sets $\ZZ(Q,\R^3)$ and $\ZZ(\bar Q,\R\la\zeta\ra^3)$
are related by
\[
\ZZ(Q,\R^3) = \lim_\zeta~\ZZ(\bar Q,\R\la\zeta\ra^3).
\]
\end{example}

Since we will often consider the semi-algebraic sets
defined by the same formula, but over different real closed
extensions of the ground field, the following notation is 
useful.

\begin{notation}
\label{not:extension}
Let $\R'$ be a real closed field containing $\R$.
Given a semi-algebraic set
$S$ in ${\R}^k$, the {\em extension}
of $S$ to $\R'$, denoted $\E(S,\R')$, is
the semi-algebraic subset of ${ \R'}^k$ defined by the same
quantifier free formula that defines $S$.
\end{notation}
The set $\E(S,\R')$ is well defined (i.e. it only depends on the set
$S$ and not on the quantifier free formula chosen to describe it).
This is an easy consequence of the transfer principle.

We now return to the discussion of the critical point method.
In order for the critical point method to work for all algebraic sets,
we associate to a possibly unbounded
algebraic set $Z \subset \R^k$ a bounded algebraic set
$Z' \subset \R\la \eps \ra^{k+1},$
whose semi-algebraically connected components are closely related to those
of $Z$.

Let $Z=\ZZ(Q,\R^k)$ and consider
$$Z'= \ZZ(Q^2+(\eps^2(X_1^2+\ldots+X_{k+1}^2)-1)^2,\R\la \eps \ra^{k+1}).$$
The set $Z'$ is the intersection of the  sphere $S^{k}_\eps$ of
center $0$ and radius $\displaystyle{1 \over \eps}$
with a cylinder
based on the extension of $Z$ to $\R \la \eps \ra$.
The intersection of $Z'$ with the hyperplane $X_{k+1}=0$
is the intersection  of $Z$ with  the  sphere $S^{k-1}_\eps$ of
center $0$ and radius $\displaystyle{1 \over \eps}$.
Denote by $\pi$ the projection from
$\R\la \eps \ra^{k+1}$ to $\R\la \eps \ra^{k}.$

The following proposition which appears in \cite{BPRbook06} then
relates the connected component of $Z$ with
those of $Z'$ and this allows us to reduce the problem of finding
points on every connected component of a possibly unbounded algebraic set
to the same problem on bounded algebraic sets.

\begin{proposition}
\label{11:prop:boundedunbounded}
Let $N$ be a finite number of points meeting  every semi-\\algebraically
connected component of $Z'$.
Then $\pi(N)$ meets  every semi-algebraically
connected component of the extension $\E(Z',\R\la \eps \ra)$
of $Z'$ to $\R\la \eps \ra$.
\end{proposition}

We obtain immediately
using Proposition \ref{11:prop:boundedunbounded}
a method for finding a point in every connected component
of an algebraic set. Note  that these points
have coordinates in the extension $\R \la \eps\ra$
rather than in the real closed field $\R$ we started with.
However, the extension from $\R$ to $\R \la \eps\ra$
preserves  semi-algebraically connected components.

For dealing with possibly singular algebraic sets
we define $X_1$-pseudo-critical points of $\ZZ(Q,\R^k)$ when
$\ZZ(Q,\R^k)$ is a bounded algebraic set.
These pseudo-critical points are a finite set of
points meeting every semi-algebraically connected
 component of $\ZZ(Q,\R^k)$. They are the
limits of the  critical points of the projection to
the $X_1$ coordinate of  a bounded nonsingular
algebraic hyper-surface defined by a particular
infinitesimal perturbation, $\bar Q$, of the polynomial $Q$. Moreover, the
equations defining the critical points of
the projection on the $X_1$ coordinate
on the perturbed algebraic set have a very  special
algebraic structure (they form a Gr\"{o}bner basis 
\cite[Section 12.1]{BPRbook06}), 
which makes possible efficient computation 
of these pseudo-critical values and points. We refer the reader to
\cite[Chapter 12]{BPRbook06} for a full exposition including the definition
and basic properties of Gr\"{o}bner basis.

The deformation $\bar Q$ of $Q$ is defined as follows.
Suppose that $\ZZ(Q,\R^k)$ is contained in the ball of center
$0$ and radius $1/c$.
Let $\bar d$ be an even integer bigger than the degree $d$ of $Q$ and let
\begin{equation}
G_k(\bar d,c)
=c^{\bar d}(X_1^{\bar d}+\cdots+
X_{k}^{\bar d}+X_2^2+\cdots+X _k^2)-(2k-1),
\end{equation}
\begin{equation}
\label{11:equation:deform1}
\bar Q=\zeta G_k(\bar d,c)+{(1-\zeta)
}Q.
\end{equation}

The algebraic set $\ZZ(\bar Q,\R \la \zeta \ra ^k)$
is a bounded and non-singular 
hyper-surface lying infinitesimally close to $\ZZ(Q,\R^k)$
and the critical points of the projection
map onto the $X_1$ co-ordinate restricted to
$\ZZ(\bar Q,\R \la \zeta \ra ^k)$
form a finite set of points.
We take the 
images of these points under $\lim_\zeta$ (cf. Definition
\ref{def:lim}) and
we call the points obtained in this manner  the $X_1$-pseudo-critical points
of $\ZZ(Q,\R^k)$.
Their projections on the $X_1$-axis  are called pseudo-critical values.

\begin{example}
\label{ex:homotopy}
We illustrate the perturbation mentioned above by a concrete example.
Let $k=3$ and $Q \in \R[X_1,X_2,X_3]$ be defined by
\begin{equation}
\label{eqn:examplehomotopy}
Q = X_2^2 - X_1^2 + X_1^4 + X_2^4 + X_3^4.
\end{equation}
Then, $\ZZ(Q,\R^3)$ is a bounded algebraic subset of $\R^3$ shown below
in Figure \ref{fig:homotopy1}.
Notice that $\ZZ(Q,\R^3)$ has a singularity at the origin.
The surface  $\ZZ(\bar Q,\R^3)$ 
with a small positive real number substituted for $\zeta$ is shown in
Figure \ref{fig:homotopy2}. 
Notice that this surface is non-singular, but has a different homotopy
type than $\ZZ(Q,\R^3)$ (it has three connected components compared to
only one of  $\ZZ(Q,\R^3)$). However, the semi-algebraic set bounded
by $\ZZ(\bar Q,\R^3)$  (i.e. the part inside the
larger component but outside the smaller ones)
is homotopy equivalent to $\ZZ(Q,\R^3)$.

\vspace*{2.75cm}
\begin{figure}[hbt]
\centerline{\scalebox{0.3}{
\begin{picture}(500,100)
\includegraphics[bb=0 0 64mm 30mm]{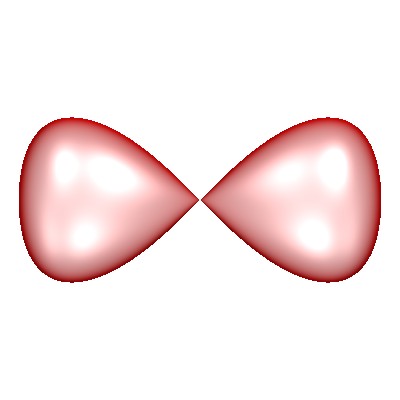}%
\end{picture}
}
}
\caption{The algebraic set $\ZZ(Q,\R^3)$.}
\label{fig:homotopy1}
\end{figure}

\vspace*{2.75cm}
\begin{figure}[hbt]
\centerline{\scalebox{0.3}{
\begin{picture}(500,100)
\includegraphics[bb=0 0 64mm 30mm]{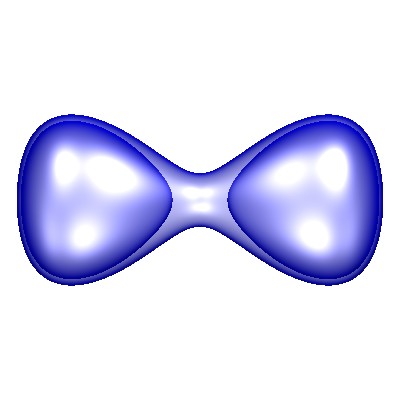}%
\end{picture}
}
}
\caption{The algebraic set $\ZZ(\bar Q,\R^3)$.}
\label{fig:homotopy2}
\end{figure}
\end{example}

By computing algebraic representations (see \cite[Section 12.4]{BPRbook06}
for the precise definition of such a representation)
of the pseudo-critical points
one obtains for any given algebraic set a finite set of points guaranteed to
meet every connected component of this algebraic set. Using some 
more arguments from real algebraic geometry one can also reduce the problem
of computing a finite set of points guaranteed to meet every connected
component of the realization of every realizable sign condition on a given
family of polynomials to finding points on certain algebraic sets defined by
the input  polynomials (or infinitesimal perturbations of these polynomials).
The details of this argument can be found in \cite[Proposition 13.2]
{BPRbook06}. 

The following theorem which is the best result of this kind appears in 
\cite{BPR3}.

\begin{theorem}\cite{BPR3}
\label{13:the:samplealg}
Let $\ZZ(Q,\R^k)$ be an algebraic set
of real dimension $k'$, where $Q$ is a
polynomial in $\R[X_1, \ldots ,X_k]$ of degree at
most $d$, and let ${\mathcal P} \subset\R[X_1, \ldots ,X_k]$ be a set of $s$
polynomials with
each $P\in {\mathcal P}$ also of degree at most
$d$.
Let $\D$ be the
ring generated by the coefficients of $Q$ and the
 polynomials in ${\mathcal P}$.
There is an algorithm which
computes a set of points
meeting every
semi-algebraically connected
component of every  realizable sign condition on $\mathcal P$
over $\ZZ(Q,\R\la \eps,\delta \ra ^k)$. The algorithm has
complexity
$$(k'(k-k')+1)\displaystyle{\sum_{j \le k'}4^j {s  \choose j}
d^{O(k)}=s^{k'} d^{O(k)}}$$
 in $\D$.
There is also an
algorithm providing the
list of
signs of all the polynomials of
$\mathcal P$
at each of these points
 with
complexity
$$
(k'(k-k')+1)\displaystyle s{\sum_{j \le k'}4^j {s  \choose j}
d^{O(k)}=s^{k'+1} d^{O(k)}}
$$ in $\D$.
\end{theorem}

Notice that the combinatorial complexity of the algorithm in Theorem
\ref{13:the:samplealg} depends on the dimension of the variety rather
than that of the ambient space. Since we are mostly concentrating
on single exponential algorithms in this part of the survey, 
we do not emphasize  this aspect too much.

\subsection{Roadmaps}
\label{sec:roadmap}
Theorem \ref{13:the:samplealg} gives a single exponential time
algorithm for testing if a given semi-algebraic set is empty or not.
However, it gives no way of testing if any two sample points computed
by it belong to the same connected component of the given
semi-algebraic set, even though the set of sample points is guaranteed to meet
each such connected component. 
In order to obtain connectivity information in single exponential time
a more sophisticated construction is required -- namely that of a {\em roadmap}
of a semi-algebraic set, which is an one dimensional semi-algebraic subset
of the given semi-algebraic set which is non-empty and connected inside
each connected component of the given set. Roadmaps
were first introduced by Canny \cite{Canny}, but similar constructions
were considered as well by Grigoriev and Vorobjov \cite{GV92} and
Gournay and Risler \cite{GR92}. 
Our exposition below follows that in \cite{BPR5,BPRbook06} where the
most efficient algorithm for computing roadmaps is given.
The notions of pseudo-critical points and values defined above play a
critical role in the design of efficient algorithms for computing 
roadmaps of semi-algebraic sets.

We first define a roadmap  of a semi-algebraic set. 
We use the following notation.
We  denote by $\pi_{1\ldots j}$ the projection,
$x \mapsto (x_1,\ldots,x_j).$
Given a set $S \subset \R^k$ and $y \in \R^j$, we denote by 
$S_y=S \cap \pi_{1 \ldots j}^{-1}(y)$.

\begin{definition}[Roadmap of a semi-algebraic set]
\label{15:def:roadmap}
Let $S \subset \R^k$ be a semi-algebraic set.
A {\em roadmap}
\index{Roadmap} for $S$
 is a semi-algebraic set $M$
of dimension at most one contained in $S$
which satisfies the following roadmap
conditions:
\begin{itemize}
\item ${\rm RM}_1$ For every semi-algebraically
connected component $D$ of $S$,
$D \cap M$ is semi-algebraically connected.
\item ${\rm RM}_2$ For every $x \in {\R}$ and
for every semi-algebraically connected component $D'$
of $S_x$, $D'\cap  M \neq \emptyset.$
\end{itemize}
\end{definition}

We describe the construction of a roadmap
$\RM(\ZZ(Q,\R^k),{\mathcal N})$ for a bounded algebraic set $\ZZ(Q,\R^k)$
 which contains a finite set of
points ${\mathcal N}$ of $\ZZ(Q,\R^k)$. A precise description of how the
construction can be performed algorithmically
can be found in \cite{BPRbook06}.
We should emphasize here that $\RM(\ZZ(Q,\R^k),{\mathcal N})$ denotes the
semi-algebraic set output by the specific algorithm described below which
satisfies the properties stated in Definition \ref{15:def:roadmap}
(cf. Proposition \ref{15:prop:rm}).

Also, in order to understand the roadmap algorithm it is easier to first
concentrate on the case of a bounded and non-singular real algebraic
set in $\R^k$ (see Figure \ref{fig:torus2} below). In this case several
definitions get simplified. For example, the pseudo-critical values
defined below are in this case ordinary critical values of the projection map
on the first co-ordinate. However, one should keep in mind that even 
if one starts with a bounded non-singular algebraic set, the input to
the recursive calls corresponding to the critical sections (see below)
are necessarily singular and thus it is not possible to treat the 
non-singular case independently.
\begin{center}
\begin{figure}
\includegraphics[height=7cm]{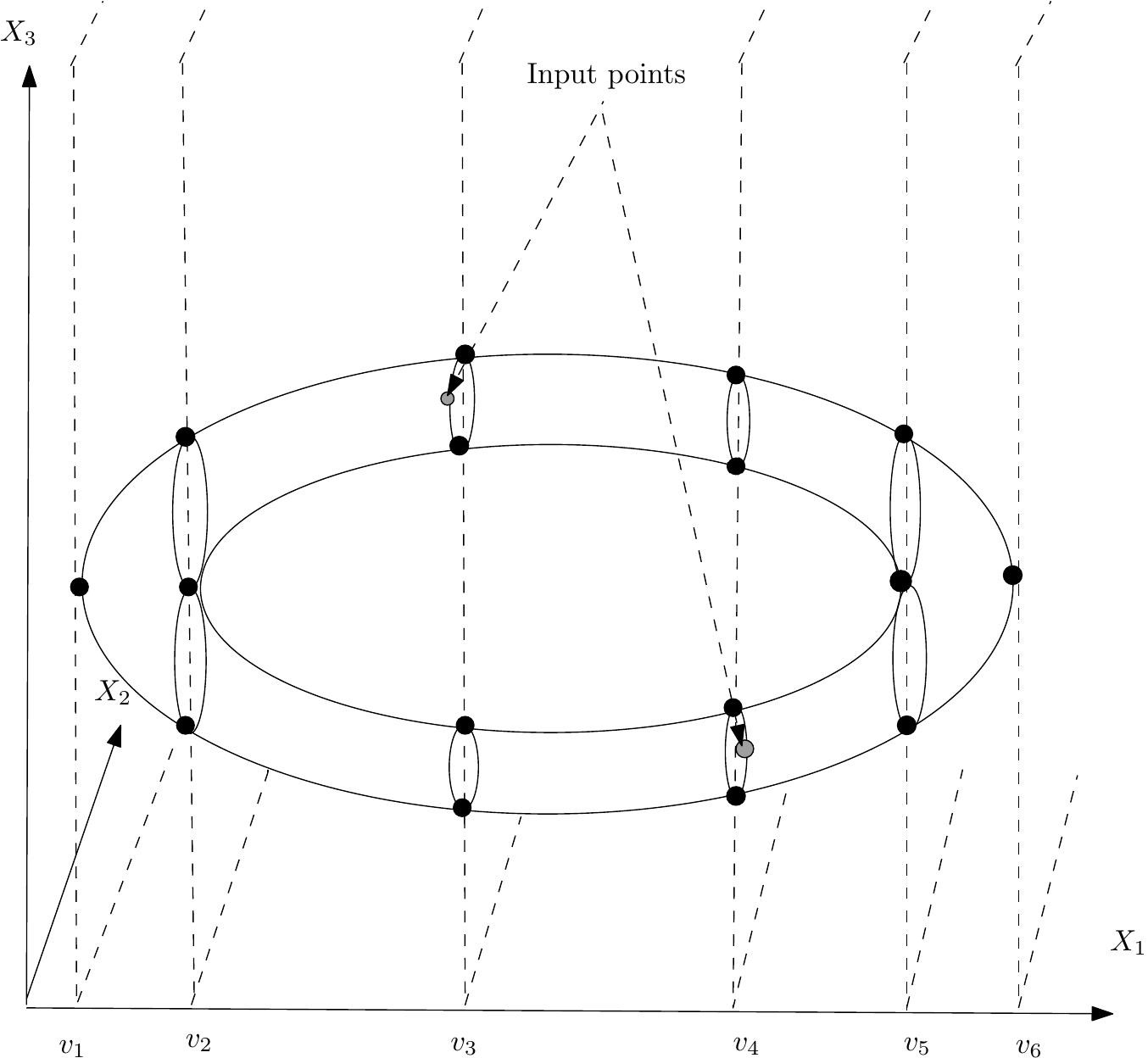}
\caption{Roadmap of the torus in $\R^3$.}
\label{fig:torus2}
\end{figure}
\end{center}
A key ingredient of the roadmap is  the construction of 
pseudo-critical points and values defined above.
The construction of the roadmap of an algebraic set
containing a finite number 
of input points ${\mathcal N}$ of this algebraic set
is as follows.
We first construct $X_2$-pseudo-critical points on $\ZZ(Q,\R^k)$
in a parametric way along  the $X_1$-axis by following continuously, 
as $x$ varies on the $X_1$-axis,
the $X_2$-pseudo-critical points on $\ZZ(Q,\R^k)_{x}$. This results in curve segments
and their endpoints on $\ZZ(Q,\R^k).$ The curve segments
 are continuous semi-algebraic curves parametrized by open intervals on the 
$X_1$-axis and their endpoints are points of $\ZZ(Q,\R^k)$
above the corresponding endpoints of the open intervals.
Since these curves and their endpoints include
for every $x\in\R$ the $X_2-$pseudo-critical points of
$\ZZ(Q,\R^k)_{x}$, they meet every connected component of
$\ZZ(Q,\R^k)_{x}$.  Thus, the set of curve
segments and their endpoints already satisfy
${\rm RM}_2.$ However, it is clear that this set might not
be semi-algebraically connected in a semi-algebraically
connected component and so ${\rm RM}_1$ might not be satisfied.
We add additional curve segments to ensure connectedness by recursing 
in certain  distinguished hyperplanes defined by
$X_1=z$ for distinguished values $z$.

The set of {\em distinguished values}
is the union of the $X_1$-pseudo-critical values, the first
coordinates of the input points ${\mathcal N}$, and the
first coordinates of the endpoints of the curve
segments. A {\em distinguished hyperplane}
is an hyperplane defined by $X_1=v$, where
$v$ is a distinguished value. The input points, the endpoints of
the curve segments, and the intersections of the curve
segments with the distinguished hyperplanes define 
the set of {\em distinguished points}.

Let
the distinguished values be
$v_1<\ldots <v_\ell.$
Note that amongst these  are the $X_1$-pseudo-critical values. Above each
interval $(v_i, v_{i+1})$ we have  constructed a collection
of curve segments ${\mathcal C}_i$ meeting every
semi-algebraically connected component of
$\ZZ(Q,\R^k)_v$ for every $v \in (v_i, v_{i+1})$. Above
each distinguished value $v_i$  we have 
a set of distinguished points ${\mathcal N}_i$. 
Each curve segment in ${\mathcal C}_i$
has an endpoint in ${\mathcal N}_i$ and another in
${\mathcal N}_{i+1}$.  Moreover, the union of the ${\mathcal N}_i$
contains ${\mathcal N}$.

We then repeat this construction in each 
distinguished hyperplane $H_i$ defined by $X_1=v_i$
with input $Q(v_i,X_2,\ldots,X_k)$ and the
distinguished points in ${\mathcal N}_i$. 
Thus, we construct distinguished values
$v_{i,1},\ldots, v_{i,\ell(i)}$ of 
$\ZZ(Q(v_i,X_2,\ldots,X_k),\R^{k-1})$ 
(with the role of $X_1$ being now played by $X_2$) and 
the process is iterated until
for 
$I=(i_1,\ldots,i_{k-2}), 1 \leq i_1 \leq\ell,
\ldots, 1 \leq i_{k-2} \leq \ell(i_1,\ldots,i_{k-3}),$ we have
distinguished values
$v_{I,1}< \ldots < v_{I, \ell(I)}$ along the
$X_{k-1}$ axis with corresponding sets of curve
segments and sets of distinguished points with the
required incidences between them.

The following theorem  is proved in \cite{BPR5}
(see also \cite{BPRbook06}).
\begin{proposition}
\label{15:prop:rm}
The semi-algebraic set
$\RM(\ZZ(Q,\R^k),{\mathcal N})$ obtained by this construction is a
roadmap for $\ZZ(Q,\R^k)$ containing ${\mathcal N}$.
\end{proposition}

Note that if $x \in \ZZ(Q,\R^k)$,  $\RM(\ZZ(Q,\R^k),\{x\})$ 
contains a path,
$\gamma(x)$,  connecting a distinguished point $p$ 
of   $\RM(\ZZ(Q,\R^k))$ to $x$.

\subsubsection{The Divergence Property of Connecting Paths}
\label{subsec:connectingpaths}
In applications to algorithms for computing Betti numbers of semi-algebraic
sets it becomes important to
examine the properties of parametrized paths which are the
unions of connecting paths starting 
at a given $p$ and ending 
at $x$,
where $x$ varies over a certain semi-algebraic subset of   
$\ZZ(Q,\R^k)$.  

We first note that for any $x = (x_1,\ldots,x_k) \in \ZZ(Q,\R^k)$
we have by construction that  $\RM(\ZZ(Q,\R^k))$ is contained in
$\RM(\ZZ(Q,\R^k),\{x\})$. In fact,
$$
\displaylines{
\RM(\ZZ(Q,\R^k),\{x\}) = 
\RM(\ZZ(Q,\R^k)) \cup \RM(\ZZ(Q,\R^k)_{x_1}, {\mathcal M}_{x_1}),
}
$$
where ${\mathcal M}_{x_1}$  
consists of 
$(x_2,\ldots,x_k)$
and the finite set of points obtained by intersecting
the curves in $\RM(\ZZ(Q,\R^k))$ parametrized by the $X_1$-coordinate
with the hyperplane $\pi_{1}^{-1}(x_1)$.

\begin{center}
\begin{figure}
\includegraphics[height=7cm]{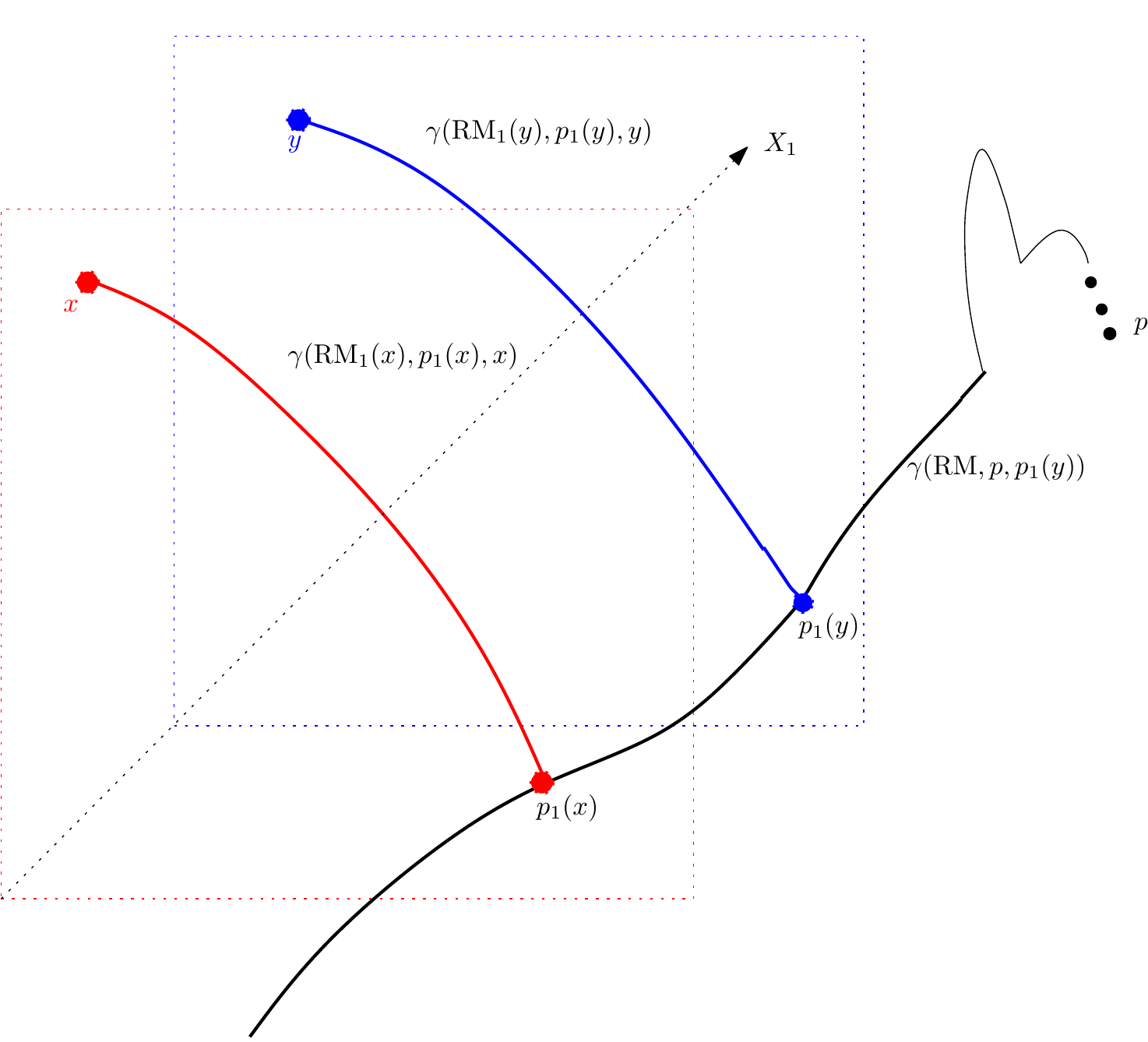}
\caption{The connecting path $\Gamma(x)$}
\label{fig:figure1}
\end{figure}
\end{center}

A connecting path $\gamma(x)$ 
(with non-self intersecting image) 
joining a distinguished point $p$ of $\RM(\ZZ(Q,\R^k))$ 
to $x$ can be extracted from $\RM(\ZZ(Q,\R^k),\{x\})$.
The connecting path $\gamma(x)$ consists of two consecutive parts,  
$\gamma_0(x)$ and $\Gamma_1(x)$.
The  path  $\gamma_0(x)$ is contained in $\RM(\ZZ(Q,\R^k))$
and the path $\Gamma_1(x)$  is contained in $\ZZ(Q,\R^k)_{x_1}$.
The part $\gamma_0(x)$ consists of a sequence of sub-paths
$\gamma_{0,0},\ldots,\gamma_{0,m}$. Each
$\gamma_{0,i}$ is a semi-algebraic path parametrized by 
one of the co-ordinates $X_1,\ldots,X_k$, over some interval
$[a_{0,i},b_{0,i}]$
with $\gamma_{0,0}(a_{0,0}) = p$.
The semi-algebraic maps
$\gamma_{0,0},\ldots,\gamma_{0,m}$ and the end-points
of their intervals of definition
$a_{0,0},b_{0,0},\ldots,a_{0,m},b_{0,m}$ are all independent of $x$
(up to the discrete choice of the path 
$\gamma(x)$ in $\RM(\ZZ(Q,\R^k),\{x\})$),
except $b_{0,m}$ which depends on $x_1$. 

Moreover,  $\Gamma_1(x)$ can again be decomposed into two parts 
$\gamma_1(x)$ and $\Gamma_2(x)$
with 
$\Gamma_2(x)$ contained in 
$\ZZ(Q,\R^k)_{(x_1,x_2)}$ 
and so on.

If $y = (y_1,\ldots,y_k) \in \ZZ(Q,\R^k)$ is another point such that
$x_1 \neq y_1$, then since
$\ZZ(Q,\R^k)_{x_1}$
and 
$\ZZ(Q,\R^k)_{y_1}$
are disjoint, 
it is clear that
$$
\RM(\ZZ(Q,\R^k),\{x\}) \cap \RM(\ZZ(Q,\R^k),\{y\}) = \RM(\ZZ(Q,\R^k)).
$$
Now consider a
connecting path  $\gamma(y)$ extracted from $\RM(\ZZ(Q,\R^k),\{y\})$. 
The images of $\Gamma_1(x)$ and $\Gamma_1(y)$ are
disjoint.
If the image of 
$\gamma_0(y)$ (which is contained in $\RM(\ZZ(Q,\R^k)$) follows the same
sequence of curve segments as $\gamma_0(x)$   
starting at $p$
(i.e. it consists of the same curves segments 
$\gamma_{0,0},\ldots,\gamma_{0,m}$ as in $\gamma_0(x)$),
then it is clear that the images of the 
paths $\gamma(x)$ and $\gamma(y)$ has the
property that they 
are identical up to a point and they are disjoint after it.
This is called the {\em divergence property} in \cite{BPR9}.

\subsubsection{Roadmaps of General Semi-algebraic Sets}
\label{subsecsec:generalroadmap}
Using the same ideas as above and some additional techniques for
controlling the combinatorial complexity of the algorithm 
it is possible to extend 
the roadmap algorithm to the case of semi-algebraic sets. The following
theorem appears in \cite{BPR5,BPRbook06} and gives the most 
efficient algorithm for constructing roadmaps.

\begin{theorem}\cite{BPR5,BPRbook06}
\label{16:the:saconnecting}
Let $Q \in
{\R}[X_1,\ldots,X_k]$
 with $\ZZ(Q,\R^k)$  of dimension
$k'$ and let ${\mathcal
P}\subset {\R}[X_1,\ldots,X_k]$
be a set of at most $s$
polynomials
for
which the degrees of the polynomials
 in ${\mathcal P}$ and $Q$ are bounded by
$d.$
Let $S$ be a ${\mathcal P}$-semi-algebraic subset of $\ZZ(Q,\R^k)$.
There is an algorithm which computes a roadmap $\RM(S)$ for $S$
with complexity $s^{k'+1} d^{O(k^2)}$
 in the ring ${\D}$
generated by the
coefficients of $Q$ and the elements
of ${\mathcal P}$.
If $\D=\Z,$ and the
bit-sizes
of the coefficients of the polynomials are bounded by
 $\tau$,
then the bit-sizes of the integers appearing in the
intermediate
computations and the output are bounded
by $\tau  d^{O(k^2)}$.
\end{theorem}

Theorem \ref{16:the:saconnecting} immediately implies that
there is an algorithm whose output is exactly one point
in every semi-algebraically connected component of
$S$ and  whose complexity in the ring generated by the 
coefficients of $Q$ and ${\mathcal P}$ is bounded by
$s^{k'+1} d^{O(k^2)}$.
In particular, this algorithm counts
the number semi-algebraically connected component of
$S$ within the same time bound.

\section{Topological Preliminaries}
\label{sec:topbackground}
The purpose of this section is to provide a self-contained introduction
to the basic mathematical machinery needed later. Some of the topics
would be familiar to most readers while a few others perhaps less so.
The sophisticated reader can choose to skip this whole section and
proceed directly to the descriptions of the various algorithms in the
later sections. 

We give a brief review of the concepts
from algebraic topology that play a role in the results surveyed 
in this paper. 
These include the definition of complexes of vector spaces 
(Section \ref{subsubsec:complex}),
definition of cohomology groups of semi-algebraic sets (Section
\ref{subsec:sahomology}), 
properties of the Euler-Poincar\'e characteristic of semi-algebraic sets
\ref{subsubsec:EP}), 
the nerve complex of covers (Section \ref{subsec:nerve}), 
a generalization of the nerve complex (Section \ref{subsec:nonleray}),
the Mayer-Vietoris double complex and its associated spectral
sequence (Section \ref{subsec:MV}),
the descent spectral sequence (Section \ref{subsec:descent}), and the
properties of homotopy colimits (Section \ref{subsec:hocolimit}).

\subsection{Homology and Cohomology groups}
\label{subsec:intuition}
Before we get to the  precise definitions of these groups it is good to 
have some intuition about them. As noted before  closed and
bounded semi-algebraic sets are finitely triangulable. This means that 
each closed and bounded semi-algebraic set $S \subset \R^k$ is
homeomorphic (in fact,  by a semi-algebraic map) to the polyhedron 
$|K|$ associated to a finite simplicial complex $K$. 
In fact $K$ can be chosen such that 
$|K| \subset \R^k$, and there is an effective algorithm
(see Theorem \ref{the:triangulation}) 
for computing $K$ given $S$. 
The simplicial cohomology (resp. homology groups) of $S$ are 
defined in terms of $K$ and are well-defined
(i.e they are independent of the chosen triangulation
which is of course very far from being unique). 

Roughly speaking 
the simplicial homology groups of a finite simplicial complex $K$
with coefficients in a field $\F$  
(which we assume to be $\Q$ in this survey)
are finite dimensional  $\F$-vector spaces and measure the 
{\em connectivity} of $|K|$ in various dimensions. For example, the zero-th
simplicial homology group, $\HH_0(K)$, has a generator corresponding to each 
connected component of $K$ and  its dimension gives  
the number of connected components of $|K|$. 
Similarly the first 
simplicial homology group, $\HH_1(K)$, is generated by the
``one-dimensional holes''
of $|K|$, and its dimension is the number of ``independent'' one-dimensional
holes of $|K|$. If $K$ is one-dimensional (that is a finite graph)
the dimension of $\HH_1(K)$ is the number of independent cycles in $K$.
Analogously, the $i$-th the 
simplicial homology group, $\HH_i(K)$, is generated by the 
``$i$-dimensional holes'' of $|K|$, and its dimension 
is the number of independent $i$-dimensional
holes of $|K|$. Intuitively an $i$-dimensional hole is an 
$i$-dimensional closed surface in $K$ (technically called a {\em cycle}) 
which does not {\em bound} any $(i+1)$-dimensional subset of $|K|$.

The simplicial {\em cohomology groups} of $K$ are dual (and isomorphic)
to the simplicial homology groups of $K$ as groups. 
However, in addition to the group structure
they also carry a multiplicative structure (the so called cup-product) 
which makes them a finer topological invariant than the homology groups. 
We are not going to use this multiplicative structure. Cohomology groups
also have nice but less geometric interpretations. 
Roughly speaking the cohomology groups of $K$ 
represent spaces of globally defined objects satisfying certain 
local conditions. 
For example,
the zero-th cohomology group, $\HH^0(K)$, can be interpreted as the
vector space of global functions on $|K|$ which are locally constant.
It is easy to see from this interpretation that 
the dimension of $\HH^0(K)$ is the number of connected components of $K$.
Similar geometric interpretations can be given for the higher cohomology 
groups, in terms of vector spaces of (globally defined) differential forms 
satisfying certain local condition (of being closed). In literature
this cohomology theory is referred to as {\em de Rham cohomology theory} 
and it is usually defined for smooth manifolds, but
it can also be defined for simplicial complexes 
(see for example \cite[Section 1.3.1]{Morita}).

\subsubsection{Homology vs Cohomology}
\label{subsubsec:homologyvscohomology}
It turns out that  the cohomological point of view gives
better intuition in designing algorithms described later in the paper.
This is our primary reason behind  preferring cohomology over homology.
Another reason for 
preferring the cohomology groups over the homology groups is that their 
interpretations continue to make sense in applications
outside of semi-algebraic geometry where
the notions of holes is meaningless (for instance, think of algebraic 
varieties defined over fields of positive characteristics) but the notion
of global functions (or for instance differential forms) 
continue to make sense.

\subsection{Definition of the Cohomology Groups of a Simplicial Complex}
\label{sec:precise}
We now give precise definitions of the cohomology groups of 
simplicial complexes. 

In order to do so we first need to introduce some
amount of algebraic machinery -- namely the concept of complexes of 
vector spaces and homomorphisms between them. 

\subsubsection{Complex of Vector Spaces}
\label{subsubsec:complex}
A complex of vector spaces is just a  sequence of vector spaces
and linear transformations
satisfying the property that the composition of two successive linear
transformations is  $0$.

More precisely
\begin{definition}[Complex of Vector Spaces]
\label{def:complex}
A sequence $\{\Ch^p\}$, $p \in {\mathbb Z}$, of
$\Q$-vector spaces together with a sequence
$\{\delta^p\}$ of
homomorphisms $\delta^p :\Ch^p \rightarrow \Ch^{p+1}$ 
(called differentials) for
which 
\begin{equation}
\label{eqn:defofcomplex}
\delta^{p+1}\circ \delta^p = 0 
\end{equation}
for all $p$ is called a complex. 
\end{definition}

The 
most important example for us
of a complex of vector spaces is the co-chain complex
of a simplicial complex $K$ denoted by $\Ch^{\bullet}(K)$.
It is defined as follows. 

\begin{definition}[Simplicial cochain complex]
\label{def:cochaincomplex}
For each $p \geq 0$, 
$\Ch^p(K)$ is a linear functional on the $\Q$-vector-space generated
by the $p$-simplices of $K$. Given $\phi \in \Ch^p(K)$, $\delta^p(\phi)$
is specified by its values on the $(p+1)$-dimensional simplices of $K$.
Given a $(p+1)$-dimensional simplex  $\sigma = [a_0,\ldots,a_{p+1}]$ of $K$
\begin{equation}
(\delta^p \phi)([a_0,\ldots,a_{p+1}]) = \sum_{i=0}^{p+1}
(-1)^i \phi([a_0,\ldots,\hat{a_i},\ldots,a_{p+1}]),
\end{equation}
where $\hat{}$ denotes omission.
\end{definition}
Notice that each $[a_0,\ldots,\hat{a_i},\ldots,a_{p+1}]$ is a $p$-dimensional
simplex of $K$ and since $\phi\in \Ch^p(K)$, 
$\phi([a_0,\ldots,\hat{a_i},\ldots,a_{p+1}]) \in \Q $ is well-defined.
It is an exercise now to check that the
homomorphisms $\delta^p :\Ch^p(K)\rightarrow \Ch^{p+1}(K)$ indeed satisfy
Eqn. \ref{eqn:defofcomplex} in the definition of a complex.

Now let $K$ be a simplicial complex and $L \subset K$ a sub-complex
of $K$ -- we will denote such a pair simply by $(K,L)$. 
Then for each $p \geq 0$ we have that
$\Ch^p(L) \subset \Ch^p(K)$ and we denote by
$\Ch^p(K,L)$ the quotient space $\Ch^p(K)/\Ch^p(L)$. It is now an easy
exercise to verify that the differentials $\delta^p$ in the complex
$\Ch^p(K)$ descend to $\Ch^p(K,L)$ and we define 

\begin{definition}[Simplicial cochain complex of a pair]
\label{def:cochaincomplexofapair}
The simplicial cochain complex of the pair $(K,L)$ to be the complex
$\Ch^\bullet(K,L)$ whose terms, $\Ch^p(K,L)$, and differentials, $\delta^p$,
are defined as above.
\end{definition}

Often, particularly in the context of algorithmic applications
it is more economical to use cellular complexes instead of 
simplicial complexes. We recall here the definition of a finite
regular cell complex referring the reader to standard sources
in algebraic topology for more in-depth study of cellular theory
(see \cite[pp. 81]{Whitehead}).

\begin{definition}[Regular cell complex]
\label{def:cellcomplex}
An $\ell$-dimensional {\em cell} in $\R^k$ is a subset of 
$\R^k$ homeomorphic to  $\overline{B_\ell(0,1)}$.
A regular cell
complex $\Sigma$ in $\R^k$ is a finite collection of cells
satisfying the following properties:
\begin{enumerate}
\item If $c_1,c_2 \in \Sigma$, then either 
$c_1 \cap c_2 = \emptyset$ or $c_1 \subset \partial c_2$ or
$c_2 \subset \partial c_1$.
\item
The boundary of each cell of $\Sigma$ is a union of cells
of $\Sigma$.
\end{enumerate}
We denote by $|\Sigma|$ the set 
$\displaystyle{
\bigcup_{c \in \Sigma} c
}.
$
\end{definition}

\begin{remark}
Notice that every simplicial complex $K$ may be considered as a regular
cell complex whose cells are the closures of the  simplices
of $K$.
\end{remark}

As in the case of simplicial complexes it is possible
to associate a complex, $\Ch^{\bullet}(\Sigma)$ (the co-chain complex of $K$), 
to each regular cell complex $K$  which is defined in an analogous manner. 
In order to avoid technicalities  
we omit the precise definition of this complex
referring the interested reader to \cite[pp. 82]{Whitehead} instead. 
We remark that the dimension of $\Ch^p(\Sigma)$ is equal to the 
number of $p$-dimensional cells in $\Sigma$ and the matrix entries for the
differentials in the complex with respect to the standard basis comes
from $\{0,1,-1\}$ just as in the case of simplicial co-chain complexes.

The advantage of using cell complexes instead of simplicial complexes
can be seen in the following example.
\begin{example}
Consider the unit sphere $\Sphere^k \subset \R^{k+1}$. For 
$0 \leq j \leq k+1$ and $\eps \in \{+,-\}$ let
\begin{equation}
c_j^{\eps} = \{ x \in \Sphere^k \;\mid\; X_0 = \cdots = X_{j-1} = 0,
\eps X_j \geq 0\}.
\end{equation}

Then it is easy to check that each $c_j^{\eps}$ is a $k-j$ dimensional
cell and the collection,
$\displaystyle{
\Sigma_k = \{c_j^{\eps} \;\mid\; 0 \leq j \leq k, \eps \in \{-,+\}\}
}
$ is a regular cell complex with $|\Sigma_k| = \Sphere^k$
(see Figure \ref{fig:exampleofsphere} for the case $k=2$).

\hide{
\begin{center}
\begin{figure}
\includegraphics[height=5cm]{cellcomplexofsphere}
\caption{Cell decomposition of $\Sphere^2$}
\label{fig:exampleofsphere}
\end{figure}
\end{center}
}

       \begin{figure}[hbt]
         \centerline{
           \scalebox{0.5}{
 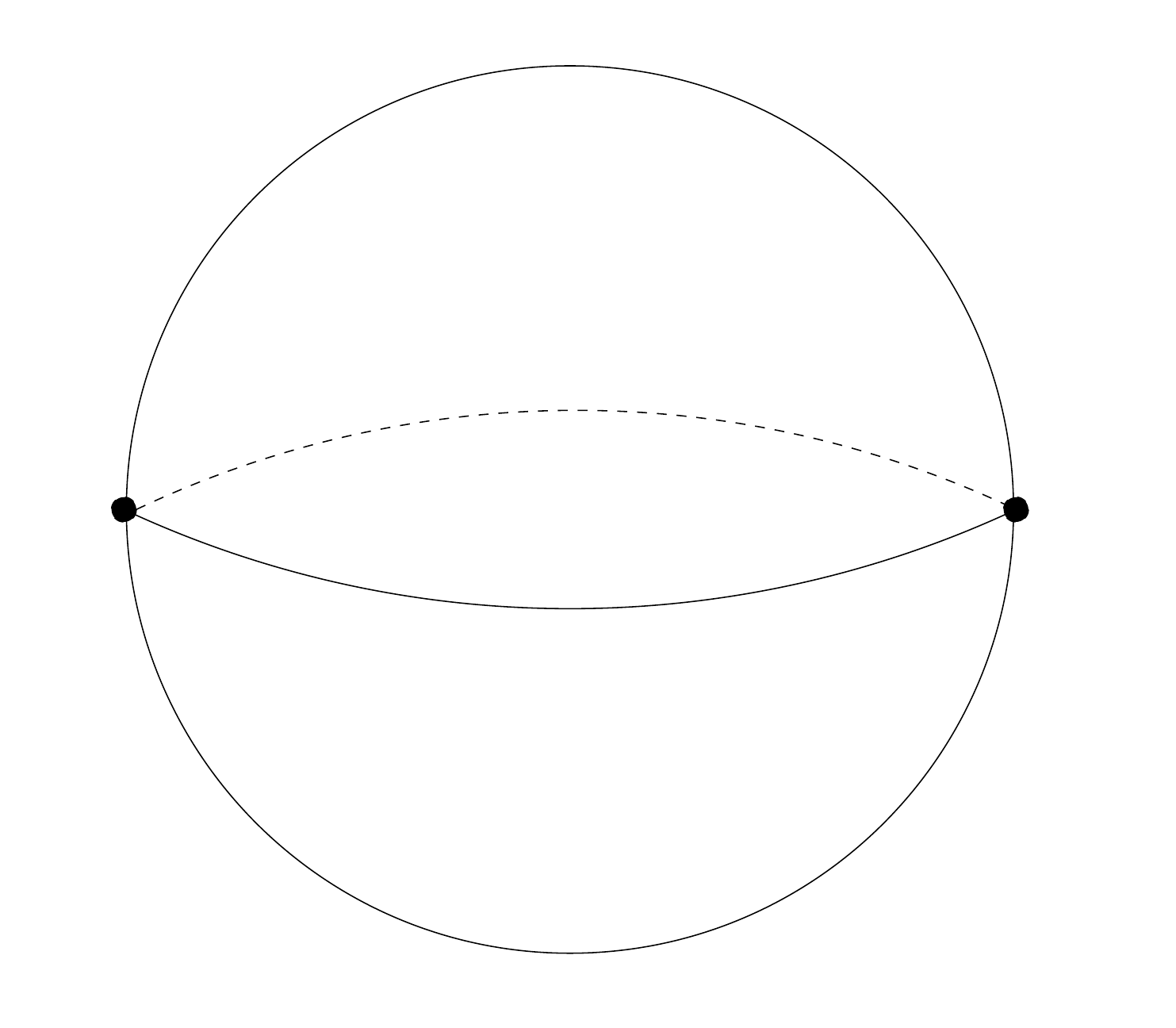
             }
           }
\caption{Cell decomposition of $\Sphere^2$}
\label{fig:exampleofsphere}
       \end{figure}

Notice that $\#\Sigma_k = 2k$.
On the other hand if we consider the sphere
as homeomorphic to the boundary of a standard $(k+1)$-dimensional simplex,
then the corresponding simplicial complex 
will contain $(2^{k+2}-2)$ simplices (which is exponentially large in $k$).
\end{example}

We now associate to each complex, $\Ch^{\bullet}$, 
a sequence of vector spaces, $\HH^p(\Ch^\bullet)$, called the cohomology
groups of $\Ch^{\bullet}$. Note that it follows from Eqn. 
\ref{eqn:defofcomplex} that for a complex $\Ch^{\bullet}$ with differentials
$\delta^p :\Ch^p\rightarrow \Ch^{p+1}$ 
the subspace $B^p(\Ch^\bullet)= {\rm Im}(\delta^{p-1})\subset \Ch^p$ 
is contained in the subspace $Z^p(\Ch^\bullet) = \Ker(\delta^p) \subset 
\Ch^p$. The subspaces $B^p(\Ch^\bullet)$ (resp. $Z^p(\Ch^\bullet)$) are 
usually referred to as the co-boundaries (resp. co-cycles) of the complex
$\Ch^{\bullet}$.  Moreover,

\begin{definition}[Cohomology groups of a complex]
\label{def:cohomologyofcomplex}
The cohomology groups, $\HH^p(\Ch^\bullet)$, are defined by
\begin{equation}
\label{eqn:defofcohomology}
\HH^p(\Ch^\bullet) = {Z^p(\Ch^{\bullet})}/{B^p(\Ch^{\bullet})}.
\end{equation}
We will denote by
$\HH^*(\Ch^{\bullet})$ the graded vector space
$\bigoplus_{p} \HH^p(\Ch^{\bullet})$.
\end{definition}

Note that the cohomology groups, $\HH^p(\Ch^\bullet)$,  are all
$\Q$-vector spaces
(finite dimensional if the vector spaces $\Ch^p$'s are themselves finite
dimensional).

\begin{definition}[Exact sequence]
A complex $\Ch^{\bullet}$ is called {\em acyclic} 
and the corresponding sequence of vector space homomorphisms
is called an {\em exact sequence} if $\HH^*(\Ch^{\bullet}) = 0$.
\end{definition}

Applying Definition \ref{def:cohomologyofcomplex} to the particular case
of the co-chain complex of a simplicial complex $K$
(cf. Definition \ref{def:cochaincomplex}) we obtain

\subsubsection{Cohomology of a Simplicial Complex}
\begin{definition}[Cohomology of a simplicial complex]
\label{def:cohomologygroupsofsimplicialcomplex}
The cohomology groups of a simplicial complex $K$ are by definition
the cohomology groups, $\HH^p(\Ch^{\bullet}(K))$, of its co-chain complex.
\end{definition}

Similarly, given a pair of simplicial complexes $(K,L)$, we define
\begin{definition}[Cohomology of a pair]
\label{def:cohomologygroupsofapairofsimplicialcomplexes}
The cohomology groups of the pair $(K,L)$ are by definition
the cohomology groups, $\HH^p(\Ch^{\bullet}(K,L))$, of its co-chain complex.
\end{definition}

\begin{example}
\label{ex:simplex}
Let $\Delta_n$ be the simplicial complex corresponding to an $n$-simplex.
In other words the simplices of $\Delta_n$ consist of 
$[i_0,\ldots,i_\ell], 0 \leq i_0 < \cdots < i_\ell \leq n$.
The polyhedron $|\Delta_n|$ is just the $n$-dimensional simplex.  Then using
Definition \ref{def:cohomologygroupsofsimplicialcomplex}
one can verify that
\begin{align*}
\HH^i(\Delta_n) &= \Q, \; i = 0, \\ 
\HH^i(\Delta_n) &= 0,  \; i > 0.
\end{align*}
\end{example}

\begin{example}
\label{ex:sphere}
Let $\partial \Delta_n$ 
be the simplicial complex corresponding to the boundary of
the  $n$-simplex.
In other words the simplices of $\partial\Delta_n$ consist of 
$[i_0,\ldots,i_\ell], 0 \leq i_0 < \cdots < i_\ell \leq n, \ell < n$.
Then again by a
direct application of Definition \ref{def:cohomologygroupsofsimplicialcomplex}
one can verify that
\begin{align*}
\HH^i(\partial \Delta_n) &= \Q, \; i = 0,n-1 \\ 
\HH^i(\partial \Delta_n) &= 0,  \; \mbox{ else}.
\end{align*}
\end{example}

The above examples serve to confirm our geometric 
intuition behind the homology groups of the spaces 
$|\Delta_n|$ and
$|\partial \Delta_n|$ 
explained in Section \ref{subsec:intuition} above -- 
namely that they are both connected
and $|\Delta_n|$ has no holes in dimension $> 0$, and $|\partial \Delta_n|$ 
has a single $(n-1)$-dimensional hole. 

\begin{example}
\label{ex:pair}
It is also an useful exercise to verify that
\begin{align*}
\HH^i(\Delta_n,\partial \Delta_n) &= \Q, \; i = 0,n \\ 
\HH^i(\Delta_n,\partial \Delta_n) &= 0,  \; \mbox{ else}.
\end{align*}
\end{example}
\begin{remark}
Example \ref{ex:pair}
illustrates  that for ``nice spaces'' 
of the kind we consider in this paper
(such as regular cell complexes)
the cohomology groups of a pair $(K,L)$ are isomorphic to the 
cohomology groups of the quotient space $|K|/|L|$. 
For instance, the above example illustrates the fact that the 
topological quotient of an $n$-dimensional ball by its boundary 
is the $n$-dimensional sphere.
\end{remark}

\subsubsection{Homomorphisms of Complexes}
\label{subsec:homomorphism}
We will also need the notion of homomorphisms of complexes which generalizes
the notion of ordinary vector space homomorphisms. 
\begin{definition}[Homomorphisms of complexes]
\label{def:complexhomomorphsim}
Given two  complexes, $\Ch^\bullet = (\Ch^p,\delta^p)$ and $\D^{\bullet}=
(\D^{p},\delta^p)$,
{\em a homomorphism of complexes},
$\phi^{\bullet}: \Ch^{\bullet} \rightarrow \D^{\bullet}$, is a
sequence of homomorphisms $\phi^p: \Ch^p \rightarrow \D^p$ for which
$\delta^p\circ \phi^p = \phi^{p+1}\circ\delta^p$ for all $p$.

In other words the following diagram is commutative.

\begin{equation}
\label{eqn:doublecomplexcommutative}
\begin{diagram}
\node{\cdots}\arrow{e}\node{\Ch^p}\arrow{e,t}{\delta^p}\arrow{s,l}{\phi^p}
\node{\Ch^{p+1}}\arrow{e}\arrow{s,l}{\phi^{p+1}}\node{\cdots} \\
\node{\cdots}\arrow{e}\node{\D^p}\arrow{e,t}{\delta^p} 
\node{\D^{p+1}}\arrow{e}\node{\cdots}
\end{diagram}
\end{equation}
\end{definition}

A homomorphism of complexes
$\phi^{\bullet}: \Ch^{\bullet} \rightarrow \D^{\bullet}$
induces homomorphisms
$\phi^i: \HH^i(\Ch^{\bullet}) \rightarrow \HH^i(\D^{\bullet})$ and 
we will
denote the corresponding homomorphism between the  graded
vector spaces $\HH^*(\Ch^{\bullet}), \HH^*(\D^{\bullet})$ by
$\phi^*$. 

\begin{definition}[Quasi-isomorphism]
\label{def:qis}
The homomorphism $\phi^{\bullet}$ is called a {\em quasi-isomorphism} if the 
homomorphism $\phi^*$ is an isomorphism.
\end{definition}

Having introduced the algebraic machinery of complexes of vector spaces,
we now define the cohomology groups of semi-algebraic sets in terms
of their triangulations and their associated simplicial complexes.

\subsection{Cohomology Groups of Semi-algebraic Sets}
\label{subsec:sahomology}
A closed and bounded semi-algebraic
set $S \subset \R^k$ is semi-algebraically triangulable
(see Theorem \ref{the:triangulation} above). 
\begin{definition}[Cohomology groups of closed and bounded semi-algebraic sets]
\label{def:sacohomology}
Given a triangulation,
$h: |K| \rightarrow S$, where $K$ is a simplicial complex,
we define the $i$-th simplicial cohomology group of $S$,
by $\HH^i(S) = \HH^i(\Ch^{\bullet}(K))$,
where $\Ch^{\bullet}(K)$ is the co-chain complex of $K$.
The groups  $\HH^i(S)$ are invariant under
semi-algebraic homeomorphisms (and they coincide with the corresponding
singular cohomology groups when $\R = \re$). We denote by $b_i(S)$ the
$i$-th Betti number of $S$ (i.e. the dimension  of $\HH^i(S)$ as a vector
space).
\end{definition}

\begin{remark}
\label{rem:conicstructureatinfinity}
For a closed but not necessarily bounded semi-algebraic set $S \subset \R^k$
we will denote by $\HH^i(S)$  the $i$-th simplicial cohomology group of 
$S \cap \overline{B_k(0,r)}$ for sufficiently large $r > 0$.
The sets $S \cap \overline{B_k(0,r)}$ are semi-algebraically homeomorphic
for all sufficiently large $r> 0$
and hence this definition makes sense.
(The last property is usually referred to as
{\em the local conic structure at infinity} of semi-algebraic sets
\cite[Theorem 5.48]{BPRbook06}).
The definition of cohomology groups of arbitrary semi-algebraic sets in
$\R^k$ requires some care and several possibilities exist and 
we refer the reader to \cite[Section 6.3]{BPRbook06} 
where one such definition is given
which agrees with singular cohomology in case $\R = \re$.
\end{remark}

\subsubsection{The Euler-Poincar\'e Characteristic: Definition and Basic
Properties}
\label{subsubsec:EP}
An useful topological invariant of semi-algebraic sets which is often easier
to compute than their Betti numbers is the 
{\em Euler-Poincar\'e characteristic}.
\begin{definition}[Euler-Poincar\'e characteristic of a closed and
bounded semi-algebraic set]
\label{def:EP}
Let $S \subset \R^k$, be a closed and bounded semi-algebraic set. Then
the Euler-Poincar\'e characteristic of $S$ is defined by
\begin{equation}
\label{eqn:defofEP}
\chi(S) = \sum_{i \geq 0} (-1)^i b_i(S).
\end{equation}
\end{definition}

From the point of view of designing algorithms, it is useful to define 
Euler-Poincar\'e characteristic also for 
{\em locally closed} semi-algebraic sets.
A semi-algebraic set is locally closed if it is
the intersection of a closed semi-algebraic set with an open one. 
A standard example of a locally closed semi-algebraic set is the
realization, $\RR(\sigma)$, of a sign-condition $\sigma$ on a 
family of polynomials.

We now define 
Euler-Poincar\'e characteristic for locally closed semi-algebraic sets
in terms of the Borel-Moore cohomology groups of such sets (defined
below).
This definition agrees with the definition of Euler-Poincar\'e
characteristic stated above  for closed and bounded semi-algebraic sets.
They may be distinct for semi-algebraic sets which are closed but not 
bounded.

\begin{definition}
\label{def:cohomologyofpairs}
The simplicial cohomology groups  of a pair of
closed and bounded semi-algebraic sets $T \subset S \subset \R^k$ are
defined as follows.
Such a pair of closed and  bounded semi-algebraic sets can be
triangulated (cf. Theorem \ref{the:triangulation})
using  a pair of simplicial complexes $(K,A)$
where $A$ is a sub-complex of $K$.
The  $p$-th simplicial cohomology group
of the pair
$(S,T)$, $\HH^p(S,T)$, is by definition to be $\HH^p(K,A)$.
The dimension of $\HH^p(S,T)$ as a $\Q$-vector space is called the
$p$-th Betti number of the pair
$(S,T)$ and denoted $b_p(S,T)$.
The  Euler-Poincar\'e characteristic
of the pair $(S,T)$ is
$$\chi(S,T) = \sum_i (-1)^i b_i(S,T).$$
\end{definition}

\begin{definition}[Borel-Moore cohomology group]
\label{def:defofBMhomology}
The $p$-th Borel-Moore cohomology group of $S\subset \R^k$,
denoted $\HH^p_{BM}(S)$, 
is defined in terms of the cohomology groups of a pair of 
closed and bounded semi-algebraic sets as follows.
For any $r > 0$ let $S_r = S \cap B_k(0,r)$. 
Note that for a locally closed semi-algebraic
set $S$ both $\overline{S_r}$ and $\overline{S_r}\setminus S_r$ are closed
and bounded, and hence $\HH^p(\overline{S_r}, \overline{S_r}\setminus S_r)$ is
well defined.
Moreover, it is a consequence of 
the local conic structure at infinity of semi-algebraic sets
(see Remark \ref{rem:conicstructureatinfinity} above)
that the cohomology group $H^p(\overline{S_r}, \overline{S_r}\setminus S_r)$
is invariant for all sufficiently  large $r > 0$.
We define
$\HH^p_{BM}(S) = \HH^p(\overline{S_r}, \overline{S_r}\setminus S_r)$ for
$r >0$ sufficiently large and it follows from the above remark that it
is well defined. 
\end{definition}

The Borel-Moore cohomology groups are  invariant under semi-algebraic
homeomorphisms (see \cite{BCR}.
It also follows clearly from the definition that for a  closed and bounded
semi-algebraic set the Borel-Moore cohomology groups coincide with the
simplicial cohomology groups.

\begin{definition}[Borel-Moore Euler-Poincar\'e characteristic]
\label{def:defofBMEP}
For a locally closed semi-algebraic set $S$ we define the Borel-Moore
Euler-Poincar\'e characteristic by
\begin{equation}
\label{eqn:BMEP}
\chi^{BM}(S) = \sum_{i=0}^k (-1)^i~b_i^{BM}(S)
\end{equation}
where $b_i^{BM}(S)$ denotes the dimension of $\HH^i_{BM}(S).$
\end{definition}

Since the Borel-Moore Euler-Poincar\'e characteristic might not be very
familiar, the reader is encouraged to compute it in a few simple examples.
In particular, one should check that for the half-open interval
$[0,1)$ which is locally closed we have
\begin{equation}
\label{eqn:bmepforhalfopeninterval}
\chi^{BM}([0,1)) = 0.
\end{equation}
We also have
\begin{equation}
\label{eqn:bmepforopeninterval}
\chi^{BM}((0,1)) = -1, 
\end{equation}
and more generally,
\begin{equation}
\label{eqn:bmepforball}
\chi^{BM}(B_k(0,1)) = (-1)^{k}.
\end{equation}

If $S$ is closed and bounded then $\chi^{BM}(S) = \chi(S)$.

The Borel-Moore Euler-Poincar\'e characteristic 
has the following additivity 
property (reminiscent of the similar property of volumes)
which makes them particularly useful
in algorithmic applications (see for example
Section \ref{subsec:epquadratic} below).

\begin{proposition}
\label{6:prop:additivity}
Let $X,X_1$ and $X_2$ be locally closed semi-algebraic
sets such that  $$X_1\cup X_2=X, X_1\cap X_2=\emptyset.$$  Then
\begin{equation}
\label{eqn:additivityofbmep}
\chi^{BM}(X)=\chi^{BM}(X_1)+\chi^{BM}(X_2).
\end{equation}
\end{proposition}

Since for closed and bounded semi-algebraic sets, the Borel-Moore 
Euler-Poincar\'e characteristic  agrees with the ordinary 
Euler-Poincar\'e characteristic, it is easy to derive the following additivity
property of the Euler-Poincar\'e characteristic of closed and
bounded sets.

\begin{proposition}
\label{6:prop:additivityordinary}
Let $X_1$ and $X_2$ be  closed and bounded semi-algebraic
sets. 
Then
\begin{equation}
\label{eqn:additivityofep}
\chi(X_1 \cup X_2)=\chi(X_1)+\chi(X_2) - \chi(X_1 \cap X_2).
\end{equation}
\end{proposition}

Note that
Proposition \ref{6:prop:additivityordinary} is an immediate consequence of
Proposition \ref{6:prop:additivity} once we notice that the sets
$Y_1 = X_1 \setminus (X_1 \cap X_2)$ and
$Y_2 = X_2 \setminus (X_1 \cap X_2)$ are locally closed,
the set $X_1 \cup X_2$  is the disjoint union of
the locally closed sets $Y_1,Y_2$ and $X_1 \cap X_2$,
and
\[
\chi^{BM}(Y_i) = \chi^{BM}(X_i) - \chi^{BM}(X_1 \cap X_2) = 
\chi(X_i) - \chi(X_1 \cap X_2), \mbox{ for } i = 1,2.
\]

More generally by applying 
Proposition \ref{6:prop:additivityordinary} inductively
we get the following inclusion-exclusion property of the
(ordinary) Euler-Poincar\'e characteristic.

For any $n \in \Z_{\geq 0}$ we denote by $[n]$ the set $\{1,\ldots,n\}$.

\begin{proposition}
\label{6:prop:epinclusionexclusion}
Let $X_1,\ldots,X_n$ be  closed and bounded semi-algebraic
sets. Then denoting by $X_I$ the semi-algebraic set 
$\displaystyle{\bigcap_{i \in I} X_i}$ for $I \subset [n]$, 
we have
\begin{equation}
\label{eqn:epinclusionexclusion}
\chi(\bigcup_{i \in [n]} X_i)=
\sum_{I \subset [n]} (-1)^{(\#I+1)}~\chi(X_I).
\end{equation}
\end{proposition}
 
\subsection{Homotopy Invariance}
\label{subsec:homotopy}
The cohomology groups of semi-algebraic sets as defined above
(Definition \ref{def:sacohomology}) 
are obviously invariant under semi-algebraic
homeomorphisms. But, in fact, they are invariant under a weaker
equivalence relation -- namely, semi-algebraic {\em 
homotopy equivalence} (defined below).
This property is crucial in the design of efficient algorithms 
for computing Betti numbers of semi-algebraic sets since it allows
us to replace a given set by one that is better behaved from the
algorithmic point of view but having the same homotopy type
as the original set. This technique is ubiquitous in 
algorithmic semi-algebraic geometry and we will see some version of
it in almost every algorithm described in the following sections
(cf.  Example \ref{ex:homotopy}).

\begin{remark}
\label{rem:sa}
The reason behind insisting on the prefix
{\em ``semi-algebraic''} with regard to homeomorphisms and 
homotopy equivalences here and in the rest of the paper, is that for
general real closed fields, the ordinary Euclidean topology 
could be rather strange. For example, the real closed field, 
$\re_{{\rm alg}}$, of  real algebraic
numbers is totally disconnected as a 
topological space under the Euclidean topology. On the other hand, if
the ground field $\R = \re$, then we can safely drop the prefix
``semi-algebraic'' in the statements made above. However, even if we
start with $\R = \re$, in many applications described below we  enlarge
the field by taking non-archimedean extensions of $\R$ 
(see Section \ref{subsec:Puiseux}), and the remarks
made above would again apply to these field extensions.
\end{remark}

\begin{definition}[Semi-algebraic homotopy]
\label{def:homotopy}
Let $X,Y$ be two closed and bounded semi-algebraic sets. Two
semi-algebraic continuous functions
$f,g: X \rightarrow Y$ are {\em semi-algebraically homotopic},
$f \sim_{sa} g$, if there is a continuous
semi-algebraic function $F : X \times [0,1] \rightarrow Y$ such that
$F(x,0) =
f(x)$ and $F(x,1) = g(x)$ for all $x \in X$.
\end{definition}

Clearly, semi-algebraic
homotopy is an
equivalence relation among semi-algebraic continuous maps from $X$ to $Y$.

\begin{definition}[Semi-algebraic homotopy equivalence]
\label{def:sahomotopyequivalence}
The sets  $X,Y$ are semi-algebraically homotopy equivalent
if there exist semi-algebraic continuous functions
$f: X \rightarrow Y$, $g: Y \rightarrow X$
such that $g \circ f \sim_{sa}{\rm Id}_X$,
$f \circ g \sim_{sa} {\rm Id}_Y$.
\end{definition}

We have
\begin{proposition}[Homotopy Invariance of the Cohomology Groups]
\label{6:prop:homotopicsa}
Let $X,Y$ be two closed and  bounded semi-algebraic sets
of $\R^k$ that
are semi-algebraically homotopy equivalent. Then, $\HH^*(X) \cong \HH^*(Y)$.
\end{proposition}

\subsection{The Leray Property and the Nerve Lemma}
\label{subsec:nerve}
It clear from the definition of the cohomology groups of closed and 
bounded semi-algebraic sets (Definition \ref{def:sacohomology} above)
the Betti numbers of such a set can be computed using elementary linear
algebra once we have a triangulation of the set. However, 
as we have seen before (cf. Theorem \ref{the:triangulation}), 
triangulations of semi-algebraic sets are expensive
to compute, requiring double exponential time.

One basic idea that underlies some of the recent progress
in designing algorithms for computing the
Betti numbers of semi-algebraic sets is that the cohomology
groups of a semi-algebraic set can often be computed from a sufficiently
well-behaved covering of the set {\em without having to triangulate the set}.

The idea of computing
cohomology from ``good'' covers is an old one in algebraic topology
and the first result in this direction is often called the ``Nerve Lemma''.
In this section we give a brief introduction to the Nerve Lemma and
its generalizations.

We first define formally the notion of a cover of a closed, bounded
semi-algebraic set.
\begin{definition}[Cover]
\label{def:covering}
Let $S \subset \R^k$ be a closed and bounded semi-algebraic set.
A cover, ${\mathcal C}(S)$, of $S$ consists  of an ordered
index set, which by a slight abuse of language we also 
denote by  ${\mathcal C}(S)$, and a map that
associates to each $\alpha \in {\mathcal C}(S)$
a closed and bounded semi-algebraic
subset $S_\alpha \subset S$ such that 
\[
S = \bigcup_{\alpha \in {\mathcal C}(S)} S_\alpha.
\]
\end{definition}

\begin{remark}
Even though the notation for a cover might seem unnecessarily heavy 
at the moment
it will prove useful later on the paper when we discuss non-Leray covers
(see Section \ref{subsec:nonleray} below).
\end{remark}

For $\alpha_0,\ldots,\alpha_p, \in {\mathcal C}(S)$,
we associate to the formal product,
$\alpha_0\cdots\alpha_p$, the closed and bounded semi-algebraic set 

\begin{equation}
\label{eqn:defofprod}
S_{\alpha_0\cdots\alpha_p} = S_{\alpha_0} \cap \cdots \cap S_{\alpha_p}.
\end{equation}

Recall that the $0$-th simplicial cohomology group of a closed and bounded 
semi-algebraic set $X$, $\HH^0(X)$, can be identified with the 
$\Q$-vector space
of $\Q$-valued locally constant functions on $X$. Clearly the dimension
of $\HH^0(X)$ is equal to the number of connected components of $X$.

For $\alpha_0,\alpha_1,\ldots,\alpha_p,\beta \in {\mathcal C}(S)$, 
and $\beta \not\in \{\alpha_0,\ldots,\alpha_p\}$,
let
\[
r_{\alpha_0,\ldots,\alpha_p;\beta}: 
\HH^0(S_{\alpha_0\cdots\alpha_p}) \longrightarrow 
\HH^0(S_{\alpha_0\cdots\alpha_{p}\cdot\beta})
\]
be the homomorphism defined as follows.
Given a locally constant function,
$\phi \in \HH^0(S_{\alpha_0\cdots\alpha_p})$,
$r_{\alpha_0\cdots\alpha_p;\beta}(\phi)$ 
is the locally constant function
on $S_{\alpha_0\cdots\alpha_{p}\cdot\beta}$ obtained by restricting
$\phi$ to $S_{\alpha_0\cdots\alpha_{p}\cdot\beta}$.

We define the generalized restriction homomorphisms
\[
\delta^p: \bigoplus_{\alpha_0 < \cdots < \alpha_p, \alpha_i \in {\mathcal C}(S)}  
\HH^0(S_{\alpha_0\cdots\alpha_p})\longrightarrow 
\bigoplus_{\alpha_0< \cdots <\alpha_{p+1}, \alpha_i \in {\mathcal C}(S)} 
\HH^0(S_{\alpha_0\cdots\alpha_{p+1}})
\]
by 
\begin{equation}
\label{eqn:generalized_restriction}
\delta^p(\phi)_{\alpha_0\cdots\alpha_{p+1}}= \sum_{0 \leq i \leq p+1} 
(-1)^{i} r_{\alpha_0\cdots\hat{\alpha_i}\cdots\alpha_{p+1}; \alpha_i}
(\phi_{\alpha_0\cdots\hat{\alpha_{i}}\cdots\alpha_{p+1}}),
\end{equation}
where $\phi \in \bigoplus_{\alpha_0< \cdots <\alpha_p \in {\mathcal C}(S)}  
\HH^0(S_{\alpha_0\cdots\alpha_p})$ and
$r_{\alpha_0\cdots\hat{\alpha_i}\cdots\alpha_{p+1}; \alpha_i}$ 
is the restriction homomorphism defined previously.
The sequence of homomorphisms $\delta^p$ gives rise to a complex,
$\LL^{\bullet}({\mathcal C}(S))$, defined by

\begin{equation}
\label{eqn:defofLL}
\LL^p({\mathcal C}(S)) = 
\bigoplus_{\alpha_0< \cdots <\alpha_p, \alpha_i \in {\mathcal C}(S)} 
\HH^0(S_{\alpha_0\cdots\alpha_p}),
\end{equation}

with the differentials 
$\delta^p: \LL^p({\mathcal C}(S)) \rightarrow \LL^{p+1}({\mathcal C}(S))$ 
defined as in Eqn. (\ref{eqn:generalized_restriction}).

\begin{definition}[Nerve complex]
\label{def:nervecomplex}
The complex $\LL^{\bullet}({\mathcal C}(S))$ is 
called the {\em nerve complex} of the cover  ${\mathcal C}(S)$. 
\end{definition}

For $\ell \geq 0$ we will denote by 
$\LL^{\bullet}_\ell({\mathcal C}(S))$ the truncated complex
defined by

\begin{align*}
\LL^p_\ell({\mathcal C}(S)) &= \LL^p({\mathcal C}(S)),\; 0 \leq p \leq \ell, \\
&= 0,\;  p > \ell.
\end{align*}

Notice that once we have a cover of $S$
and we identify the connected components of the various intersections,
$S_{\alpha_0\cdots\alpha_p}$, we have natural bases for the vector
spaces
$$
\LL^p({\mathcal C}(S)) = 
\bigoplus_{\alpha_0< \cdots <\alpha_p, \alpha_i \in {\mathcal C}(S)} 
\HH^0(S_{\alpha_0\cdots\alpha_p})
$$ 
appearing as terms of the nerve complex. Moreover, the matrices
corresponding to the homomorphisms $\delta^p$ in this basis depend only
on the inclusion relationships between the connected components of
$S_{\alpha_0\cdots\alpha_{p+1}}$ and those of 
$S_{\alpha_0\cdots\alpha_p}$. 

\begin{definition}[Leray Property]
\label{def:leray}
We say that the cover
${\mathcal C}(S)$
{\em satisfies the Leray property} if
each non-empty intersection
$S_{\alpha_0\cdots\alpha_p}$ is contractible.
\end{definition}

Clearly, in this case
$$
\begin{array}{cccc}
\HH^0(S_{\alpha_0\cdots\alpha_p}) &\cong & \Q, &\mbox{if $S_{\alpha_0\cdots\alpha_p} \neq \emptyset$} \\
                                  &\cong &  0,  & \mbox{if $S_{\alpha_0\cdots\alpha_p} = \emptyset$}.
\end{array}
$$

It is a classical fact (usually referred to as the {\em Nerve Lemma}) that

\begin{theorem}[Nerve Lemma]
\label{the:nerve}
Suppose that the cover ${\mathcal C}(S)$  satisfies the Leray property.
Then for each $i \geq 0$,
\[ \HH^i(\LL^{\bullet}({\mathcal C}(S))) \cong \HH^i(S).\]
\end{theorem}
(See for instance \cite{Rotman} for a proof.)

\begin{remark}
There are several interesting extensions of Theorem
\ref{the:nerve} (Nerve Lemma). For instance, 
if the Leray property is weakened to say that each $t$-ary intersection
is $(k-t+1)$-connected, then one can conclude that the nerve complex is 
$k$-connected. We refer the reader to the article by Bj\"{o}rner 
\cite{Bjorner} for more details.
\end{remark}

Notice that Theorem \ref{the:nerve} gives a method for 
computing the Betti numbers of $S$ using linear algebra
from a cover of $S$ by contractible sets for which all
non-empty intersections are also contractible, once we are able to  test 
emptiness of the various intersections $S_{\alpha_0\cdots\alpha_p}$.

Now suppose that each individual member, $S_{\alpha_0}$, of the cover is 
contractible, but the various intersections  
$S_{\alpha_0\cdots\alpha_p}$
are not necessarily contractible for $p \geq 1$. 
Theorem \ref{the:nerve} does not hold in this case. 
However, the following theorem is proved in \cite{BPR9} and underlies the
single exponential algorithm for computing the first Betti number of
semi-algebraic sets described there.

\begin{theorem}\cite{BPR9}
\label{the:bettione}
Suppose that each individual member, $S_{\alpha_0}$, of the cover 
${\mathcal C}(S)$ is  contractible.
Then,
\[ \HH^i(\LL^{\bullet}_2({\mathcal C}(S))) \cong \HH^i(S),
\mbox{ for } i = 0,1.
\]
\end{theorem}

\begin{remark}
\label{rem:bettione}
Notice that from a cover by contractible sets
Theorem \ref{the:bettione} allows us to compute using linear algebra,
$b_0(S)$ and $b_1(S)$, once we have identified the non-empty 
connected components
of the pair-wise and triple-wise intersections of the sets in the cover
and their inclusion relationships. 
\end{remark}

\begin{example}
We illustrate Remark \ref{rem:bettione} with a simple example.

Consider the following set $S$ depicted in Figure  \ref{fig-bettione}
below  and let ${\mathcal C}(S) = \{0,1,2\}$
and the corresponding sets $S_0,S_1,S_2$ are the three edges as shown in
in the figure. 

\hide{
\begin{center}
\begin{figure}
\includegraphics[height=5cm]{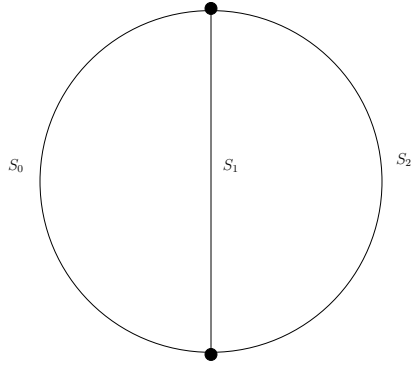}
\caption{Example illustrating Theorem \ref{the:bettione}}  
\label{fig-bettione}
\end{figure}
\end{center}
}

       \begin{figure}[hbt]
         \centerline{
           \scalebox{0.5}{
 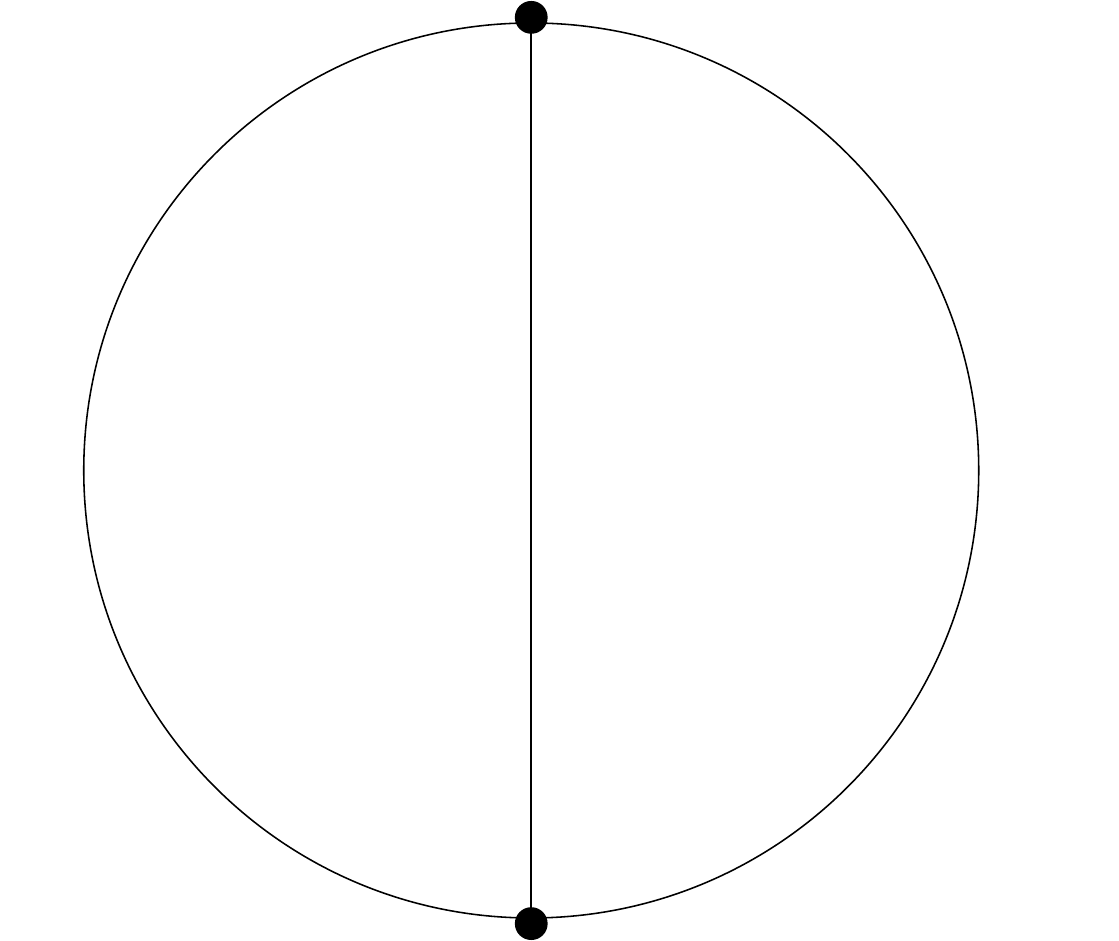
             }
           }
\caption{Example illustrating Theorem \ref{the:bettione}}  
\label{fig-bettione}
       \end{figure}

Notice that each pair-wise and triple-wise intersections
in this case has two connected components. 
Let us construct the complex $\LL^{\bullet}_2({\mathcal C}(S))$.
We have
\begin{equation}
\LL^0({\mathcal C}(S)) = 
\bigoplus_{\alpha_0\in {\mathcal C}(S)} 
\HH^0(S_0) \oplus \HH^0(S_1)\oplus \HH^0(S_1)  \cong\Q\oplus\Q\oplus\Q,
\end{equation}
\begin{equation}
\LL^1({\mathcal C}(S)) = 
\bigoplus_{\alpha_0 < \alpha_1 \in {\mathcal C}(S)} 
\HH^0(S_{01}) \oplus \HH^0(S_{02})\oplus \HH^0(S_{12})  
\cong\Q\oplus\Q\oplus\Q\oplus\Q\oplus\Q\oplus\Q,
\end{equation}
and
\begin{equation}
\LL^2({\mathcal C}(S)) = 
\bigoplus_{\alpha_0<\alpha_1<\alpha_2\in {\mathcal C}(S)} 
\HH^0(S_{012})  \cong\Q\oplus\Q.
\end{equation}

We now display the matrices $M_0$ and $M_1$ 
corresponding to the homomorphisms $\delta^0$ and
$\delta^1$ respectively (with respect to the obvious basis 
corresponding to the connected components of the various intersections).

We have
\begin{equation}
M_0 = 
\left(
\begin{array}{ccc}
1 & -1 & 0 \\ 
1 & -1 & 0 \\ 
1 & 0  & -1\\
1 & 0  & -1\\
0 & 1  & -1 \\
0 & 1  & -1
\end{array}
\right),
\end{equation}
and
\begin{equation}
M_1= 
\left(
\begin{array}{cccccc}
1 & 0 & -1 & 0 & 1 & 0 \\
0 & 1 & 0 & -1 & 0 & 1 
\end{array}
\right).
\end{equation}

It is now easy to verify 
that ${\rm rank}(M_0) = 2$ and ${\rm rank}(M_1) = 2$.
We derive applying Theorem \ref{the:bettione} that
\begin{align*}
b_0(S) &= \dim \ker M_0  \\
       &= \dim (\LL^0({\mathcal C}(S))) - {\rm rank}(M_0) \\
       &= 3 - 2 \\
       &= 1,
\end{align*}
and 
\begin{align*}
b_1(S) &= \dim \ker(M_1) - \dim {\rm Im}(M_0)  \\
       &= \dim (\LL^1({\mathcal C}(S))) - {\rm rank}(M_1) - {\rm rank}(M_0)\\
       &= 6 -2- 2 \\
       &= 2.
\end{align*}
We refer the reader to \cite{BK05} for more complicated higher dimensional
examples of a similar nature.
\end{example}

\begin{remark}
\label{rem:counterexample}
It is easy to see that if we
extend the complex in Theorem  \ref{the:bettione} by one more term,
that is consider the complex, $\LL^{\bullet}_{3}({\mathcal C}(S))$,
then the cohomology of the complex does not yield information about
$\HH^2(S)$. Just consider the cover of the standard sphere 
$\Sphere^2 \subset \R^3$  and the cover $\{H_1,H_2\}$ of $\Sphere^2$ where
$H_1,H_2$ are closed hemispheres meeting at the equator. The 
corresponding complex, $\LL^{\bullet}_{3}({\mathcal C})$,  is as follows.

\[
0 \rightarrow \HH^0(H_1) \bigoplus \HH^0(H_2) 
\stackrel{\delta^0} \longrightarrow 
\HH^0(H_{1}\cap H_{2}) \stackrel{\delta^1} \longrightarrow 
0
\longrightarrow
0
\]

Clearly, $\HH^2(\LL^{\bullet}_{3}({\mathcal C})) \not\simeq \HH^2(\Sphere^2)$,
and indeed it is impossible to compute  $b_i(S)$ 
just from the information on the number of connected components of 
intersections of the sets of a cover of $S$ 
by contractible sets for  $i \geq 2$.
For example the nerve complex corresponding to the cover of the
sphere by two hemispheres is isomorphic to the nerve complex of a
cover of the unit segment $[0,1]$ by the subsets $[0,1/2]$ and 
$[1/2,1]$, but clearly $\HH^2(\Sphere^2) = \Q$, while $\HH^2([0,1]) = 0$. 
\end{remark}

\subsection{Non-Leray Covers}
\label{subsec:nonleray}
In the design of algorithms for computing covers of semi-algebraic sets
it is often difficult to satisfy the full Leray property.
In order to utilize  covers not satisfying the Leray property it
is necessary to consider a generalization of the nerve complex.
However, before we can describe this generalization we need to expand
slightly the algebraic machinery at our disposal.

We first introduce the notion of a {\em double complex} which is 
in essence a {\em complex of complexes}.

\begin{definition}[Double complex]
\label{def:doublecomplex}
A {\em double complex} is a bi-graded vector space
\[
{\Ch}^{\bullet,\bullet} = \bigoplus_{p,q \in {\mathbb Z}} \Ch^{p,q}
\]
with co-boundary operators
$d : \Ch^{p,q} \rightarrow \Ch^{p,q+1}$ and
$\delta: \Ch^{p,q} \rightarrow \Ch^{p+1,q}$ and such that 
$d\circ\delta +\delta\circ d = 0$ (see diagram below).
We say that $\Ch^{\bullet,\bullet}$ is {\em a first quadrant
double complex} if it additionally satisfies
the condition that $\Ch^{p,q} = 0$ if either $p < 0$ or $q < 0$.
Double complexes lying in other quadrants are defined in an analogous manner.
\[
\begin{diagram}
\node{\vdots} \node{\vdots}\node{\vdots} \\
\node{\Ch^{0,2}} \arrow{e,t}{\delta}\arrow{n,l}{d}
\node{\Ch^{1,2}}\arrow{e,t}{\delta}\arrow{n,l}{d}
\node{\Ch^{2,2}}\arrow{e,t}{\delta}\arrow{n,l}{d}\node{\cdots} \\
\node{\Ch^{0,1}} \arrow{e,t}{\delta}\arrow{n,l}{d}
\node{\Ch^{1,1}}\arrow{e,t}{\delta}\arrow{n,l}{d}
\node{\Ch^{2,1}}\arrow{e,t}{\delta}\arrow{n,l}{d}\node{\cdots} \\
\node{\Ch^{0,0}} \arrow{e,t}{\delta}\arrow{n,l}{d}
\node{\Ch^{1,0}}\arrow{e,t}{\delta}\arrow{n,l}{d}
\node{\Ch^{2,0}}\arrow{e,t}{\delta}\arrow{n,l}{d}\node{\cdots} \\
\end{diagram}
\]
\end{definition}

\begin{definition}[The Associated Total Complex]
\label{def:totalcomplex}
The complex defined by 
\[
{\rm Tot}^n(\Ch^{\bullet,\bullet}) = \bigoplus_{p+q=n} \Ch^{p,q},
\]
with differential
\[
\D^n  = 
\bigoplus_{p+q=n} d + (-1)^{p}\delta: 
{\rm Tot}^{n}(\Ch^{\bullet,\bullet}) 
\longrightarrow {\rm Tot}^{n+1}(\Ch^{\bullet,\bullet}),
\]
is denoted by ${\rm Tot}^{\bullet}(\Ch^{\bullet,\bullet})$ and called
the {\em associated total complex} of $\Ch^{\bullet,\bullet}$.
\end{definition}

Associated to double complexes, or more accurately to their filtrations,
is another algebraic object that is quite ubiquitous in 
modern algebraic topology.

\begin{definition}[Spectral Sequence]
\label{def:spectral}
A {\em (cohomology) spectral sequence} is a sequence of bi-graded 
(this is a direct sum of vector subspaces indexed by $\Z \times \Z$) 
complexes
$\{E_r^{i,j} \mid i,j,r \in \mathbb{Z}, r \geq a\}$
endowed with differentials 
$d_r^{i,j}: E^{i,j}_r \rightarrow E^{i+r,j-r+1}_r$ such that
$(d_r)^2=0$ for all $r.$ 
Moreover, we require the existence of isomorphism between 
the complex $E_{r+1}$ 
and the homology of $E_r$ with respect to $d_r$: 
$$
E_{r+1}^{i,j} \cong \HH_{d_r}(E_r^{i,j})=
\frac{\ker d_r^{i,j}}{d_r^{i+r,j-r+1}\left(E_r^{i+r,j-r+1}\right)}
$$

The spectral sequence is called a {\em first quadrant spectral sequence}
(see Figure \ref{fig:spectral})
if the initial complex $E_a$ lies in the first quadrant,
i.e. $E_a^{i,j}=0$ whenever $ij<0.$ In that case, all subsequent
complexes $E_r$ also lie in the first quadrant.  Since the
differential $d_r^{i,j}$ maps outside of the first quadrant for $r>i$,
the homomorphisms of a first quadrant spectral sequence $d_r$ are
eventually zero, and thus the groups $E_r^{i,j}$ are all isomorphic to
a fixed group $E_\infty^{i,j}$ for $r$ large enough, and we say the
spectral sequence is convergent.
\end{definition}

\begin{figure}[hbt] 
\begin{center}
\begin{picture}(0,0)%
\includegraphics{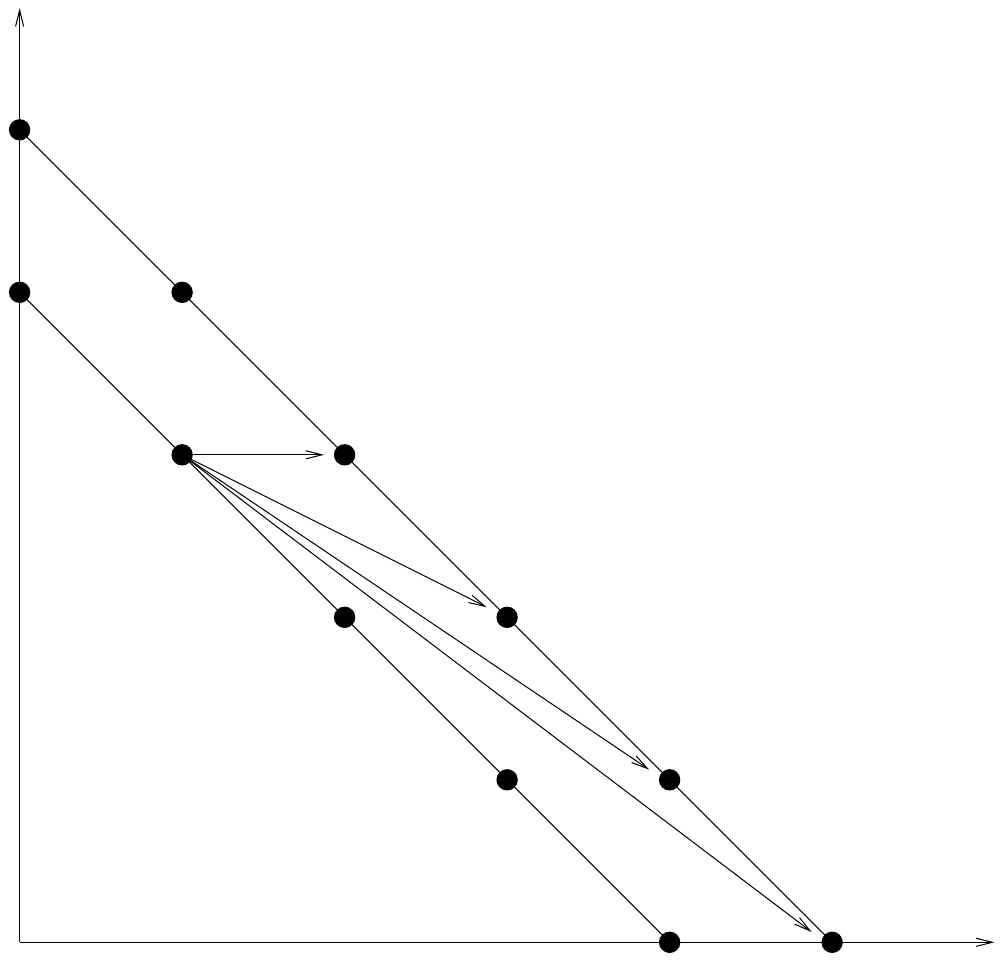}%
\end{picture}%
\setlength{\unitlength}{2565sp}%
\begingroup\makeatletter\ifx\SetFigFont\undefined%
\gdef\SetFigFont#1#2#3#4#5{%
  \reset@font\fontsize{#1}{#2pt}%
  \fontfamily{#3}\fontseries{#4}\fontshape{#5}%
  \selectfont}%
\fi\endgroup%
\begin{picture}(7512,7414)(601,-7394)
\put(6751,-7336){\makebox(0,0)[lb]{\smash{\SetFigFont{8}{9.6}{\familydefault}{\mddefault}{\updefault}{$p + q = \ell+1$}%
}}}
\put(5251,-7336){\makebox(0,0)[lb]{\smash{\SetFigFont{8}{9.6}{\familydefault}{\mddefault}{\updefault}{$p + q = \ell$}%
}}}
\put(8101,-7186){\makebox(0,0)[lb]{\smash{\SetFigFont{8}{9.6}{\familydefault}{\mddefault}{\updefault}{$p$}%
}}}
\put(601,-136){\makebox(0,0)[lb]{\smash{\SetFigFont{8}{9.6}{\familydefault}{\mddefault}{\updefault}{$q$}%
}}}
\put(2401,-3286){\makebox(0,0)[lb]{\smash{\SetFigFont{8}{9.6}{\familydefault}{\mddefault}{\updefault}{$d_1$}%
}}}
\put(3226,-3886){\makebox(0,0)[lb]{\smash{\SetFigFont{8}{9.6}{\familydefault}{\mddefault}{\updefault}{$d_2$}%
}}}
\put(4276,-4786){\makebox(0,0)[lb]{\smash{\SetFigFont{8}{9.6}{\familydefault}{\mddefault}{\updefault}{$d_3$}%
}}}
\put(5851,-6211){\makebox(0,0)[lb]{\smash{\SetFigFont{8}{9.6}{\familydefault}{\mddefault}{\updefault}{$d_4$}%
}}}
\end{picture}
\caption{$d_r: E_r^{p,q} \rightarrow E_r^{p+r, q- r +1}$}
\label{fig:spectral}
\end{center}
\end{figure}

There are two spectral sequences,
$'E_*^{p,q},{''E}_*^{p,q}$,  (corresponding to taking row-wise or
column-wise filtrations respectively) 
associated with a first quadrant double complex $\Ch^{\bullet,\bullet}$ 
which  will be important for us. 
Both of these converge to $\HH^*({\rm Tot}^{\bullet}(\Ch^{\bullet,\bullet})).$
This means that the homomorphisms, $d_r$ are eventually zero, and hence the
spectral sequences stabilize and
\begin{equation}
\label{eqn:totspectral}
\bigoplus_{p+q = i}{'E}_{\infty}^{p,q} \cong 
\bigoplus_{p+q = i}{''E}_{\infty}^{p,q} \cong 
\HH^i({\rm Tot}^{\bullet}(\Ch^{\bullet,\bullet}))
\end{equation}
for each $i \geq 0$.

The first terms of these are
\begin{equation}
\label{eqn:'E}
{'E}_1 = \HH_{d}(\Ch^{\bullet,\bullet}), 
{'E}_2 = \HH_{\delta}\HH_d(\Ch^{\bullet,\bullet}),
\end{equation}
and
\begin{equation}
\label{eqn:''E}
{''E}_1 = \HH_{\delta} (\Ch^{\bullet,\bullet}), 
{''E}_2 = \HH_d \HH_{\delta} (\Ch^{\bullet,\bullet}).
\end{equation}

Given  two (first quadrant) double complexes, $\Ch^{\bullet,\bullet}$ and
$\bar{\Ch}^{\bullet,\bullet}$, a homomorphism of double complexes,
\[
\phi^{\bullet,\bullet}: \Ch^{\bullet,\bullet} \longrightarrow 
\bar{\Ch}^{\bullet,\bullet}
\]
is a collection of homomorphisms
$\phi^{p,q}: \Ch^{p,q} \longrightarrow \bar{\Ch}^{p,q}$ such that all
the following diagrams commute.

\[
\begin{array}{ccc}
\Ch^{p,q} &
\stackrel{\delta}{\longrightarrow} & \Ch^{p+1,q} \\ 
\Big\downarrow\vcenter{\rlap{$\phi^{p,q}$}} & &
\Big\downarrow\vcenter{\rlap{$\phi^{p+1,q}$}} \\ 
\bar{\Ch}^{p,q} &
\stackrel{\delta}{\longrightarrow} & \bar{\Ch}^{p+1,q}
\end{array}
\]

\[
\begin{array}{ccc}
 \Ch^{p,q} &
\stackrel{d}{\longrightarrow} & \Ch^{p,q+1}\\
\Big\downarrow\vcenter{\rlap{$\phi^{p,q}$}} & &
\Big\downarrow\vcenter{\rlap{$\phi^{p,q+1}$}} \\ 
\bar{\Ch}^{p,q} &
\stackrel{d}{\longrightarrow} & \bar{\Ch}^{p,q+1}
\end{array}
\]

A homomorphism of double complexes
\[
\phi^{\bullet,\bullet}: \Ch^{\bullet,\bullet} \longrightarrow \bar{\Ch}^{\bullet,\bullet},
\]
induces an homomorphism of the corresponding total complexes which we
will denote by
\[
\Tot^{\bullet}(\phi^{\bullet,\bullet}): \Tot^{\bullet}(\Ch^{\bullet,\bullet}) \longrightarrow \Tot^{\bullet}(\bar{\Ch}^{\bullet,\bullet}).
\]
It also induces homomorphisms
$'\phi_s: {'E}_s \longrightarrow {'\bar{E}}_s$
(respectively,
$''\phi_s: {''E}_s \longrightarrow {''\bar{E}}_s$)
between the associated spectral sequences (corresponding either to the
row-wise or column-wise filtrations).
For the precise definition of homomorphisms of spectral sequences
see \cite{Mcleary}.
We will need the following useful fact 
(see \cite[pp. 66]{Mcleary}).

\begin{theorem}[Comparison Theorem]
\label{the:spectral}
If ${'\phi}_s$  (respectively, ${''\phi}_s$)
is an isomorphism for some $s \geq 1$ then
${'E}_r^{p,q}$ and ${'\bar{E}}_r^{p,q}$ 
(respectively, ${''E}_r^{p,q}$ and ${''\bar{E}}_r^{p,q}$)
are isomorphic for
all $r \geq s$. In particular the induced homomorphism
\[
\Tot^{\bullet}(\phi^{\bullet,\bullet}): \Tot^{\bullet}(\Ch^{\bullet,\bullet}) \longrightarrow 
\Tot^{\bullet}(\bar{\Ch}^{\bullet,\bullet})
\]
is a quasi-isomorphism.
\end{theorem}

Having introduced the definitions of double complexes and spectral sequences
above, we now describe two particular double complexes that are of interest
to us in this paper.

\subsection{The Mayer-Vietoris Double Complex and its 
Associated Spectral Sequence}
\label{subsec:MV}
Let $A_1,\ldots,A_n$ be sub-complexes of a finite simplicial complex
$A$ such that $A = A_1 \cup \cdots \cup A_n$. 
Note that the intersections of any number of the sub-complexes, $A_i$,
is again a sub-complex of $A$.
We  denote by $A_{\alpha_0\cdots \alpha_p}$ the sub-complex
$A_{\alpha_0} \cap \cdots \cap A_{\alpha_p}$.

\begin{definition}[The Generalized Mayer-Vietoris Exact Sequence]
\label{def:MVsequence}
The {\em generalized Mayer-Vietoris sequence} is the following
exact sequence of vector spaces.

\[
0 \longrightarrow \Ch^{\bullet}(A) \stackrel{r^{\bullet}}{\longrightarrow}
\bigoplus_{1 \leq \alpha_0 \leq n} \Ch^{\bullet}(A_{\alpha_0}) 
\stackrel{\delta^{0,\bullet}}{\longrightarrow}
\bigoplus_{1 \leq \alpha_0<\alpha_1 \leq n} \Ch^{\bullet}(A_{\alpha_0\cdot \alpha_1}) 
 \stackrel{\delta^{1,\bullet}}{\longrightarrow}
\cdots 
\]
\[
\bigoplus_{1 \leq \alpha_0 < \cdots < \alpha_p \leq n}\Ch^{\bullet}(A_{\alpha_0\cdots \alpha_p})
 \stackrel{\delta^{p-1,\bullet}}{\longrightarrow}
\bigoplus_{1 \leq \alpha_0< \cdots <\alpha_{p+1} \leq n}\Ch^{\bullet}(A_{\alpha_0\cdots \alpha_{p+1}})
 \stackrel{\delta^{p,\bullet}}{\longrightarrow}
\cdots
\]

\hide{
\[
\begin{diagram}
\node{0} \arrow{e} 
\node{\Ch^{\bullet}(A)}\arrow{e,t}{r^{\bullet}}
\node{\bigoplus_{1 \leq \alpha_0 \leq n} \Ch^{\bullet}(A_{\alpha_0})}
\arrow{e,t}{\delta^{0,\bullet}}
\node{\bigoplus_{1 \leq \alpha_0<\alpha_1 \leq n} \Ch^{\bullet}(A_{\alpha_0\cdot \alpha_1})} 
\arrow{e,t}{\delta^{1,\bullet}}
\node{\cdots} \\
\end{diagram}
\begin{diagram}
\\
\node{\bigoplus_{1 \leq \alpha_0 < \cdots < \alpha_p \leq n}\Ch^{\bullet}(A_{\alpha_0\cdots \alpha_p})}
\arrow{e,t}{\delta^{p-1,\bullet}}
\node{\bigoplus_{1 \leq \alpha_0< \cdots <\alpha_{p+1} \leq n}\Ch^{\bullet}(A_{\alpha_0\cdots \alpha_{p+1}})}
\arrow{e,t}{\delta^{p,\bullet}}
\node{\cdots}
\end{diagram}
\]
}
where $r^{\bullet}$ is induced by restriction and the homomorphisms 
$\delta^{p,\bullet}$ are defined as follows.

Given an $\omega \in \bigoplus_{\alpha_0< \cdots <\alpha_p}\Ch^q(A_{\alpha_0\cdots \alpha_p})$ 
we define $\delta^{p,q}(\omega)$ as follows:

First note that 
$
\displaystyle{
\delta^{p,q}\omega \in \bigoplus_{\alpha_0< \cdots <\alpha_{p+1}}\Ch^q(A_{\alpha_0 \cdots \alpha_{p+1}})},$ 
and it suffices to define
\[
(\delta^{p,q}\omega)_{\alpha_0,\ldots,\alpha_{p+1}}
\] 
for each $(p+2)$-tuple
$1 \leq \alpha_0< \cdots <\alpha_{p+1} \leq n$.
Note that, $(\delta^{p,q}\omega)_{\alpha_0,\ldots,\alpha_{p+1}}$ is a linear
form on the vector space, $C_q(A_{\alpha_0\cdots \alpha_{p+1}})$, and
hence is determined by its values on the $q$-simplices 
in the complex $A_{\alpha_0\cdots \alpha_{p+1}}$. Furthermore, 
each $q$-simplex, $s \in  A_{\alpha_0\cdots \alpha_{p+1}}$ is automatically
a simplex of the complexes 
\[
A_{\alpha_0\cdots\hat{\alpha_i}\cdots \alpha_{p+1}}, \; 0 \leq i \leq p+1.
\]

We define
\[
\label{delta}
(\delta^{p,q} \omega)_{\alpha_0,\ldots,\alpha_{p+1}}(s) =  
 \sum_{0 \leq j \leq p+1} (-1)^i \omega_{\alpha_0,\ldots,\hat{\alpha_j},\ldots,\alpha_{p+1}}
(s).
\]
\end{definition}

The fact that the generalized Mayer-Vietoris  sequence
is exact is classical (see \cite{Rotman} or \cite{Basu4} for example).

We now define the Mayer-Vietoris double complex of the complex $A$
with respect to the sub-complexes $A_{\alpha_0}, 1 \leq \alpha_0 \leq n$, 
which we will denote by  $\N^{\bullet,\bullet}(A)$ (we suppress the
dependence of the complex on sub-complexes $A_{\alpha_0}$ in the notation
since this dependence will be clear from context). 

\begin{definition}[Mayer-Vietoris Double Complex]
\label{def:MV}
The Mayer-Vietoris double complex of a simplicial complex $A$
with respect to the sub-complexes $A_{\alpha_0}, 1 \leq \alpha_0 \leq n$, 
$\N^{\bullet,\bullet}(A)$, is the double complex defined by
$$
\displaylines{
\N^{p,q}(A) = 
\bigoplus_{1 \leq \alpha_0< \cdots <\alpha_p \leq n}
\Ch^q(A_{\alpha_0\cdots\alpha_p}).
}
$$
The horizontal differentials are as defined above. The vertical
differentials are those induced by the ones in the different
complexes, $\Ch^{\bullet}(A_{\alpha_0\cdots\alpha_p}).$

$\N^{\bullet,\bullet}(A)$ is depicted in the following figure.
\begin{equation}
\label{eqn:MVfig}
\begin{diagram}
\node{}\node{}\node{}\\
\node{\bigoplus_{\alpha_0} \Ch^2(A_{\alpha_0})} \arrow{e,t}{\delta^{0,2}}
\arrow{n,l}{d^2}
\node{\bigoplus_{\alpha_0 < \alpha_1} \Ch^2(A_{\alpha_0\cdot\alpha_1})}
\arrow{e,t}{\delta^{1,2}}\arrow{n,l}{d^2}\node{\cdots} \\
\node{\bigoplus_{\alpha_0} \Ch^1(A_{\alpha_0})} 
\arrow{e,t}{\delta^{0,1}}\arrow{n,l}{d^1}
\node{\bigoplus_{\alpha_0 < \alpha_1} \Ch^1(A_{\alpha_0\cdot\alpha_1})}
\arrow{e,t}{\delta^{1,1}}\arrow{n,l}{d^1}\node{\cdots} \\
\node{\bigoplus_{\alpha_0} \Ch^0(A_{\alpha_0})}
\arrow{e,t}{\delta^{0,0}}\arrow{n,l}{d^0}
\node{\bigoplus_{\alpha_0 < \alpha_1} \Ch^0(A_{\alpha_0\cdot\alpha_1})}
\arrow{e,t}{\delta^{1,0}}\arrow{n,l}{d^0}\node{\cdots} 
\end{diagram}
\end{equation}
\end{definition}

\begin{remark}
\label{rem:dual}
There is also dual version of the Mayer-Vietoris double complex where the
unions and intersections are inter-changed and the directions of the
arrows get reversed. The reader is referred to \cite{Basu4} for 
more detail.
\end{remark}

Finally, for complexity reasons it is often  useful to consider truncations
of the Mayer-Vietoris double complex. 

\begin{definition}[Truncated Mayer-Vietoris Double Complex]
\label{def:MVdoubletruncated}
For any $t \geq 0,$
we denote by $\N_t^{\bullet,\bullet}(A)$ the following truncated complex.
\[
\begin{array}{cccc}
\N_t^{p,q}(A) & = & \N^{p,q}(A), & 
0 \leq p+q \leq t, \cr
\N_t^{p,q}(A) & = &  0, & \mbox{otherwise}.\cr
\end{array}
\]
\end{definition}

The following proposition is classical (see \cite {Rotman} or \cite{Basu4} 
for a proof) and 
follows from the exactness of the generalized Mayer-Vietoris sequence.

\begin{proposition}
\label{prop:MV}
The  spectral sequences,
${'E}_r,{''E}_r$,
associated to $\N^{\bullet,\bullet}(A)$ converge to $\HH^*(A)$ and thus,
$$
\displaylines{
\HH^*(\Tot^{\bullet}(\N^{\bullet,\bullet}(A))) \cong \HH^*(A).
}
$$
Moreover, the homomorphism 
\[
\psi^{\bullet}: \Ch^{\bullet}(A) \rightarrow \Tot^{\bullet}(\N^{\bullet,\bullet}(A))
\]
induced by the homomorphism $r^{\bullet}$ (in the generalized Mayer-Vietoris
sequence) is a quasi-isomorphism.
\end{proposition}

We denote by  $\Ch^{\bullet}_{\ell+1}(A)$ the truncation of the
complex $\Ch^{\bullet}(A)$ after the $(\ell+1)$-st term.
As an immediate corollary we have that,
\begin{corollary}
\label{cor:MV}
For any $\ell \geq 0$,
the homomorphism 
\begin{equation}
\label{eqn:restriction}
\psi_{\ell+1}^{\bullet}: \Ch^{\bullet}_{\ell+1}(A) \rightarrow \Tot^{\bullet}(\N_{\ell+1}^{\bullet,\bullet}(A))
\end{equation}

induced by the homomorphism $r^{\bullet}$ (in the generalized Mayer-Vietoris
sequence) is a quasi-isomorphism.
Hence, for $0 \leq i \leq \ell,$
$$
\displaylines{
\HH^i(\Tot^{\bullet}(\N_{\ell+1}^{\bullet,\bullet}(A))) \cong \HH^i(A).
}
$$
\end{corollary}

\begin{remark}
\label{rem:induction}
Notice that in the truncated Mayer-Vietoris double complex,
$\N_t^{\bullet,\bullet}(A)$, the $0$-th column is a complex
having at most $t+1$ non-zero terms, the first column can have at most
$t$ non-zero terms, and in general the $i$-th column has at most $t+1 - i$
non-zero terms. This observation along with the fact that,
{\em 
each term in the double complex $\N_t^{\bullet,\bullet}(A)$
depends on tuples of at most $t+1$ of the $A_{\alpha}$'s at a time, 
}
play a crucial role in the inductive
arguments used in the design of single exponential time
algorithm for computing the first few Betti numbers of
semi-algebraic sets. 
\end{remark}

\subsection{The Descent Double Complex and its Associated Spectral Sequence}
\label{subsec:descent}
For the algorithmic problem of computing the Betti numbers of 
projections of semi-algebraic sets, another spectral sequence
plays an important role.

\begin{definition}[Locally split maps]
\label{df:split}
A continuous surjection $f:X \to Y$ is called {\em locally split} if there
exists an open covering $\U$ of $Y$ such that for all $U \in \U$,
there exists a continuous section $\sigma: U \to X$ of $f$,
i.e. $\sigma$ is a continuous map such that $f(\sigma(y))=y$ for all
$y \in U$. 
\end{definition}

In particular, if $X$ is an open semi-algebraic set and $f: X \to Y$
is a projection, the map $f$ is obviously locally split. 
For any semi-algebraic surjection $f:X \to Y$, we denote
by $W_f^p(X)$ the $(p+1)$-fold fibered power of $X$ over $f$,
\[
W^p_f(X) = \{(\bar{x}_0,\ldots,\bar{x}_p) \in X^{p+1}
\mid f(\bar{x}_0) = \cdots = f(\bar{x}_p)\}.
\]
The map $f$ induces for each $p \geq 0,$ 
a map from $W_f^p(X)$ to $Y$,
sending $(\bar{x}_0,\ldots,\bar{x}_p)$ to the common value 
$f(\bar{x}_0) = \cdots = f(\bar{x}_p),$
and 
abusing notation a little
we will denote this  map by $f$ as well. 

\subsubsection{The Descent Double Complex}
\label{subsubsec:ddd}
Let $\Ch^{\bullet}(W_f^p(X))$ denote the singular co-chain complex of
$W_f^p(X)$ (refer to \cite{Rotman} for definition). 
For each $p \geq 0$, we now define a homomorphism,
\[
\delta^p: \Ch^{\bullet}(W_f^p(X)) \longrightarrow 
\Ch^{\bullet}(W_f^{p+1}(X))
\]
as follows:
for each $i, 0 \leq i \leq p$, define
$\pi_{p,i}: W_f^p(X) \rightarrow W_f^{p-1}(X)$  by
\[
\pi_{p,i}({x}_0,\ldots,{x}_p) = ({x}_0,\ldots,\widehat{{x}_i},
\ldots,{x}_p)
\] 
($\pi_{p,i}$ drops the $i$-th coordinate).

We will denote by $(\pi_{p,i})_*$ the induced map on
$C_\bullet(W_f^{p}(X)) \rightarrow C_\bullet(W_f^{p-1}(X))$
and let $\pi_{p,i}^*: \Ch^{\bullet}(W_f^{p-1}(X)) \rightarrow 
\Ch^{\bullet}(W_f^{p}(X))$
denote the dual map.
For $\phi \in \Ch^{\bullet}(W_f^p(X))$, we define
$\delta^p \;\phi$ by
\begin{equation}\label{eqn:deltadef}
\delta^p\;\phi =\sum_{i=0}^{p+1}(-1)^i \pi_{p+1,i}^*\ \phi.
\end{equation}

\begin{definition}[Descent Double Complex]
\label{def:descentdoublecomplex}
Now, let $\D^{\bullet,\bullet}(X)$ denote the  double complex defined by
$\D^{p,q}(X) = \Ch^q(W_f^p(X))$ with vertical and horizontal
homomorphisms
given by $\dd^q=(-1)^p d^q$ and $\delta$ respectively,
where $d$ is the singular coboundary operator,
and $\delta$ is the map defined in ~\eqref{eqn:deltadef}. Also, let
$\D^{p,q}(X) = 0$ if $ p < 0$ or $q < 0$. 
\end{definition}
{\tiny
\[
\begin{diagram}
\node{}\node{\vdots}\node{\vdots}\node{\vdots} \\
\node{0} \arrow{e}\node{\Ch^3(W_f^0(X))}\arrow{e,t}{\delta}
\arrow{n,l}{\dd}
\node{\Ch^3(W_f^1(X))}\arrow{e,t}{\delta}\arrow{n,l}{\dd}
\node{\Ch^3(W_f^2(X))} \arrow{e,t}{\delta}\arrow{n,l}{\dd} \\
\node{0} \arrow{e}\node{\Ch^3(W_f^0(X))}\arrow{e,t}{\delta}
\arrow{n,l}{\dd}
\node{\Ch^2(W_f^1(X))}\arrow{e,t}{\delta}\arrow{n,l}{\dd}
\node{\Ch^2(W_f^2(X))} \arrow{e,t}{\delta}\arrow{n,l}{\dd} \\
\node{0} \arrow{e}\node{\Ch^1(W_f^0(X))}\arrow{e,t}{\delta}
\arrow{n,l}{\dd}
\node{\Ch^1(W_f^1(X))}\arrow{e,t}{\delta}\arrow{n,l}{\dd}
\node{\Ch^1(W_f^2(X))} \arrow{e,t}{\delta}\arrow{n,l}{\dd} \\
\node{0} \arrow{e}\node{\Ch^0(W_f^0(X))}\arrow{e,t}{\delta}
\arrow{n,l}{\dd}
\node{\Ch^0(W_f^1(X))}\arrow{e,t}{\delta}\arrow{n,l}{\dd}
\node{\Ch^0(W_f^2(X))} \arrow{e,t}{\delta}\arrow{n,l}{\dd} \\
\node{}\node{0}\arrow{n} \node{0}\arrow{n} \node{0}\arrow{n}
\end{diagram}
\]
}

\begin{theorem}\cite{GVZ,BZ}
\label{the:descent}
For any continuous  semi-algebraic surjection $f:X \to Y$,
where $X$ and $Y$ are open semi-algebraic subsets of $\R^n$ and $\R^m$ 
respectively (or, more generally, for any locally split continuous
surjection $f$),
the spectral sequence associated to the double complex
$\D^{\bullet,\bullet}(X)$ with 
$E_1 = \HH^d(\D^{\bullet,\bullet}(X))$ converges  to 
$\HH^*(\Ch^\bullet(Y)) \cong \HH^*(Y).$
In particular,
\begin{enumerate}
\item
$\displaystyle{
E_1^{i,j} = \HH^j(W_f^i(X)),
}
$ and
\item
$\displaystyle{
E_\infty \cong  \HH^*(\Tot^{\bullet}(\D^{\bullet,\bullet}(X))) \cong \HH^*(Y).
}
$
\end{enumerate}
\end{theorem}

We will also need 
a well-known construction in homotopy theory called {\em homotopy colimits},
which we define now.

\subsection{Homotopy Colimits}
\label{subsec:hocolimit}
Let ${\mathcal A} = \{A_1,\ldots,A_n\}$,  where each $A_i$ is a sub-complex
of a finite CW-complex.
Let $\Delta_{[n]}$ denote the standard simplex of dimension $n-1$ with
vertices  in $[n]$.
For $I \subset [n]$, we denote by $\Delta_I$ 
the $(\#I-1)$-dimensional face of $\Delta_{[n]}$ corresponding
to $I$, and by $A_I$ (resp. $A^I$) the CW-complex 
$\displaystyle{\bigcap_{i \in I} A_i}$ (resp.
$\displaystyle{\bigcup_{i \in I} A_i}$).
The homotopy colimit, $\hocolimit({\mathcal A})$,  is a CW-complex defined as follows.
\begin{definition}[Homotopy colimit]
\label{def:hocolimit}
\begin{equation}
\label{eqn:defofhocolimit}
\hocolimit({\mathcal A}) = 
\bigcupdot_{I \subset [n]}
\Delta_I \times A_I/\sim
\end{equation}
where the equivalence relation $\sim$ is defined as follows.
For $I \subset J \subset [n]$, let $s_{I,J}: \Delta_I \hookrightarrow \Delta_J$
denote the inclusion map of the face $\Delta_I$ in $\Delta_J$, and let
$i_{IJ}: A_J \hookrightarrow A_I$ denote the inclusion map of
$A_J$ in $A_I$.
Given $(s,x) \in \Delta_I \times A_I$ and $(t,y) \in \Delta_J 
\times A_J$ with $I \subset J$, then $(s,x) \sim 
(t,y)$ if and only if
$t = s_{IJ}(s)$ and $x = i_{IJ}(y)$.
\end{definition}

We have an obvious map 
\begin{equation}
\label{eqn:defoff}
f_{{\mathcal A}}: \hocolimit({\mathcal A}) \longrightarrow 
\colimit({\mathcal A}) = A^{[n]}
\end{equation}
sending $(s,x) \mapsto x$. 
Notice that for each $x \in A^{[n]}$, 
\[
f_{{\mathcal A}}^{-1}(x) = \overline{|\Delta_{I_x}|},
\]
where $I_x = \{i \;\mid\; x \in A_i\}$. In particular, 
$f_{{\mathcal A}}^{-1}(x)$
is contractible for each $x \in f_{{\mathcal A}}^{-1}(x)$, and it follows 
from the Smale-Vietoris  theorem \cite{Smale} that

\begin{lemma}
\label{lem:hocolimit1}
The map $f_{{\mathcal A}}$ is a homotopy equivalence.
\end{lemma}

For $\ell \geq 0$, we will denote by $\hocolimit_{\leq \ell}({\mathcal A})$ the
subcomplex of $\hocolimit({\mathcal A})$ defined by
\begin{equation}
\label{def:truncatedhocolimit}
\hocolimit_{\leq \ell}({\mathcal A}) = 
\bigcupdot_{I \subset [n], \#I \leq \ell+2} 
\Delta_I \times A_I/\sim
\end{equation}

The following theorem is the key ingredient in the algorithm for computing 
Betti numbers of arrangements described in Section \ref{sec:arrangements}.

\begin{theorem}
\label{the:homologyoftruncated}
For $0 \leq j \leq \ell$ we have,
\[
\HH^j(\hocolimit_{\leq \ell}({\mathcal A})) \cong 
\HH^j(\hocolimit({\mathcal A})) \cong
\HH^j(A^{[n]}).
\]
\end{theorem}

\begin{proof}
By Lemma \ref{lem:hocolimit1} we have that

\begin{equation}
\label{eqn:inproofofhomologyoftruncated1}
\HH^j(\hocolimit({\mathcal A})) \cong \HH^j(A^{[n]}), \;\;j \geq 0.
\end{equation}

We also have by construction that 
the $(\ell+1)$-st skeletons of $\hocolimit_{\leq \ell}({\mathcal A})$ and
$\hocolimit({\mathcal A})$ coincide, which implies that

\begin{equation}
\label{eqn:inproofofhomologyoftruncated2}
\HH^j(\hocolimit_{\leq \ell}({\mathcal A})) \cong 
\HH^j(\hocolimit({\mathcal A})), \;\; 0 \leq j \leq \ell.
\end{equation}

The theorem now follows from (\ref{eqn:inproofofhomologyoftruncated1})
and (\ref{eqn:inproofofhomologyoftruncated2}) above.
\end{proof}

\section{Algorithms for Computing the First Few Betti Numbers}
\label{sec:bettifew}
We are now in a position to describe some of the new ideas that make
possible the design of algorithms with single exponential complexity
for computing the higher Betti numbers of semi-algebraic sets.

\subsection{Computing Covers by Contractible Sets}
One important idea  in the algorithm for computing the
first Betti number of semi-algebraic sets, is   
the construction of certain semi-algebraic sets called 
{\em parametrized paths}.
Under a certain hypothesis, these sets  are semi-algebraically contractible.
Moreover, there exists an algorithm for computing 
a covering of a given basic semi-algebraic set, $S \subset \R^k$,
by a single exponential number of parametrized paths. 

\subsubsection{Parametrized Paths}
\label{sec:pp}
We are given  a polynomial $Q \in \R[X_1,\ldots,X_k]$
such that
$\ZZ(Q,\R^k)$ is bounded and
a finite set of polynomials ${\mathcal P} \subset \D[X_1,\ldots,X_k]$.

The main technical construction underlying the algorithm for
computing the first Betti number in \cite{BPR9}, is to
to obtain a covering of a given 
$\mathcal P$-closed semi-algebraic set contained in $\ZZ(Q,\R^k)$
by a family of semi-algebraically contractible subsets.
This construction
is based on a parametrized version of the connecting algorithm:
we compute a family of polynomials such that for each realizable
sign condition $\sigma$ on this family, the description of the connecting
paths of different points in the realization, $\RR(\sigma,\ZZ(Q,\R^k)),$ are
uniform. 
We first define parametrized paths.
A parametrized path is a
semi-algebraic set which is a union of semi-algebraic
paths having the 
divergence property (see Section \ref{subsec:connectingpaths}).

More precisely,
\begin{definition}[Parametrized paths]
\label{def:parametrizedpath}
A parametrized path 
$\gamma$ is
 a continuous semi-algebraic mapping  from 
$V \subset \R^{k+1}
\rightarrow \R^k,$
such that,
denoting by 
$U=\pi_{1\ldots k}(V)\subset \R^k$,
there exists a semi-algebraic continuous function
$\ell: U \rightarrow [0,+\infty),$ 
and there exists  a point $a$ in $\R^k$, such that
\begin{enumerate}
\item $V =\{(x,t) \mid x \in U, 0 \le t \le \ell(x)\},$
\item $\forall \; x \in U, \; \gamma(x,0)=a$,
\item $\forall \; x \in U, \; \gamma(x,\ell(x))=x$,
\item 
$$
\displaylines{
\forall \; x \in U, \forall \; y \in U, \forall \; s \in [0,\ell(x)], \forall \; t \in [0,\ell(y)]\cr
 \left(\gamma(x,s)=\gamma(y,t) \Rightarrow s=t \right),
}
$$
\item  
$$
\displaylines{
\forall \; x \in U, \forall \; y \in U, \forall \; s \in [0,\min(\ell(x),\ell(y))]\cr
\left(\gamma(x,s)=\gamma(y,s) \Rightarrow \forall \; t \le s \; \gamma(x,t)=\gamma(y,t) \right).
}
$$
\end{enumerate}
\end{definition}

\begin{center}
\begin{figure}
\includegraphics[height=7cm]{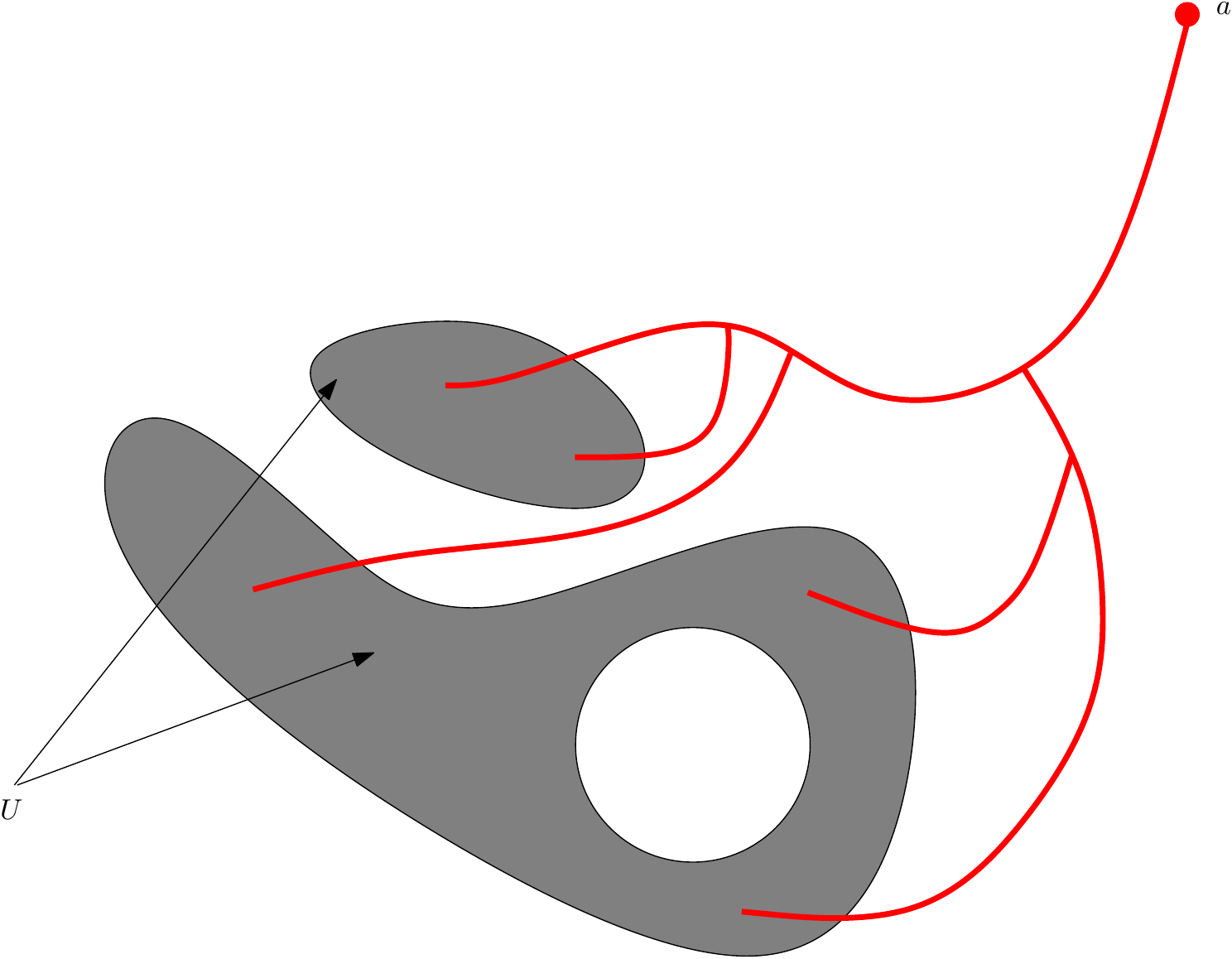}
\caption{A parametrized path}
\end{figure}
\end{center}

Given a parametrized path, $\gamma: V \rightarrow \R^k$, we will
refer to $U = \pi_{1\ldots k}(V)$ as its {\em base}.
Also,  any semi-algebraic subset 
$U' \subset U$ of the base of such a parametrized
path,
defines in a natural way  the restriction of 
$\gamma$ to the base $U'$, which is
another parametrized path,
obtained by restricting $\gamma$ to the set $V' \subset V$,
defined by 
$V' = \{(x,t) \mid x \in U', 0 \leq t \leq \ell(x) \}$.

The following proposition which appears in  \cite{BPR9} describes a crucial
property of parametrized paths, which makes them useful in algorithms
for computing Betti numbers of semi-algebraic sets.

\begin{proposition}\cite{BPR9}
\label{contractible}
Let $\gamma: V \rightarrow R^k$ be a parametrized path such that
$U = \pi_{1\ldots k}(V)$ is closed and bounded.
Then, the image of $\gamma$ is semi-algebraically contractible.
\end{proposition}

It is also shown in \cite{BPR9} that,
\begin{theorem}
\label{the:alg:parametrized}
Moreover, there exists an algorithm that takes as input
a finite set of  polynomials ${\mathcal P} \subset \R[X_1,\ldots,X_k]$,
and produces as output,
\begin{itemize}
\item
a finite set of polynomials ${\mathcal A} \subset \R[X_1,\ldots,X_k]$,
\item 
a finite set $\Theta$ of quantifier free formulas, 
with atoms of the form $P = 0, P > 0, P< 0, \;P \in {\mathcal A}$,
such that for every semi-algebraically connected component $S$ of
the realization of
every weak sign condition on ${\mathcal P}$ on $\ZZ(Q, \R^k)$,
there exists a subset $\Theta(S) \subset \Theta$
 such that
$
\displaystyle{
S=\bigcup_{\theta \in \Theta(S)} \RR(\theta,\ZZ(Q,\R^k)),
}
$
\item for every $\theta \in \Theta$,
a parametrized path
$$
\displaylines{
\gamma_\theta : V_\theta \rightarrow \R^k,
}
$$
with base $U_\theta =  \RR(\theta,\ZZ(Q,\R^k))$, 
such that for each $y \in \RR(\theta,\ZZ(Q,\R^k)),$ 
$\Im \;\gamma_\theta(y,\cdot)$  is a semi-algebraic path which
connects the point $y$
to a distinguished point  $a_\theta$
of some roadmap
$\RM(\ZZ({\mathcal P}' \cup \{Q\}, \R^k))$
where ${\mathcal P}'\subset {\mathcal P}$, staying inside
${\RR}(\overline\sigma(y),\ZZ(Q,\R^k)).$
\end{itemize}

Moreover, the complexity of the algorithm is  
$s^{k'+1} d^{O(k^4)}$, where
 $s$ is a bound on the number of elements of
${\mathcal  P}$
and $d$ is a bound on the degrees of $Q$ and
the elements of ${\mathcal  P}$.
\end{theorem}

\subsubsection{Constructing Coverings of Closed Semi-algebraic Sets
by Closed Contractible Sets}
\label{sec:acycliccov}
The parametrized paths obtained in Theorem \ref{the:alg:parametrized} are
not necessarily closed or even contractible, but become so after making
appropriate modifications. At the same time it is possible to maintain the
covering property, namely for any given ${\mathcal P}$-closed semi-algebraic
$S$ set, there exists a set of modified parametrized paths, whose union
is $S$. Moreover, these modified sets are closed and contractible.
We omit the details of this (technical) 
construction  referring the reader to \cite{BPR9} for more detail. 
Putting together the constructions outlined above we have:

\begin{theorem}
\label{the:alg:covering}
There exists an algorithm that given as input
a  ${\mathcal P}$-closed and bounded semi-algebraic set $S$, 
outputs
a set of formulas $\{\phi_1,\ldots,\phi_M\}$ such that
\begin{itemize}
\item 
each  $\RR(\phi_i,\R'^k)$ is 
semi-algebraically contractible, and
\item 
$\displaystyle{
\bigcup_{1 \leq i \leq M} 
\RR(\phi_i,\R'^k) = \E(S,\R'),
}
$
\end{itemize}
where $\R'$ is some real closed extension of $\R$. 
The complexity of the algorithm is bounded by $s^{(k+1)^2}d^{O(k^5)}$,
where $s = \#{\mathcal P}$ and $d = \max_{P \in {\mathcal P}} deg(P).$
\end{theorem}

\subsection
{Computing the First Betti Number}
\label{sec:closedcase}
It is now an easy consequence of the existence of single exponential
time covering algorithm (Theorem \ref{the:alg:covering}), 
and Theorem \ref{the:bettione} stated above, 
along with the fact that we can compute descriptions of the connected
components of semi-algebraic sets in single exponential time,
that we can compute the first Betti number of 
closed and bounded semi-algebraic sets in single exponential time
(see Remark \ref{rem:bettione} above), since the 
dimensions of the images and kernels of
the homomorphisms of the complex,
$\LL^{\bullet}_2({\mathcal C}(S))$ in  Theorem \ref{the:bettione},
can then be computed using traditional algorithms from linear algebra.
As mentioned earlier, 
for arbitrary semi-algebraic sets (not necessarily closed and
bounded),
there is a single exponential time reduction to the closed and
bounded case using the construction of Gabrielov and Vorobjov
\cite{BPR9,GV05}.

\subsection{Computing the Higher Betti Numbers}
\subsubsection{Double Complexes Associated to Certain Covers}
\label{sec:covering}
We now describe how covers by contractible sets of a closed and bounded 
semi-algebraic set, $S$,  can be used for computing the higher 
(than the first) Betti numbers of $S$. 
Recall (Remark \ref{rem:counterexample}) 
that it is no longer possible to use the nerve complex of
a (possibly non-Leray) cover by contractible sets for this purpose.

In this section, we consider a fixed family of polynomials,
${\mathcal P} \subset \R[X_1,\ldots,X_k],$ as well as a fixed  
${\mathcal P}$-closed and bounded semi-algebraic set, $S \subset \R^k$. 
We also fix a number, $\ell, 0 \leq \ell \leq k.$

We define below 
(see Definition \ref{def:admissible} below)
a finite set of indices, $\A_S$, which we call the set of {\em admissible
indices}, and a map that associates to each $\alpha \in \A_S$ a closed and
bounded semi-algebraic subset $X_\alpha \subset S$, which we call an
{\em admissible subset}. The reason behind having the set of indices 
is that 
in the construction of our complex the same set might occur with different
indices and we would like to distinguish these occurrences from each other.

To each $\alpha  \in \A_S$, 
we will associate its level, denoted $\level(\alpha)$, which is an integer
between $0$ and $\ell$. The set $\A_S$ will be partially ordered, and we
denote by $\ancestor(\alpha)\subset \A_S$, the set of ancestors
of $\alpha$ under this partial order. For $\alpha,\beta \in \A_S$,
$\beta \in \ancestor(\alpha)$ will imply that  
$X_\alpha \subset X_\beta$.

For each admissible index $\alpha \in \A_S$, 
we define  a double complex, $\M^{\bullet,\bullet}(\alpha)$, 
such that 
\begin{equation}
\label{eqn:definingM}
\HH^i(\Tot^{\bullet}(\M^{\bullet,\bullet}(\alpha))) \cong \HH^i(X_\alpha),\;
0 \leq i \leq \ell - \level(\alpha),
\end{equation}
and for each pair $\alpha,\beta \in \A_S$
with $\alpha \in \ancestor(\beta)$ a homomorphism,
\begin{equation}
r_{\alpha,\beta}^{\bullet,\bullet}:  \M^{\bullet,\bullet}(\alpha) 
\rightarrow \M^{\bullet,\bullet}(\beta),
\end{equation}
which induces the restriction homomorphisms between the cohomology
groups via the isomorphisms in 
$$
\displaylines{
r^*_{\alpha,\beta}: \HH^i(X_\alpha) \rightarrow \HH^i(X_\beta)
}
$$ 
for $0 \leq i \leq \ell - \level(\alpha)$  via the isomorphisms in 
(\ref{eqn:definingM}).
 
The main idea behind the construction of the double complex 
$\M^{\bullet,\bullet}(\alpha)$ is a recursive one. 
Associated to  any cover of
$X_\alpha$ there exists  a double complex (the Mayer-Vietoris double complex)
arising from the generalized Mayer-Vietoris exact sequence 
(see Section \ref{subsec:nonleray}).
If the individual sets of the cover of $X$ are all contractible, then the
first column 
of the $'E_1$-term of the corresponding spectral sequence
(cf. Eqn. (\ref{eqn:'E}))
is  zero except at the first row.
The cohomology groups of the associated total complex of the Mayer-Vietoris
double complex are isomorphic to those of $X_\alpha$ 
and thus in order to compute
$b_0(X_\alpha),\ldots,b_{\ell - \level(\alpha)}(X_\alpha),$
it suffices to compute a suitable
truncation of the Mayer-Vietoris double complex. Computing 
(even the truncated) Mayer-Vietoris double complex directly within a 
single exponential time complexity
is not possible by any known method, since we are unable 
to compute triangulations of semi-algebraic sets in single exponential time.
However, 
making use of the cover construction recursively, we are
able to compute another double complex, 
$\M^{\bullet,\bullet}(\alpha)$, 
which has much smaller size,
but whose associated spectral sequence, $'E_*$, 
is isomorphic 
to the one corresponding to the Mayer-Vietoris double complex. Hence, 
by Theorem \ref{the:spectral} (Comparison Theorem)
$\Tot^{\bullet}(\M^{\bullet,\bullet}(\alpha))$, is quasi-isomorphic to the
associated total complex of the 
Mayer-Vietoris double complex (see Proposition \ref{prop:main} below).
The construction of 
$\M^{\bullet,\bullet}(\alpha)$ is possible in single exponential time since the
covers can be computed in single exponential time.

Finally, given any closed and bounded semi-algebraic set $X \subset \R^k$,
we will denote by ${\mathcal C}'(X)$, a fixed cover of $X$
(we will assume that the construction implicit in 
Theorem \ref{the:alg:covering} provides such a cover).

We now define $\A_S$,
and for each $\alpha \in \A_S$ a cover
${\mathcal C}(\alpha)$ of $X_\alpha$ 
obtained by enlarging the cover ${\mathcal C}'(X_\alpha)$. 

\begin{definition}[Admissible indices and covers]
\label{def:admissible}
$\A_S$ is defined by induction on level.
\begin{enumerate}
\item
Firstly, 
$0 \in \A_S$, $\level(0) = 0,$ $X_0 = S$, $\ancestor(0) = \emptyset$,
and ${\mathcal C}(0) = {\mathcal C}'(S)$. 

\item 
The admissible indices  at level ${i+1}$ are now inductively defined in terms
of the admissible indices at level $\leq i$.

The set of admissible indices at level $i+1$ is
\begin{equation}
\bigcupdot_{\alpha \in \A_S, \level(\alpha) = i}
\{\alpha_0\cdot\alpha_1\cdots\alpha_j \mid 
\alpha_i \in {\mathcal C}(\alpha), 0 \leq j \leq \ell-i+1\},
\end{equation}
where $\bigcupdot$ denotes the disjoint union,
and for each $\beta = \alpha_0\cdot\alpha_1\cdots\alpha_j$
we set $X_\beta = X_{\alpha_0} \cap \cdots \cap X_{\alpha_j}$.

We now enlarge the set of ancestor relations by adding:
\begin{enumerate}
\item
For each 
$\{\alpha_0,\ldots,\alpha_m\} \subset \{\beta_0,
\ldots,\beta_n \} \subset {\mathcal C}(\alpha)$,  
with $n \leq \ell-i+1$, 
$\alpha_0\cdots\alpha_m \in  \ancestor(\beta_0\cdots\beta_n),
$ 
and $\alpha \in \ancestor(\beta_0\cdots\beta_n)$.
\item
Moreover, if 
$\alpha_1\cdots\alpha_m, \beta_0\cdots\beta_n \in \A_S$ are such that
for every  $j \in \{0,\ldots, n\}$ there exists $i \in \{0,\ldots, m\}$ 
such that
$\alpha_i$ is an ancestor of $\beta_j$, then $\alpha_0\cdots\alpha_m$ is
an ancestor of $\beta_0\cdots\beta_n$.
\item
The ancestor relation is transitively closed, so that ancestor of 
an ancestor is also an ancestor.
\end{enumerate}

Finally, for each $\alpha \in \A_S$ at level 
$i+1$, 
we define
${\mathcal C}(\alpha)$ as follows. 
Let $\ancestor(\alpha) = \{\alpha_1,\ldots,\alpha_N\}$. 
\begin{equation}
{\mathcal C}(\alpha) = 
\bigcupdot
{\mathcal C}'(\beta_1  \cdots \beta_N \cdot \alpha).
\end{equation}
where the disjoint union is taken over all tuples
$(\beta_1,\ldots,\beta_N)$ satisfying for each $1 \leq i, j \leq N$,
$\beta_i \in {\mathcal C}(\alpha_i)$, and 
if $\alpha_i \in \ancestor(\alpha_j)$ then 
$\beta_i \in \ancestor(\beta_j)$. 
\end{enumerate}
\end{definition}

Observe that by the above definition, if $\alpha,\beta \in \A_S$
and $\beta \in \ancestor(\alpha)$,
then each $\alpha' \in {\mathcal C}(\alpha)$ has a unique ancestor 
in  each  ${\mathcal C}(\beta)$, which we will denote by
$a_{\alpha,\beta}(\alpha')$.

The mappings $a_{\alpha,\beta}$
has the property that if $\beta \in \ancestor(\alpha)$ and
$\gamma \in \ancestor(\beta)$, then 
$a_{\alpha,\gamma} = a_{\beta,\gamma}\circ a_{\alpha,\beta}$.

Now, suppose that there is a procedure for computing  ${\mathcal C}'(X)$, 
for any given ${\mathcal P}'$-closed and bounded semi-algebraic set, $X$,
such that
the number and the degrees of the polynomials appearing the descriptions
of the semi-algebraic sets, $X_\alpha, \alpha \in {\mathcal C}'(X)$,
is bounded by 
\begin{equation}
\label{eqn:c_1}
D^{k^{c_1}},
\end{equation}
where $c_1 > 0$ is some absolute constant, and 
$D = \sum_{P \in {\mathcal P}'} \deg(P).$
  
Using  the above procedure for computing ${\mathcal C}'(X)$, and
the definition of $\A_S$,  we have the following
quantitative bounds on $\#\A_S$  and the 
 semi-algebraic sets $X_\alpha, \alpha \in \A_S$,
which is crucial in proving the single exponential 
complexity bound of the algorithm for computing the first few
Betti numbers of semi-algebraic sets.

\begin{proposition} \cite{Basu7}
\label{prop:bound}
Let $S \subset \R^k$ be a ${\mathcal P}$-closed semi-algebraic set, where
${\mathcal P} \subset \R[X_1,\ldots,X_k]$ is a family of $s$ polynomials of 
degree at most $d$. Then $\#A_S$, as well as 
the number of polynomials used to define the semi-algebraic sets 
$X_\alpha, \alpha \in \A_S$ and the  
the degrees of these polynomials, are all bounded by
$(sd)^{k^{O(\ell)}}.$
\end{proposition}

\subsubsection{Double Complex Associated to a Cover}
\label{subsec:doublecover}
Given the different covers described above, 
we now associate to each 
$ \alpha \in \A_S$  
a double complex,
${\mathcal M}^{\bullet,\bullet}(\alpha),$ 
and for  every $\beta \in \A_S$, such that $\alpha \in \ancestor(\beta)$,
and $\level(\alpha) = \level(\beta)$,
a restriction homomorphism:
\[
r_{\alpha,\beta}^{\bullet,\bullet}:  \M^{\bullet,\bullet}(\alpha) 
\rightarrow \M^{\bullet,\bullet}(\beta),
\]
satisfying the following:

\begin{enumerate}
\item
\begin{equation}
\label{eqn:iso}
\HH^i(\Tot^{\bullet}(\M^{\bullet,\bullet}(\alpha)))  \cong 
\HH^i(X_\alpha), \;\mbox{\rm for} \;
0 \leq i \leq \ell - \level(\alpha).
\end{equation}
\item
The restriction homomorphism
\[
r_{\alpha,\beta}^{\bullet,\bullet}:  \M^{\bullet,\bullet}(\alpha) 
\rightarrow \M^{\bullet,\bullet}(\beta),
\]
induces the restriction homomorphisms between the cohomology
groups:
$$
\displaylines{
r^*_{\alpha,\beta}: \HH^i(X_\alpha) \rightarrow \HH^i(X_\beta)
}
$$ 
for $0 \leq i \leq \ell - \level(\alpha)$  via the isomorphisms in (\ref{eqn:iso}).
\end{enumerate}

We now define the double complex $\M^{\bullet,\bullet}(\alpha)$.
The 
double complex $\M^{\bullet,\bullet}(\alpha)$ 
is constructed inductively using induction on $\level(\alpha)$.

\begin{definition}
\label{def:double}
The base case is when 
$\level(\alpha) = \ell$.
In this case the double complex, $\M^{\bullet,\bullet}(\alpha)$
is defined by:
\[
\begin{array}{ccll}
\M^{0,0}(\alpha) & = & \bigoplus_{{\alpha_0}\; \in\; {\mathcal C}(\alpha)}\; \HH^0(X_{\alpha_0}),
\cr
\M^{1,0}(\alpha) & = & \bigoplus_{{\alpha_0}, {\alpha_1} \;\in\; {\mathcal C}(\alpha)}\; \HH^0(X_{\alpha_0\cdot\alpha_1}), \cr
\M^{p,q}(\alpha) & = &  0, \; \mbox{if} \; q > 0 \;\mbox{or} \; p > 1 . \cr
\end{array}
\]

This is shown diagrammatically below.

\[
{\tiny
\begin{diagram}
\node{0}\arrow{e}\node{0}\arrow{e}\node{0} \\
\node{0}\arrow{e}\arrow{n}\node{0}\arrow{e}\arrow{n}\node{0}\arrow{n}\\
\node{\bigoplus_{{\alpha_0} \in {\mathcal C}(\alpha)} 
\HH^0(X_{\alpha_0})}\arrow{e,t}{\delta}\arrow{n}
\node{\bigoplus_{{\alpha_0}, {\alpha_1} \in {\mathcal C}(\alpha)} 
\HH^0(X_{\alpha_0\cdot \alpha_1})}
\arrow{e}\arrow{n}
\node{0}\arrow{n}
\end{diagram}
}
\]

The only non-trivial homomorphism in the above complex, 
$$
\displaylines{
\delta:\bigoplus_{{\alpha_0} \in {\mathcal C}(\alpha)} 
\HH^0(X_{\alpha_0}) 
\longrightarrow
\bigoplus_{{\alpha_0},{\alpha_1} \in {\mathcal C}(\alpha)} 
\HH^0(X_{\alpha_0\cdot \alpha_1}) 
}
$$
is defined as follows.

$\delta(\phi)_{\alpha_0,\alpha_1} = 
(\phi_{\alpha_1} - \phi_{\alpha_0})|_{X_{\alpha_0\cdot\alpha_1}}$ for
$\phi \in \bigoplus_{{\alpha_0} \in {\mathcal C}(\alpha)} \HH^0(X_{\alpha_0}).$

For every $\beta \in \A_S$, such that $\alpha \in \ancestor(\beta)$,
and $\level(\alpha) = \level(\beta) = \ell$,
we define 
$r_{\alpha,\beta}^{0,0}: \M^{0,0}(\alpha) \rightarrow \M^{0,0}(\beta)$,
as follows.

Recall that,
$\displaystyle{
\M^{0,0}(\alpha) = \bigoplus_{\alpha_0 \;\in\; {\mathcal C}(\alpha)} 
\; \HH^0(X_{\alpha_0}),
}
$
and 
$
\displaystyle{
\M^{0,0}(\beta) = \bigoplus_{\beta_0 \;\in\; {\mathcal C}(\beta)} \; \HH^0(X_{\beta_0}).
}
$

For $\phi \in \M^{0,0}(\alpha)$ and $\beta_0 \in {\mathcal C}(\beta)$ 
we define
$$
\displaylines{
r_{\alpha,\beta}^{0,0}(\phi)_{\beta_0} = \phi_{a_{\beta,\alpha}(\beta_0)}
|_{X_{\beta_0}}.
}
$$

We define
$r_{\alpha,\beta}^{1,0}: \M^{1,0}(\alpha) \rightarrow \M^{1,0}(\beta),$
in a similar manner. More precisely,
for $\phi \in \M^{0,0}(\alpha)$ and
$\beta_0,\beta_1 \in {\mathcal C}(\beta)$,
we define
$$
\displaylines{
r_{\alpha,\beta}^{1,0}(\phi)_{\beta_0,\beta_1} = 
\phi_{a_{\beta,\alpha}(\beta_0)\cdot a_{\beta,\alpha}(\beta_1)}|_{X_{\beta_0\cdot\beta_1}}.
}
$$ 

(The inductive step)
In general the $\M^{p,q}(\alpha)$ are defined as follows using induction
on $\level(\alpha)$ and with $n_\alpha= \ell - \level(\alpha)+1.$
$$
\displaylines{
\begin{array}{ll}
\M^{0,0}(\alpha) =  \bigoplus_{{\alpha_0} \;\in\; {\mathcal C}(\alpha)} 
\;\HH^0(X_{\alpha_0}),& \cr
\M^{0,q}(\alpha) =   0, & 0< q, \cr
\M^{p,q}(\alpha) =   \bigoplus_{\alpha_0< \cdots <\alpha_p, \;
{\alpha_i}  \in {\mathcal C}(\alpha)}
\; \Tot^q(\M^{\bullet,\bullet}({\alpha_0 \cdots \alpha_p})),  & 0 < p,\;
0< p+q \leq n_\alpha,  \cr
\M^{p,q}(\alpha) =   0, &\;\mbox{else}.
\end{array}
}
$$

The double complex $\M^{\bullet,\bullet}(\alpha)$ is shown  in the following 
diagram:

\[
\divide\dgARROWLENGTH by8
{\tiny
\begin{diagram}
\node{0}\arrow{e}\node{0}
\node{\cdots}\node{0}\\
\node{0}\arrow{e}\arrow{n}
\node{
\bigoplus_{\alpha_0 < \alpha_1}\Tot^{n_\alpha-1}(\M^{\bullet,\bullet}
({\alpha_0\cdot \alpha_1}))}
\arrow{n}
\node{\cdots}
\node{0}\arrow{n}\\
\node{0}\arrow{e}\arrow{n,l}{d}
\node{
\bigoplus_{\alpha_0 < \alpha_1}\Tot^{n_\alpha-2}(\M^{\bullet,\bullet}({\alpha_0\cdot\alpha_1}))}
\arrow{n,l}{d}
\node{\cdots}\node{0}\arrow{n,l}{d} \\
\node{\vdots}\node{\vdots}\node{\vdots}
\node{\vdots}\\
\node{0}\arrow{e}
\node{
\bigoplus_{\alpha_0 < \alpha_1}\Tot^2(\M^{\bullet,\bullet}({\alpha_0\cdot\alpha_1}))}
\node{\cdots}
\node{0} \\
\node{0}\arrow{e}\arrow{n}
\node{\bigoplus_{\alpha_0 < \alpha_1}\Tot^1(\M^{\bullet,\bullet}({\alpha_0\cdot\alpha_1}))} 
\arrow{n,l}{d}
\node{\cdots}
\node{0}\arrow{n,l}{d} \\
\node{\bigoplus_{\alpha_0 \in {\mathcal C}(X)} \HH^0(S_{\alpha_0})} 
\arrow{e,t}{\delta}\arrow{n,l}{d} 
\node{
\bigoplus_{\alpha_0 < \alpha_1}\Tot^0(\M^{\bullet,\bullet}({\alpha_0\cdot\alpha_1}))} 
\arrow{n,l}{d}
\node{\cdots}
\node{
\bigoplus_{\alpha_0< \cdots <\alpha_{n_\alpha}}\Tot^0(\M^{\bullet,\bullet}
({\alpha_0\cdots\alpha_{n_\alpha}}))}
\arrow{n,l}{d}
\end{diagram}
}
\]

The vertical homomorphisms, $d$, in $\M^{\bullet,\bullet}(\alpha)$ are those 
induced by the differentials in the various 
$$
\Tot^{\bullet}(\M^{\bullet,\bullet}({\alpha_0\cdots\alpha_p})),
{\alpha_i} \in {\mathcal C}(\alpha).
$$

The horizontal ones are defined by generalized restriction as follows.
Let 
\[\phi \in \bigoplus_{\alpha_0< \cdots <\alpha_p, \alpha_i \in 
{\mathcal C}(\alpha)} 
\Tot^q(\M^{\bullet,\bullet}({\alpha_0\cdots\alpha_p})),
\]

with 
\[
\phi_{\alpha_0,\ldots,\alpha_p} = \bigoplus_{0 \leq j \leq q}
\phi_{\alpha_0,\ldots,\alpha_p}^j,
\]
and 
\[
\phi_{\alpha_0,\ldots,\alpha_p}^j \in \M^{j,q-j}({\alpha_0\cdots\alpha_p}).
\]

We define
$$\displaylines{
\delta:
\bigoplus_{\alpha_0< \cdots <\alpha_p, \alpha_i \in {\mathcal C}(\alpha)} 
\Tot^q(\M^{\bullet,\bullet}({\alpha_0\cdots\alpha_p})) \longrightarrow
\bigoplus_{\alpha_0< \cdots <\alpha_{p+1}}
\Tot^q(\M^{\bullet,\bullet}({\alpha_0\cdots\alpha_{p+1}}))
}
$$
by
$$
\displaylines{
\delta(\phi)_{\alpha_0,\ldots, \alpha_{p+1}} = 
\bigoplus_{0 \leq i \leq p+1} (-1)^i 
\bigoplus_{0 \leq j \leq q}
r_{{\alpha_0\cdots\hat{\alpha_i}
\cdots\alpha_{p+1}},{\alpha_0\cdots\alpha_{p+1}}}^{j,q-j}
(\phi_{\alpha_0,\ldots,\hat{\alpha_i},
\ldots,\alpha_{p+1}}^j),
}
$$
noting that for each $i, 0 \leq i \leq p+1,$ 
${\alpha_0\cdots\hat{\alpha_i}\cdots\alpha_{p+1}}$
is an ancestor of ${\alpha_0\cdots\alpha_{p+1}}$, and

\[
\level({\alpha_0\cdots\hat{\alpha_i}\cdots\alpha_{p+1}})=
\level({\alpha_0\cdots\alpha_{p+1}}) = \level(\alpha)+ 1,
\]
and hence the homomorphisms 
$r_{{\alpha_0\cdots\hat{\alpha_i}
\cdots\alpha_{p+1}},{\alpha_0\cdots\alpha_{p+1}}}^{j,q-j}$
are already defined by induction.

Now let,
$\alpha,\beta \in \A_S$ 
with $\alpha$ an ancestor of $\beta$ and
$\level(\alpha) = \level(\beta).$
We define the restriction homomorphism, 
$$
r_{\alpha,\beta}^{\bullet,\bullet}:  
\M^{\bullet,\bullet}(\alpha) \longrightarrow \M^{\bullet,\bullet}(\beta)
$$
as follows.

As before,
for $\phi \in \M^{0,0}(\alpha)$ and 
$\beta_0 \in {\mathcal C}(\beta)$ we define
$$
\displaylines{
r_{\alpha,\beta}^{0,0}(\phi)_{\beta_0} = \phi_{a_{\beta,\alpha}(\beta_0)}|_{X_{\beta_0}}.
}
$$
For $0 < p, 0< p+q \leq \ell-\level(\alpha)+1,$
we define
\[
r_{\alpha,\beta}^{p,q}: \M^{p,q}(\alpha) \rightarrow \M^{p,q}(\beta),
\]
as follows.

Let 
$
\displaystyle{
\phi \in \M^{p,q}(\alpha) = \bigoplus_{\alpha_0< \cdots <\alpha_p, \;
{\alpha_i}  \in {\mathcal C}(\alpha)}
\; \Tot^q(\M^{\bullet,\bullet}({\alpha_0\cdots\alpha_p})).
}
$ 
We define
\[
r_{\alpha,\beta}^{p,q}(\phi) =
\bigoplus_{\beta_0< \cdots <\beta_p, \beta_i \in {\mathcal C}(\beta)} 
\bigoplus_{0 \leq i \leq  q}r_{a_{\beta,\alpha}({\beta_0\cdots\beta_p}),
{\beta_0\cdots\beta_p}}^{i,q-i} \phi_{a_{\beta,\alpha}
({\beta_0}),\ldots, a_{\beta,\alpha}({\beta_p})}^i,
\]
where 
$a_{\beta,\alpha}({\beta_0\cdots\beta_p}) = 
a_{\beta,\alpha}({\beta_0})\cdots a_{\beta,\alpha}({\beta_p}).$
Moreover,
$$
\level(a_{\beta,\alpha}({\beta_0\cdots\beta_p}))= 
\level({\beta_0\cdots\beta_p})
= \level(\alpha) +1,
$$
and hence we can assume that the homomorphisms
$r_{a_{\beta,\alpha}({\beta_0\cdots\beta_p}),{\beta_0\cdots\beta_p}}^
{\bullet,\bullet}$ 
used in the definition of
$r_{\alpha,\beta}^{\bullet,\bullet}$ are already defined by induction.
\end{definition}

The following proposition proved in \cite{Basu7} is the 
key ingredients in the single exponential time algorithm for
computing the first few Betti numbers of semi-algebraic sets.

\begin{proposition}\cite{Basu7}
\label{prop:main}
For each $\alpha\in \A_S$ 
the double complex $\M^{\bullet,\bullet}(\alpha)$ satisfies the following 
properties:

\begin{enumerate}
\item
$
\HH^i(\Tot^{\bullet}(\M^{\bullet,\bullet}(\alpha)))  \cong 
\HH^i(X_\alpha)$ 
for $0 \leq i \leq \ell - \level(\alpha)$.
\item
For every $\beta \in \A_S$, such that $\alpha$ is an ancestor of $\beta$,
and $\level(\alpha) = \level(\beta)$,
the homomorphism,
$ r_{\alpha,\beta}^{\bullet,\bullet}:  
\M^{\bullet,\bullet}(\alpha) \rightarrow \M^{\bullet,\bullet}(\beta),$
induces the  restriction homomorphisms between the cohomology
groups:
$$
\displaylines{
r^*: \HH^i(X_\alpha) \longrightarrow \HH^i(X_\beta)
}
$$ 
for $0 \leq i \leq \ell - \level(\alpha)$  via the isomorphisms in (1).
\end{enumerate}
\end{proposition}

\begin{proof}[Proof Sketch]
For the benefit of the reader we include an outline of the
proof of Proposition \ref{prop:main}
referring the reader to \cite{Basu7} for more detail.
The proof is by induction on ${\rm level}(\alpha)$.
After having chosen a suitably fine triangulation $\Delta$ 
of $S$ which respects all
the admissible subsets $X_\alpha$, 
we construct by induction on ${\rm level}(\alpha)$,
for each $\alpha \in \A_S$, a double complex $\D^{\bullet,\bullet}(\alpha)$
and homomorphisms
\[
\phi^{\bullet,\bullet}: \M^{\bullet,\bullet}(\alpha) \rightarrow \D^{\bullet,\bullet}(\alpha),
\]
and 
\[
\psi^{\bullet}: \C^{\bullet}(\Delta_\alpha) 
\rightarrow \Tot^\bullet(\D^{\bullet,\bullet}(\alpha)),
\]
where $\Delta_\alpha$ denotes the restriction of the triangulation
$\Delta$ to $X_\alpha$. 
It is then shown (inductively) that
each homomorphism in the following diagram is a quasi-isomorphism.

\begin{equation}
\begin{diagram}
\node{}\node{\Tot^\bullet(\D^{\bullet,\bullet}(\alpha))}\node{} \\
\node{\Tot^{\bullet}(\M^{\bullet,\bullet}(\alpha))}\arrow{ne,t}{\Tot^{\bullet}(
\phi^{\bullet,\bullet})}
\node{}
\node{\C^{\bullet}(\Delta_\alpha)}
\arrow{nw,t}{\psi^{\bullet}}
\end{diagram}
\end{equation}
This suffices to prove the proposition.
\end{proof}

We now give an example of the construction of the 
complex described above in a very simple situation.

\begin{example}
We take for the set $S$, the unit sphere
$\Sphere^2 \subset \R^3$.
Even though this example looks very simple, it is actually illustrative
of the main topological ideas behind the construction of the complex
$\M^{\bullet,\bullet}(S)$ starting from a cover of $S$ by two
closed hemispheres meeting at the equator. Since the intersection of
the two hemisphere is a topological circle which is not 
contractible, Theorem \ref{the:nerve} is not applicable.
Using Theorem  \ref{the:bettione} we can compute $\HH^0(S),\HH^1(S)$, but it
is not enough to compute $\HH^2(S)$.  
The recursive
construction of ${\mathcal M}^{\bullet,\bullet}$ described in the last section
overcomes this problem and this is illustrated in the example.

\hide{
\begin{figure}[hbt] 
\begin{center}
\includegraphics[width=5cm]{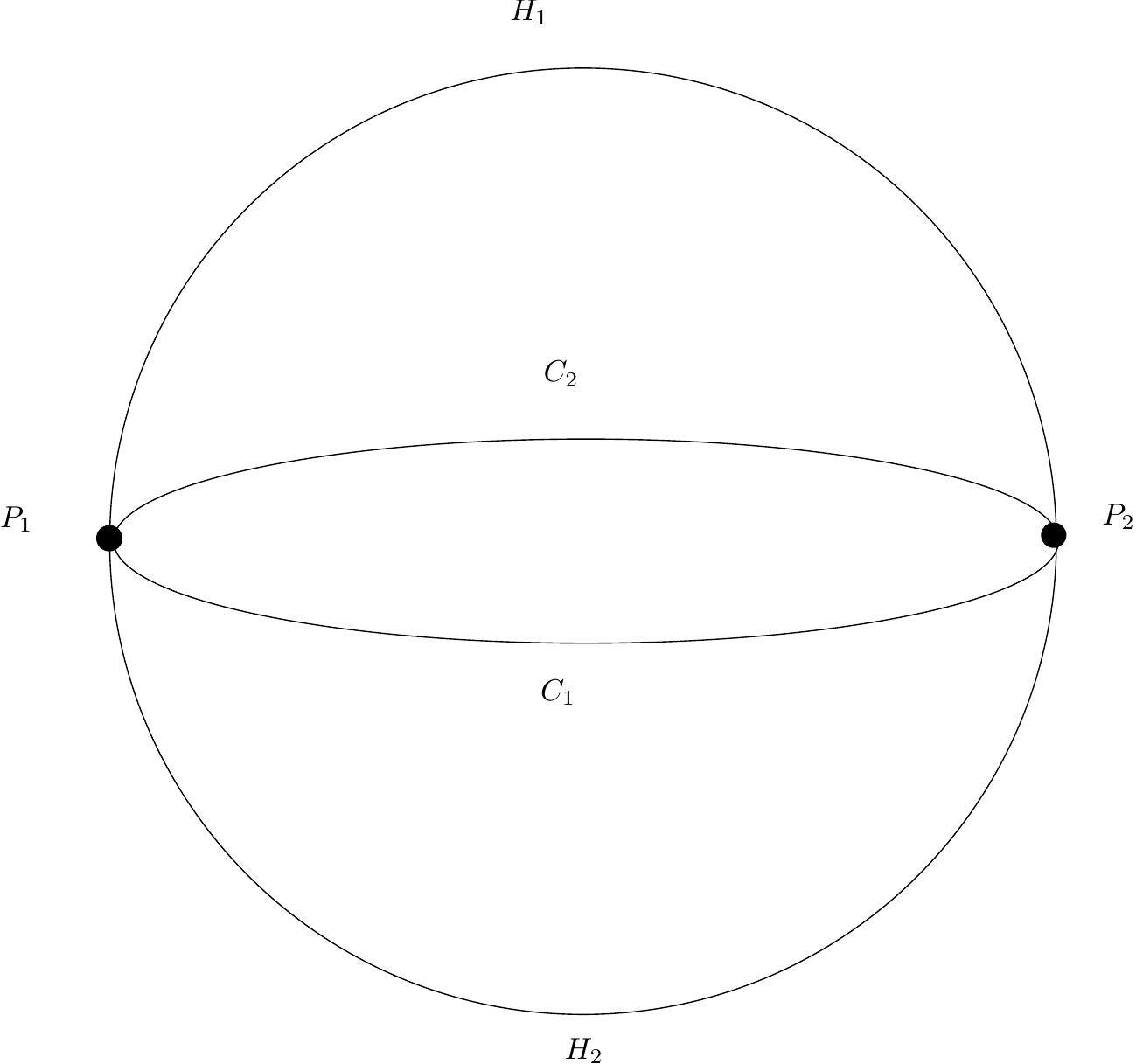}%
\end{center}
\caption{Example of $\Sphere^2 \subset \R^3$}
\label{fig:example}
\end{figure}
}

       \begin{figure}[hbt]
         \centerline{
           \scalebox{0.5}{
 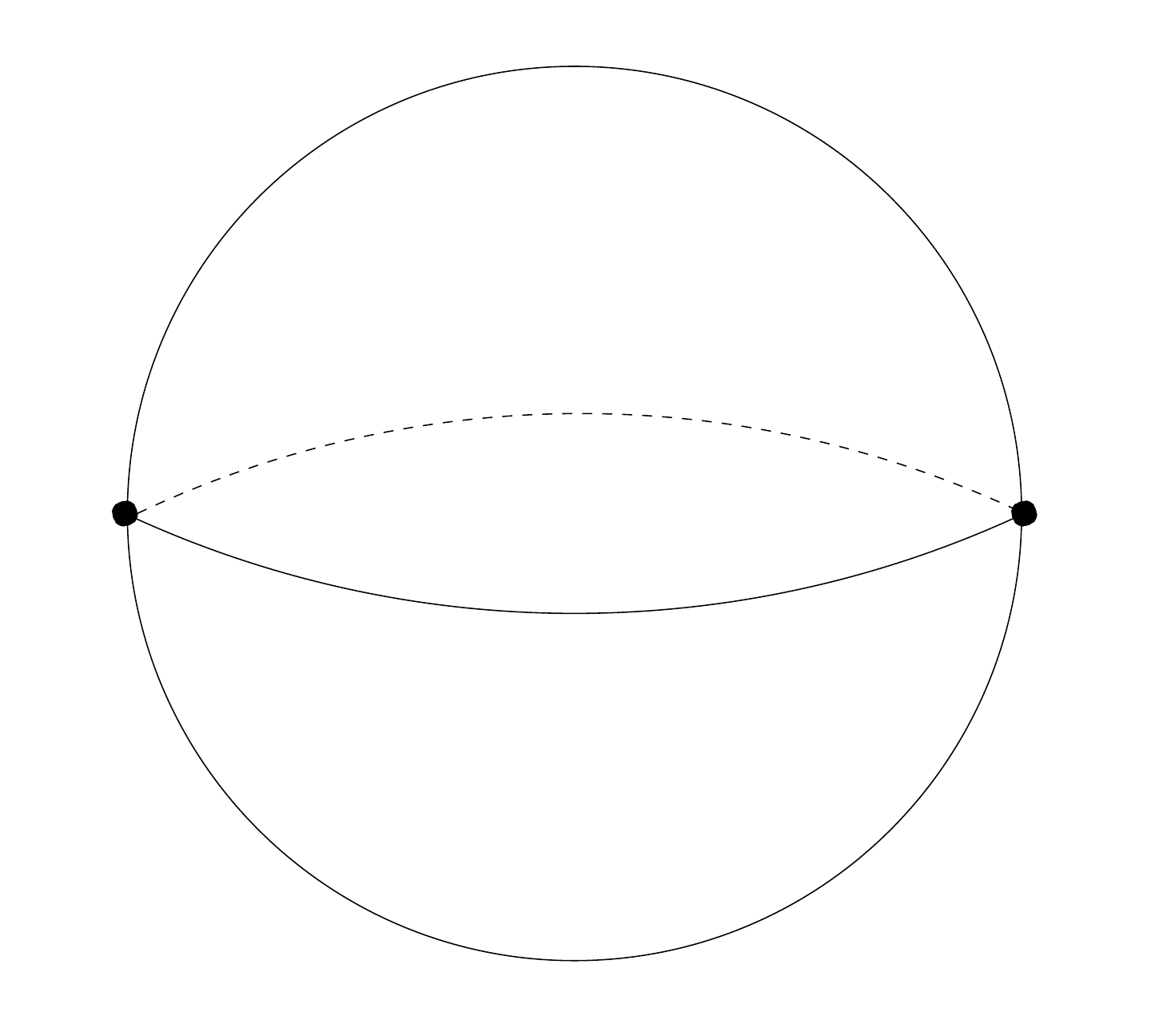
             }
           }
\caption{Example of $\Sphere^2 \subset \R^3$}
\label{fig:example}
       \end{figure}

We first fix some notation (see Figure \ref{fig:example}).
Let $H_1$ and $H_2$ denote the
closed upper and lower hemispheres respectively. Let $H_{12} = H_1 \cap H_2$
denote the equator, 
and let $H_{12} = C_1 \cup C_2$, where $C_1,C_2$ are closed
semi-circular arcs. Finally, let $C_{12} = C_1 \cap C_2 = \{P_1,P_2\}$, where 
$P_1,P_2$ are two antipodal points.

For the purpose of this example, we will take for the 
covers ${\mathcal C'}$ the obvious ones, namely:
$$
\begin{array}{cccc}
{\mathcal C}'(S) &=& \{H_1,H_2\},& \\
{\mathcal C}'(H_i) &=& \{H_i\}, &i =1,2, \\
{\mathcal C}'(H_{12}) &=& \{C_1,C_2\},& \\
{\mathcal C}'(C_i) &=& \{C_i\},& i =1,2, \\
{\mathcal C}'(C_{12}) &=& \{P_1,P_2\}, \\
{\mathcal C}'(P_i) &=&\{P_i\},& i =1,2. \\
\end{array}
$$
Note that, in order not to complicate notation further, 
we are using the same names
for the elements of ${\mathcal C}'(\cdot)$, as well as their
associated sets. 
Strictly speaking, we should have defined,
\[
{\mathcal C}'(S) =  \{\alpha_1,\alpha_2\}, \; X_{\alpha_1} = H_1, X_{\alpha_2}
 = H_2, \ldots .\]
However, since each set occurs at most once, this
does not create confusion in this example.

Note that the elements of the sets occurring on the right are all closed, 
bounded contractible subsets of $S$.
It is now easy to check from Definition \ref{def:admissible},
that the elements of $\A_S$ in order of their levels
as follows.
\begin{enumerate}
\item Level $0$:
\[
0 \in \A_S, \level(0) = 0,\]
and
$$
\begin{array}{cccc}
{\mathcal C}(0) &=& \{\alpha_1,\alpha_2\}, & X_{\alpha_1} = H_1, X_{\alpha_2} = H_2.\\
\end{array}
$$
 
\item Level $1$:
The elements of level $1$ are 
\[
\alpha_1,\alpha_2,\alpha_1\cdot\alpha_2,
\]
and
$$
\begin{array}{cccc}
{\mathcal C}(\alpha_1) &=& \{\beta_1\},  & X_{\beta_1} = H_1,\\  
{\mathcal C}(\alpha_2) &=& \{\beta_2\},  & X_{\beta_2} = H_2,\\  
{\mathcal C}(\alpha_1 \cdot \alpha_2) &=& \{\beta_3,\beta_4\}, & X_{\beta_3} = C_1, 
X_{\beta_4} = C_2. \\  
\end{array}
$$

\item Level $2$:
The elements of level $2$ are 
$\beta_1,\beta_2, \beta_3,\beta_4,\beta_3\cdot\beta_4$. We also have,
$$
\begin{array}{ccccc}
{\mathcal C}(\beta_i) &=& \{\gamma_i\}, & X_{\gamma_i} = H_i, & i = 1,2,\\  
{\mathcal C}(\beta_i) &=& \{\gamma_i\}, & X_{\gamma_i} = C_{i-2},   &  i = 3,4,\\
{\mathcal C}(\beta_{3}\cdot\beta_{4}) &=& \{\gamma_5,\gamma_6\},& 
X_{\gamma_i} = P_{i-4},  & i = 5,6. \\  
\end{array}
$$
\end{enumerate}

We now display diagrammatically the various complexes,
$\M^{\bullet,\bullet}(\alpha)$ for $\alpha \in \A_S$ starting 
at level $2$.

\begin{enumerate}
\item Level $2$:
For $1 \leq i \leq 4$, we have
\begin{eqnarray*}
\M^{\bullet,\bullet}(\beta_i)  &=&
{\tiny
\begin{diagram}
\node{0}\arrow{e}\node{0}\\
\node{\HH^0(X_{\gamma_i})}\arrow{e}\arrow{n}\node{0}\arrow{n}
\end{diagram}
}
\end{eqnarray*}
Notice that for $1 \leq i \leq 4$, 
\[
\HH^0(\Tot^{\bullet}(\M^{\bullet,\bullet}(\beta_i))) 
\cong \HH^0(X_{\beta_i}) \cong \Q.
\]

The complex $\M^{\bullet,\bullet}(\beta_3\cdot\beta_4)$ is 
shown below.
{\tiny
\[
\begin{diagram}
\node{0}\arrow{e}\node{0}\\
\node{\HH^0(P_1)\bigoplus \HH^0(P_2)}\arrow{e}\arrow{n}\node{0}
\arrow{n}
\end{diagram}
\]
}
Notice that, 
\[
\HH^0(\Tot^{\bullet}(\M^{\bullet,\bullet}(\beta_3\cdot\beta_4))) \cong \HH^0(X_{\beta_3\cdot\beta_4}) \cong \Q\oplus \Q.
\]

\item Level $1$:
For $i=1,2$, the complex $\M^{\bullet,\bullet}(\alpha_i)$ is as follows.
{\tiny
\[
\begin{diagram}
\node{0}\arrow{e}\node{0}\arrow{e}\node{0}\\
\node{0}\arrow{e}\arrow{n}\node{0}\arrow{e}\arrow{n}\node{0}\arrow{n}
\\
\node{\HH^0(H_i)}\arrow{e}\arrow{n}\node{0}\arrow{e}\arrow{n}\node{0}
\arrow{n}
\end{diagram}
\]
}
Notice that for $i=1,2$ and $j =0,1$, 
\[
\HH^j(\Tot^{\bullet}(\M^{\bullet,\bullet}(\alpha_i))) \cong 
\HH^j(H_i).
\] 

The complex $\M^{\bullet,\bullet}(\alpha_1\cdot\alpha_2)$is shown below.
{\tiny
\[
\begin{diagram}
\node{0}\arrow{e}\node{0}\arrow{e}\node{0}\\
\node{0}\arrow{e}\arrow{n}\node{0}\arrow{e}\arrow{n}\node{0}
\arrow{n}\\
\node{\HH^0(C_1)\bigoplus\HH^0(C_2)}
\arrow{e}\arrow{n}\node{\HH^0(P_1)\bigoplus\HH^0(P_2)}
\arrow{e}\arrow{n}\node{0}\arrow{n}
\end{diagram}
\]
}
Notice that for $j =0,1$, 
\[
\HH^j(\Tot^{\bullet}(\M^{\bullet,\bullet}(\alpha_1\cdot\alpha_2))) \cong  \HH^j(H_{12}).
\]

\item Level $0$:

The complex $\M^{\bullet,\bullet}(0)$ is shown below:

{\tiny
\[
\begin{diagram}
\node{0}\arrow{e}
\node{0}\arrow{e}
\node{0}\arrow{e}
\node{0}\\
\node{0}\arrow{e}\arrow{n}
\node{0}\arrow{e}\arrow{n}
\node{0}\arrow{e}\arrow{n}
\node{0}\arrow{n}\\
\node{0}\arrow{e}\arrow{n} 
\node{\HH^0(P_1)\bigoplus\HH^0(P_2)}\arrow{e}\arrow{n}
\node{0}\arrow{e}\arrow{n}
\node{0}\arrow{n} \\
\node{\HH^0(H_1)\bigoplus\HH^0(H_2)}
\arrow{e,t}{\delta^{0,0}}\arrow{n}
\node{\HH^0(C_1)\bigoplus\HH^0(C_2)}\arrow{e}\arrow{n,l}{d^{1,0}}
\node{0}\arrow{e}\arrow{n}
\node{0}\arrow{n}
\end{diagram}
\]
}

The matrices for the homomorphisms, $\delta^{0,0}$ and $d^{1,0}$ in the
obvious bases are both equal to
$$
\left(
\begin{array}{cc}
1\;\; &\;\; 1 \\
1\;\; &\;\; 1
\end{array}
\right).
$$

From the fact that the rank of the above matrix is $1$, 
it is not too difficult to deduce that, 
$\HH^j(\Tot^{\bullet}(\M^{\bullet,\bullet}(0))) \cong \HH^j(S)$,
for $j =0,1,2$, that is
$$
\begin{array}{ccc}
\HH^0(\Tot^{\bullet}(\M^{\bullet,\bullet}(0))) &\cong & \Q, \cr
\HH^1(\Tot^{\bullet}(\M^{\bullet,\bullet}(0))) &\cong & 0, \cr
\HH^2(\Tot^{\bullet}(\M^{\bullet,\bullet}(0))) &\cong &  \Q. 
\end{array}
$$
\end{enumerate}
\end{example}

\subsubsection{Algorithm for Computing the First Few Betti Numbers}
Using the construction of the double complex outline in Section 
\ref{subsec:doublecover},
as well as the single exponential time algorithm for obtaining 
covers by contractible sets described in Section \ref{sec:acycliccov}, 
along with straightforward  algorithms from  linear algebra,
it is now easy to obtain the following result:

\begin{theorem}\cite{Basu7}
\label{the:bettifew}
For any given $\ell$, there is an algorithm that
takes as input a ${\mathcal P}$-formula describing a semi-algebraic set 
$S \subset \R^k$,
and outputs $b_0(S),\ldots,b_\ell(S).$ 
The complexity of the algorithm is $(sd)^{k^{O(\ell)}}$, 
where $s  = \#({\mathcal P})$ and $d = \max_{P\in {\mathcal P}}{\rm deg}(P).$
\end{theorem}
Note that the complexity is single exponential in $k$ for every fixed $\ell$.

\section{The Quadratic Case}
\label{sec:quadratic}

\subsection{Brief Outline}
\label{sec:outline}
We denote by $\Sphere^k \subset \R^{k+1}$ the unit sphere 
centered at the origin.
Consider the case of 
semi-algebraic subsets of the unit sphere, $\Sphere^k \subset \R^{k+1},$
defined by homogeneous quadratic inequalities.
There is a straightforward reduction of the  
general problem to this special case.

Let $S \subset \Sphere^{k}$ be the set defined on $\Sphere^k$  by 
$s$ inequalities, $P_1 \leq 0,\ldots, P_s \leq 0$, where
$P_1,\ldots,P_s \in \R[X_0,\ldots,X_k]$ are homogeneous quadratic
polynomials. For  each $i, 1 \leq i \leq s$, 
let $S_i \subset \Sphere^k$ denote 
the set defined on $\Sphere^k$ by $P_i \leq 0$. Then, 
$
\displaystyle{
S = \bigcap_{i=1}^s S_i.
}
$
There are two main  ingredients in the polynomial time algorithm for
computing the top Betti numbers of $S$.

The first main idea is to consider $S$ as the intersection of the
various $S_i$'s and to utilize  the double complex arising from the generalized
Mayer-Vietoris exact sequence (see Section \ref{sec:topbackground}). 
It follows from the exactness of the generalized Mayer-Vietoris sequence
(see Definition \ref{def:MVsequence}),
that the top dimensional cohomology groups of $S$ are isomorphic to those
of the total complex associated to a suitable truncation of the Mayer-Vietoris
double complex. However, computing even the truncation of the 
Mayer-Vietoris double complex, starting from a triangulation of $S$ would 
entail a doubly exponential complexity. However, we utilize the fact that
terms appearing in the truncated complex depend on
the unions  of the $S_i$'s taken at most $\ell+2$ at a time
(cf. Remark \ref{rem:induction}). 
There are at most 
$\displaystyle{
\sum_{j=1}^{\ell+2} {s \choose j}}
$ 
such sets.
Moreover, for semi-algebraic
sets defined by the disjunction of a small number of quadratic inequalities, 
we are able to compute in polynomial (in $k$) time a complex,
whose homology groups are  isomorphic to those of the given sets. 
The construction
of these complexes in polynomial time is the second important ingredient
in our algorithm and  is outlined below 
in Section \ref{subsec:comp}.
These complexes along with the homomorphisms between them
define
another double complex
whose associated spectral sequence (corresponding
to the column-wise filtration) is isomorphic 
from the $E_2$ term onwards 
to the
corresponding one of the (truncated) Mayer-Vietoris double complex
(see Theorem \ref{the:main} below).
Since, we know that 
the latter converges to the homology groups of $S$, 
the top Betti numbers of $S$ are equal to
the ranks of the homology groups of the associated total complex of the 
double complex we computed. These can then be computed using well known 
efficient algorithms from linear algebra.

In order to carry through the program described above we need to understand
a few things about the topology of sets defined by homogeneous quadratic
inequalities.

\subsection{Topology of Sets Defined by Quadratic Inequalities}
\label{sec:top_quadratic}
In this section we state a few results concerning the topology
of sets defined by quadratic inequalities, which are exploited in 
designing efficient algorithms for computing their Betti numbers.

\subsubsection{Case of One Quadratic Form}
We first consider the case of a single quadratic form 
$Q \in \R[X_0,\ldots,X_k]$. Let $S \subset \Sphere^k$ be the
set defined by $Q \geq 0$ on the unit sphere in $\R^{k+1}$.
The crucial fact that distinguishes quadratic forms from forms
of higher degree is that the homotopy type of the set $S$ is 
determined by a single invariant attached to the quadratic form
$Q$, namely its {\em index}. 

\begin{definition}[Index of a quadratic form]
\label{def:index}
For any  quadratic form $Q$, ${\rm index}(Q)$ is the number of
negative eigenvalues of the symmetric matrix of the corresponding bilinear
form, that is of the matrix $M$ such that,
$Q(x) = \langle Mx, x \rangle$ for all $x \in \R^{k+1}$ 
(here $\langle\cdot,\cdot\rangle$ denotes the usual inner product). 
We will also
denote by $\lambda_i(Q), 0 \leq i \leq k$, the eigenvalues of $M$, 
in non-decreasing order, i.e.
\[ \lambda_0(Q) \leq \lambda_1(Q) \leq \cdots \leq \lambda_k(Q).
\]
\end{definition}

A simple argument involving diagonalizing the quadratic form $Q$ 
(see \cite{Agrachev} or \cite{Basu8})
yields that the homotopy type of the set $S$ defined above is related to the 
${\rm index}(Q)$ by

\begin{proposition}
\label{prop:oneform}
The set $S$ is homotopy equivalent to the $\Sphere^{k - {\rm index}(Q)}$.
\end{proposition}

\begin{example}
\label{ex:oneform}
The following figure (Figure \ref{fig:illus1})
illustrates Proposition \ref{prop:oneform}.
We display (from left to right) the subsets of $\Sphere^2$ 
described by the inequalities
\begin{align*}
X_0^2 + 2X_1^2 + 3X_2^2 &\geq 0  \;\; ({\rm index} = 0),\\
X_0^2 + 2X_1^2 - 3X_2^2 &\geq 0  \;\; ({\rm index} = 1),\\
X_0^2 - 2X_1^2 - 3X_2^2 &\geq 0  \;\; ({\rm index} = 2)
\end{align*}
respectively. Notice that each of the quadratic forms defining these sets
are already in a diagonal form, and hence its index can be read off
directly from the signs of the coefficients (the index is the number
of negative coefficients). 
By Proposition \ref{prop:oneform} these 
sets have the homotopy types of $\Sphere^2$, $\Sphere^1$, and
$\Sphere^0$ respectively, as can be also seen from the displayed images
below.

\vspace*{2cm}
\begin{figure}[hbt]
\centerline{\scalebox{0.5}{
\begin{picture}(500,100)
\includegraphics[bb=0 0 64mm 30mm]{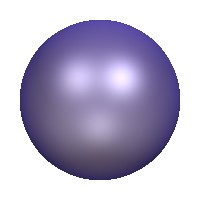}%
\includegraphics[bb=0 0 64mm 30mm]{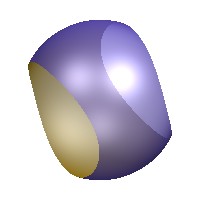}%
\includegraphics[bb=0 0 64mm 30mm]{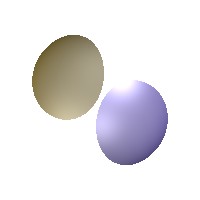}%
\end{picture}
}
}
\caption{Subsets of $\Sphere^2$ defined by one homogeneous quadratic 
inequality of index $0$, $1$ and $2$.}
\label{fig:illus1}
\end{figure}
\end{example}

\subsubsection{Case of Several Quadratic Forms}
Now let $Q_1,\ldots,Q_{s}$ be homogeneous quadratic polynomials in
$\R[X_0,\ldots,X_k]$. 

We denote by $Q = (Q_1,\ldots,Q_s): \R^{k+1} \rightarrow \R^s$,
the map defined by the forms  $Q_1,\ldots,Q_s$.  

Let
\begin{equation}
\label{eqn:defofT}
T = \bigcup_{1 \leq i \leq s}\{ x \in \Sphere^k \mid   Q_i(x) \leq   0 \},
\end{equation}
and let 
\begin{equation}
\label{eqn:defofOmega}
\Omega = \{\omega \in \R^{s} \mid  |\omega| = 1, \omega_i \leq 0, 1 \leq i \leq s\}.
\end{equation}

For $\omega \in \Omega$ we denote by ${\omega}Q$ the quadratic form
\begin{equation}
\label{eqnLdefofomegaQ}
{\omega} Q = \sum_{i=1}^{s} \omega_i Q_i.
\end{equation}

Let $B \subset \Omega \times \Sphere^k$ be the set defined by
\begin{equation}
\label{eqn:defofB}
B = \{ (\omega,x)\mid \omega \in \Omega, x \in \Sphere^k \;\mbox{and} \; 
{\omega}Q(x) \geq 0\}.
\end{equation}

We denote by $\phi_1: B \rightarrow \Omega $ and 
$\phi_2: B \rightarrow \Sphere^k$ the two projection maps. 

\[
\begin{diagram}
\node{}\node{B} \arrow{sw,t}{\phi_1} 
\arrow{se,t}{\phi_2} \\
\node{\Omega} \node{} \node{\Sphere^k}
\end{diagram}
\]

With the notation developed above we have
\begin{proposition} \cite{Agrachev}
\label{prop:homotopy2}
The map $\phi_2$ gives a homotopy equivalence between $B$ and 
$\phi_2(B) = T$.
\end{proposition}

We denote by 
\begin{equation}
\label{eqn:defofOmega_j}
\Omega_j = \{\omega \in \Omega \;  \mid \;  \lambda_j({\omega}P) \geq 0 \}.
\end{equation}

It is clear that the $\Omega_j$'s induce a filtration of the space
$\Omega$, i.e.,
$
\Omega_0 \subset \Omega_1 \subset \cdots \subset \Omega_{k}.
$

The following lemma follows directly from Proposition
\ref{prop:oneform}.
It is an important ingredient in the
algorithms for computing the Euler-Poincar\'e 
characteristic as well as the Betti numbers of semi-algebraic sets
defined by quadratic inequalities described later.

\begin{lemma}
\label{lem:sphere}
The fiber of the map $\phi_1$ over a point 
$\omega \in \Omega_{j}\setminus \Omega_{j-1}$ has the homotopy type
of a sphere of dimension $k-j$. 
\end{lemma}

We illustrate Lemma \ref{lem:sphere} with an example.

\begin{example}
\label{ex:twoquadrics}
In this example $s=2, k=2$, and

\begin{align*}
Q_1 =& - X_0^2 - X_1^2 - X_2^2, \\
Q_2 =&   X_0^2  + 2 X_1^2  + 3 X_2^2.
\end{align*}

\begin{figure}[hbt]
\scalebox{0.8}{
\includegraphics{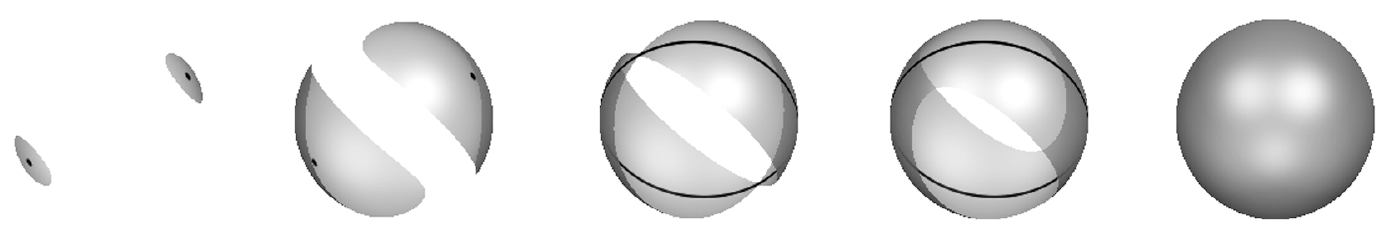}%
}
\caption{Type change: $\emptyset\to \Sphere^0\to \Sphere^1\to \Sphere^2$. 
$\emptyset$ is not shown. }
\label{fig:illus2}
\end{figure}

The set $\Omega$ is the part of the unit circle in the third quadrant of the
plane.
In the following Figure \ref{fig:illus2}, we display
the fibers of the map $\varphi_1^{-1}(\omega) \subset B$ for a sequence of 
values of $\omega$ starting from $(-1,0)$ and ending at 
$(0,-1)$. We also show the spheres
of dimensions $0,1$, and $2$, that these fibers
retract to. At $\omega = (-1,0)$, it is easy to verify that
${\rm index}(\omega Q) = 3$, and the 
fiber $\varphi_1^{-1}(\omega) \subset B$
is empty. Starting from
$\omega = (-\cos(\arctan(1)),-\sin(\arctan(1)))$ we have 
${\rm index}(\omega Q) = 2$
and the fiber 
$\varphi_1^{-1}(\omega)$ consists of the union of two spherical caps 
homotopy equivalent to $\Sphere^0$.
Starting from
$\omega = (-\cos(\arctan(1/2)),-\sin(\arctan(1/2)))$ we have
${\rm index}(\omega Q) = 1$, and 
the fiber $\varphi_1^{-1}(\omega)$ is homotopy equivalent to $\Sphere^1$. Finally,
starting from
$\omega = (-\cos(\arctan(1/3)),-\sin(\arctan(1/3)))$,
${\rm index}(\omega Q) = 0$, and 
the fiber $\varphi_1^{-1}(\omega)$ stays equal to to $\Sphere^2$.
\end{example}

As a consequence of Lemma \ref{lem:sphere} we obtain 
the following proposition which relates
the Euler-Poincar\'e characteristic of the set
$T$ (cf. Eqn. (\ref{eqn:defofT})) 
with the Borel-Moore Euler-Poincar\'e characteristic
(cf. Definition \ref{def:defofBMEP})
of $\Omega_j\setminus \Omega_{j-1}, \; 0 \leq j \leq k+1$.

\begin{proposition}
\label{prop:chi1}
\begin{equation}
\label{eqn:EP}
\chi(T) = 
\chi^{BM}(T) = \sum_{j=0}^{k+1} \chi^{BM}(\Omega_{j}\setminus \Omega_{j-1})
(1 + (-1)^{(k-j)}).
\end{equation}
\end{proposition}

It is instructive to continue Example \ref{ex:twoquadrics} and compute the
Euler-Poincar\'e characteristic of the set $T$ in that case using
Proposition \ref{prop:chi1}.
\begin{example}
In this example for each $0 < j \leq 3$, 
\[
\Omega_j \setminus \Omega_{j-1} \mbox{ is homeomorphic to } [0,1)
\]
and for $j=0$ we have 
\[
\Omega_0 \setminus \Omega_{-1} = \Omega_0 \mbox{ is homeomorphic to } [0,1].
\]
Recall from Eqn. (\ref{eqn:bmepforhalfopeninterval})
that $\chi^{BM}([0,1)) = 0$ and
$\chi^{BM}([0,1]) = \chi([0,1]) = 1$.
Finally using Eqn. (\ref{eqn:EP}) we deduce that
\[
\chi(T) = 2.
\]
\end{example}

\subsection{Computing the Euler-Poincar\'e Characteristics of Sets
Defined by Few Quadratic Inequalities}
\label{subsec:epquadratic}
Proposition \ref{prop:chi1} reduces the problem of 
computing the Euler-Poincar\'e characteristic of the set $T$, which
is defined by quadratic forms in $k+1$ variables, 
to computing the Borel-Moore Euler-Poincar\'e characteristics of the
sets $\Omega_{j}\setminus \Omega_{j-1} \subset \Omega$,
which are defined by $O(k)$ 
polynomials  in $s$ variables of degree also bounded by $O(k)$.
We now utilize an efficient algorithm 
for listing the Borel-Moore Euler-Poincar\'e characteristics of 
the realizations of all 
all sign conditions of a family of polynomials
developed in \cite{BPR6} to compute the terms occurring on the right
hand side of Eqn. (\ref{eqn:EP}). The complexity of this algorithm is
exponential in the number of variables (which is $O(s)$ in this case)
and polynomial in the number and degrees of the input polynomials
(which are $O(k)$ in this case).

The set $T$ is defined by a {\em disjunction} of homogeneous quadratic
inequalities. But since the Euler-Poincar\'e characteristic 
satisfies the inclusion-exclusion formula
(cf. Proposition \ref{6:prop:epinclusionexclusion}), 
we are able to compute
it for sets defined by a conjunction of such inequalities within the same
asymptotic time bound by making at most $2^s$ calls to the 
algorithm for disjunctions. 

For inhomogeneous quadratic inequalities there is
an easy  reduction  to the homogeneous case
(see \cite{Basu6} for detail).
As a result we obtain

\begin{theorem}\cite{Basu6}
\label{the:algoEP}
There exists an algorithm 
which takes as input a closed semi-algebraic set
$S \subset \R^k$ defined by 
\[
P_1 \leq 0, \ldots, P_s \leq 0,
P_i \in \R[X_1,\ldots,X_k], \deg(P_i) \leq 2,
\]
and computes the Euler-Poincar\'e characteristic of $S$. The complexity of the
algorithm is $k^{O(s)}$.
\end{theorem}

\begin{remark}
Very recently \cite{BP'R07} the above algorithm has been generalized to
the following setting.
Let 
\[
{\mathcal Q} = \{Q_1,\ldots,Q_s\} 
\subset \R[X_1,\ldots,X_k,Y_1,\ldots,Y_\ell]
\]
with
$
\deg_X(Q_i) \leq 2, \deg_Y(Q_i) \leq d, 1 \leq i \leq s,
$
and
\[
{\mathcal P}
\subset \R[Y_1,\ldots,Y_\ell]
\]
with 
$\deg(P) \leq d, P \in {\mathcal P}$ and $\#{\mathcal P} = m$.
Let $S \subset \R^{k+\ell}$  be a 
${\mathcal P} \cup {\mathcal Q}$-closed semi-algebraic set.
Then,
\begin{theorem}\cite{BP'R07}
\label{the:generalEP}
There exists an algorithm for computing the Euler-Poincar\'e
characteristic of $S$ whose complexity is bounded by
$(k \ell m d)^{O(s(s+\ell))}$.
\end{theorem}

Notice that Theorem \ref{the:generalEP} is a generalization
of Theorem \ref{the:algoEP} in several respects. It allows a subset
of the variables to occur with degrees bigger than $2$ (and the
complexity of the algorithm is exponential in the number of these
variables) and it takes as input general 
${\mathcal P}\cup{\mathcal Q}$-closed semi-algebraic sets, not just
basic closed ones.
\end{remark}

\subsection{Computing the Betti Numbers}
\label{subsec:comp}
\subsubsection{The Homogeneous Case}
We first consider the homogeneous case.

Let ${\mathcal P}= (P_1,\ldots,P_s) \subset \R[X_0,\ldots,X_k]$ be a 
$s$-tuple of quadratic forms (i.e. homogeneous quadratic polynomials).
For any subset ${\mathcal Q} \subset {\mathcal P}$ we denote by
$T_{\mathcal Q} \subset \Sphere^k$ the semi-algebraic set
$$
\displaylines{
T_{\mathcal Q} = \bigcup_{P \in {\mathcal Q}}
               \{x \in \Sphere^k\; \mid \; P(x) \leq 0 \},
}
$$
and let
$$
\displaylines{
S = \bigcap_{P \in {\mathcal P}}
               \{x \in \Sphere^k\; \mid \; P(x) \leq 0 \}.
}
$$
\noindent
We denote by $\Ch^\bullet({\mathcal H}(T_{\mathcal Q}))$
the co-chain complex of a triangulation 
${\mathcal H}(T_{\mathcal Q})$ of $T_{\mathcal Q}$
which is to be chosen sufficiently fine.

We first describe for each subset ${\mathcal Q} \subset {\mathcal P}$
with $\#{\mathcal Q} = \ell < k$
a complex, ${\mathcal M}^{\bullet}_{\mathcal Q}$, 
and natural homomorphisms
$$
\psi_{\mathcal Q}: \Ch^\bullet({\mathcal H}(T_{\mathcal Q})) \rightarrow
{\mathcal M}^\bullet_{\mathcal Q}
$$
which induce isomorphisms
$$
\psi_{\mathcal Q}^*:
H^{*}(\Ch^{\bullet}({\mathcal H}(T_{\mathcal Q}))) \rightarrow
H^*({\mathcal M}_{\mathcal Q}^\bullet).
$$ 

Moreover, for ${\mathcal B} \subset {\mathcal A} \subset {\mathcal P}$ with
$\#{\mathcal A} = \#{\mathcal B} + 1  < k$, we construct a homomorphism
of complexes
$$
\phi_{{\mathcal A},{\mathcal B}}: {\mathcal M}^{\bullet}_{\mathcal A} \rightarrow 
{\mathcal M}^{\bullet}_{\mathcal B}
$$
such that the following diagram commutes.

\begin{equation}
\label{eqn:commutative}
\begin{diagram}
\node{H^*({\mathcal M}^{\bullet}_{\mathcal A})}
\arrow{e,t}{\phi_{{\mathcal A},{\mathcal B}}^*}
\node{H^*({\mathcal M}^{\bullet}_{\mathcal B})} \\
\node{H^*(\Ch^\bullet({\mathcal H}(T_{\mathcal A})))}
\arrow{e,t}{r^*}\arrow{n,l}{\psi_{\mathcal A}^*}
\node{H^*(\Ch^\bullet({\mathcal H}(T_{\mathcal B})))}
\arrow{n,l}{\psi_{\mathcal B}^*}
\end{diagram}
\end{equation}

\noindent
In the above diagram
$\phi_{{\mathcal A},{\mathcal B}}^*$ and $r^*$ are the induced homomorphisms
of $\phi_{{\mathcal A},{\mathcal B}}$ and the restriction homomorphism $r$ 
respectively.

Now consider a fixed subset ${\mathcal Q} \subset {\mathcal P}$, which 
without loss of generality, we take to be $\{P_1,\ldots,P_\ell\}$. Let
$$
P =(P_1,\ldots,P_\ell): \R^{k+1} \rightarrow \R^\ell
$$ 
denote the corresponding quadratic map.

Let
$\R^{\mathcal Q} = \R^{\ell}$ and
$$
\Omega_{\mathcal Q} = \{\omega \in \R^{\ell} \mid  |\omega| = 1, 
\omega_i \leq 0, 1 \leq i \leq \ell\}.
$$

Let 
$B_{\mathcal Q} \subset \Omega_{\mathcal Q} \times \Sphere^k$ be the set defined by
\[
B_{\mathcal Q} = \{ (\omega,x) \mid \omega \in \Omega_{\mathcal Q}, x \in \Sphere^k \;\mbox{and} \; 
{\omega}P(x) \geq  0\},
\]
and we denote by $\phi_{1,{\mathcal Q}}: {B}_{\mathcal Q} \rightarrow \Omega_{\mathcal Q}$ 
and 
$\phi_{2,{\mathcal Q}}: {B}_{\mathcal Q} \rightarrow \Sphere^k$ the two projection maps.

For each subset ${\mathcal Q}' \subset {\mathcal Q}$ we have a natural inclusion
$\Omega_{{\mathcal Q}'} \hookrightarrow \Omega_{\mathcal Q}$.

\subsubsection{Index Invariant Triangulations}
We now define a certain special kind of 
semi-algebraic triangulation of $\Omega_{{\mathcal Q}}$
that will play an important role in our algorithm.

\begin{definition}[Index Invariant Triangulation]
\label{def:iit}
An {\em index invariant triangulation} of $\Omega_{{\mathcal Q}}$ consists of:
\begin{enumerate}
\item
A semi-algebraic triangulation,
$$
h: \Delta_{{\mathcal Q}}  \rightarrow \Omega_{{\mathcal Q}}
$$ of  $\Omega_{{\mathcal Q}}$
which is compatible with the subsets $\Omega_{{\mathcal Q}'}$ for every
${\mathcal Q}' \subset {\mathcal Q}$ and
such that for any simplex $\sigma$ of $\Delta_{{\mathcal Q}}$, 
${\rm index}(\omega P_{\mathcal Q})$ 
as well as the multiplicities of the eigenvalues of $\omega P_{\mathcal Q}$
stay invariant as $\omega$ varies over $h(\sigma)$;
\item
for every simplex $\sigma$ of $\Delta_{{\mathcal Q}}$
with ${\rm index}(\omega P_{\mathcal Q}) =  j$ for $\omega \in h(\sigma)$, 
a uniform description of a family of
orthonormal vectors $e_0(\sigma,\omega),\ldots, e_k(\sigma,\omega)$,
parametrized by $(\omega,x) \in h(\sigma)$
having the property that 
\[
\{e_j(\sigma,\omega),\ldots,e_k(\sigma,\omega)\}
\] 
is a basis for the linear subspace $L^+(\omega) \subset \R^{k+1}$
(which is the orthogonal complement to the sum of the
eigenspaces corresponding to the
first $j$ eigenvalues of $\omega P_{\mathcal Q}$).
\end{enumerate}
\end{definition}

An algorithm to compute index invariant triangulations
is described in \cite{Basu8} (see also \cite{BP'R07})
the complexity of this algorithm is 
bounded by $k^{2^{O(s)}}$.
The same bound holds for 
the size of the complex $\Delta_{{\mathcal Q}}$ as well as the
degrees of the polynomials occurring in the 
parametrized representation of the vectors
$\{e_0(\sigma,\omega),\ldots,e_k(\sigma,\omega)\}$.

Now fix an index invariant triangulation
$
h: \Delta_{{\mathcal Q}} \rightarrow \Omega_{{\mathcal Q}}
$ 
satisfying the complexity estimates stated above.

We now construct a cell complex homotopy equivalent to
$B_{{\mathcal Q}}$. It is obtained by glueing together certain regular 
cell complexes,  ${\mathcal K}(\sigma)$,
where $\sigma \in \Delta_{{\mathcal Q}}$.

       \begin{figure}[hbt]
         \centerline{
           \scalebox{0.5}{
 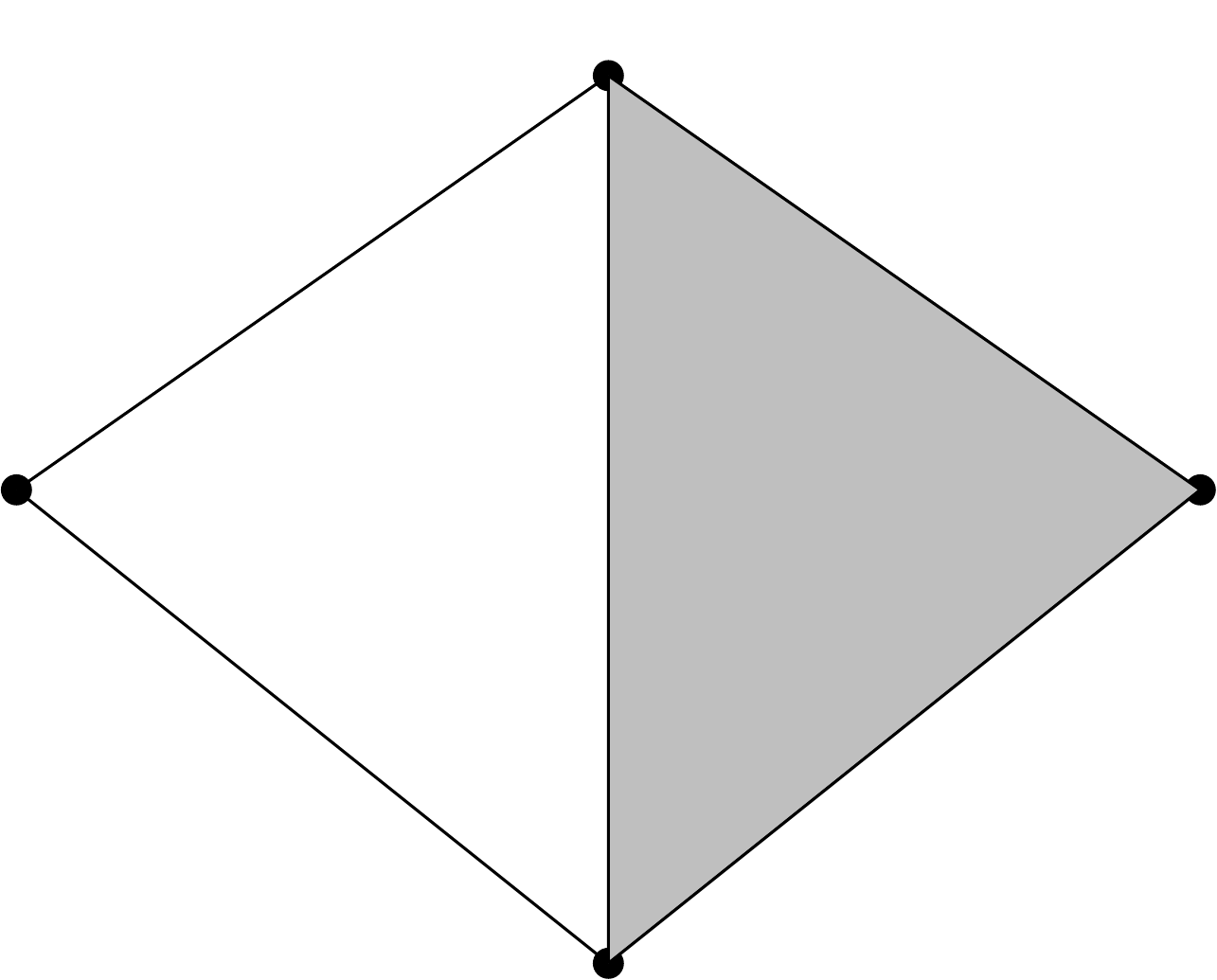
             }
           }
         \caption{The complex $\Delta_{{\mathcal Q}}$.}
         \label{fig-eg1}
       \end{figure}

       \begin{figure}[hbt]
         \centerline{
           \scalebox{0.5}{
  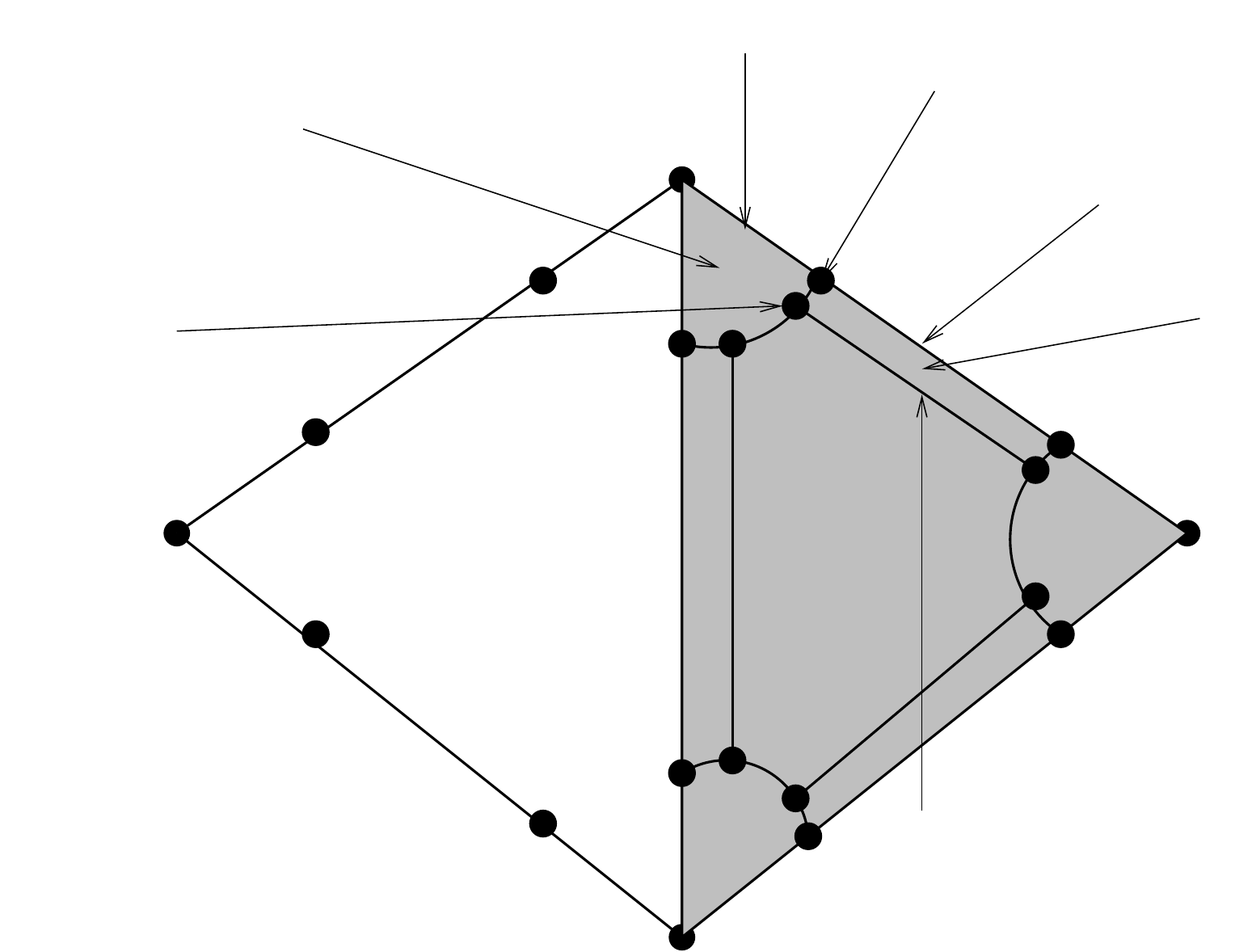
             }
           }
         \caption{The corresponding complex 
${\mathcal C}(\Delta_{{\mathcal Q}})$.}
         \label{fig-eg2}
       \end{figure}

Let $1 \gg  \eps_0 \gg \eps_1 \gg \cdots \gg \eps_{\ell} > 0$ be 
infinitesimals. 
For $\tau \in \Delta_{{\mathcal Q}}$
we denote by $D_{\tau}$ the subset of $\bar\tau$ defined by
\[
D_{\tau} = \{v \in \bar\tau \;\mid\; 
\dist(v,\theta) \geq \eps_{\dim(\theta)} \mbox{ for all }
\theta \prec \sigma \}
\]
where $\dist$ refers to the ordinary Euclidean distance.
Now let $\sigma \prec \tau$ be two simplices of $\Delta_{{\mathcal Q}}$.
We denote by 
$D_{\sigma,\tau}$ the subset of $\bar\tau$ defined by
\[
D_{\sigma,\tau} = \{v \in \bar\tau \;\mid\; \dist(v,\sigma) \leq 
\eps_{\dim(\sigma)},
\mbox{ and }\dist(v,\theta) \geq \eps_{\dim(\theta)} \mbox{ for all }
\theta \prec \sigma \}.
\]

Note that 
$$
\displaylines{
|\Delta_{{\mathcal Q}}| = 
\bigcup_{\sigma \in \Delta_{\mathcal Q}} D_{\sigma} \cup
\bigcup_{\sigma,\tau \in \Delta_{\mathcal Q},\sigma \prec \tau} D_{\sigma,\tau}.
}
$$

Also, observe that the various $D_\tau$'s and $D_{\sigma,\tau}$'s 
are all homeomorphic to closed balls
and moreover all non-empty intersections between them also have the 
same property.

\begin{definition}
\label{def:defofC}
The union of the $D_{\tau}$'s and $D_{\sigma,\tau}$'s together with the
non-empty intersections between them form  a regular cell complex
(cf. Definition \ref{def:cellcomplex}),
${\mathcal C}(\Delta_{{\mathcal Q}})$, whose underlying
topological space is $|\Delta_{{\mathcal Q}}|$ 
(see Figures \ref{fig-eg1} and \ref{fig-eg2}). 
\end{definition}
       
We now associate to each 
$D_{\sigma}$  (respectively,  $D_{\sigma,\tau}$)
a regular cell complex, ${\mathcal K}(\sigma)$, (respectively,
${\mathcal K}(\sigma,\tau)$)
homotopy equivalent to 
$\phi_{1}^{-1}(h(D_\sigma))$
(respectively,
$
\displaystyle{
\phi_{1}^{-1}(h(D_{\sigma,\tau})).
}
$

For each $\sigma \in \Delta_{{\mathcal Q}}$
and $\omega  \in h(\sigma)$ let
$\{e_0(\sigma,\omega),\ldots,e_{k}(\sigma,\omega)\}$ be 
the continuously varying  orthonormal basis of $\R^{k+1}$
computed previously.

The orthonormal basis 
\[
\{e_0(\sigma,\omega),\ldots,e_{k}(\sigma,\omega)\}
\] 
determines a complete flag of subspaces, 
${\mathcal F}(\sigma,\omega)$, consisting of
\begin{align*}
F^0(\sigma,\omega) = &~0, \\
F^1(\sigma,\omega) = &~\spanof(e_k(\sigma,\omega)),\\
F^2(\sigma,\omega,x) = &
~\spanof(e_k(\sigma,\omega),e_{k-1}(\sigma,\omega)), \\
\vdots & \\
F^{k+1}(\sigma,\omega) =&~\R^{k+1}.
\end{align*}

\begin{definition}
\label{def:cell}
For $0 \leq j \leq k$  let $c_{j}^+(\sigma,\omega)$ 
(respectively, $c_{j}^-(\sigma,\omega)$)
denote the $(k-j)$-dimensional cell consisting of the intersection of the
$F^{k-j+1}(\sigma,\omega)$
with the unit hemisphere in $\R^{k+1}$ 
defined by 
\begin{align*}
\{x \in \Sphere^k
\;\mid\;&
 \langle x,e_j(\sigma,\omega)\rangle \geq 0\}\\
\text{(respectively, } \quad
\{x \in \Sphere^k \;\mid\;& \langle x,e_j(\sigma,\omega)
\rangle \leq 0\}\quad
).
\end{align*}
\end{definition}

The regular cell complex ${\mathcal K}(\sigma)$ 
(as well as $\mathcal{K}(\sigma,\tau)$)
is defined as follows.

For each $v \in |\Delta_{{\mathcal Q}}|$ and 
$\sigma \in \Delta_{{\mathcal Q}}$ let 
$v(\sigma) \in |\sigma|$ denote the point of $|\sigma|$
closest to $v$.

The cells of ${\mathcal K}(\sigma)$ are
\[
\{(x,\omega) \mid x \in c_j^{\pm}(\sigma,\omega), \omega
\in h(c)\}
\]
where ${\rm index}(\omega P_{\mathcal Q}) \leq j \leq k$
and 
$c \in {\mathcal C}(\Delta_{{\mathcal Q}})$
is either $D_\sigma$ itself  or a cell
contained in the boundary  of $D_\sigma$.

Similarly, the cells of ${\mathcal K}(\sigma,\tau)$ are
\[
\{(x,\omega) \mid x \in c_j^{\pm}(\sigma,h(v(\sigma))), v = h^{-1}(\omega) 
\in c\}
\]
where 
${\rm index}(\omega P_{\mathcal Q}) \leq j \leq k$ and
$c \in {\mathcal C}(\Delta_{{\mathcal Q}})$
is either $D_{\sigma,\tau}$ itself  or a cell
contained in the boundary  of $D_{\sigma,\tau}$.

Our next step is to obtain cellular subdivisions
of each non-empty intersection amongst the
spaces associated to the complexes constructed above and thus obtain
a regular cell complex,
${\mathcal K}(B_{{\mathcal Q}})$, whose associated space,
$|{\mathcal K}(B_{{\mathcal Q}})|$, will be shown to be 
homotopy equivalent to $B_{{\mathcal Q}}$.

First notice that $|{\mathcal K}(\sigma',\tau')|$ (respectively, 
$|{\mathcal K}(\sigma)|$) has a non-empty intersection with 
$|{\mathcal K}(\sigma,\tau)|$ only if $D_{\sigma',\tau'}$ (respectively,
$D_{\sigma'}$) intersects $D_{\sigma,\tau}$. 

Let $D$ be some non-empty intersection amongst the 
$D_{\sigma}$'s and $D_{\sigma,\tau}$'s,
i.e. $D$ is a cell of ${\mathcal C}(\Delta_{{\mathcal Q}})$.
Then, 
$D \subset |\tau|$ for a unique simplex $\tau \in \Delta_{{\mathcal Q}}$
and 
$$
\displaylines{
D =  D_{\sigma_1,\tau} \cap \cdots \cap D_{\sigma_p,\tau} \cap D_\tau
}
$$
with $\sigma_1 \prec \sigma_2 \prec \cdots \prec \sigma_p \prec \sigma_{p+1} =
\tau$
and $p \leq \ell$. 

For each $i, 1 \leq i \leq p+1$,
let
$ \{ f_0(\sigma_i,v),\ldots, f_{k}(\sigma_i,v)\}$
denote a orthonormal basis of $\R^{k+1}$ where
\[
f_j(\sigma_i,v) = 
\lim_{t \rightarrow 0}
e_j(\sigma_i, h(t v(\sigma_i) + (1-t) v(\sigma_1))), 0 \leq j \leq k,
\] 
and let
${\mathcal F}(\sigma_i,v)$
denote the corresponding flag
consisting of
\begin{align*}
F^0(\sigma_i,v) = &~0, \\
F^1(\sigma_i,v) = &~\spanof(f_k(\sigma_i,v)),\\
F^2(\sigma_i,v) = &
~\spanof(f_k(\sigma_i,v),f_{k-1}(\sigma_i,v)), \\
\vdots & \\
F^{k+1}(\sigma_i,v) =& ~\R^{k+1}.
\end{align*}
 
We thus have $p+1$ different flags 
\[
{\mathcal F}(\sigma_1,v),
\ldots, {\mathcal F}(\sigma_{p+1},v), 
\]
and these give rise to $p+1$ different regular cell decompositions of 
$\Sphere^k$.

       \begin{figure}[h]
         \centerline{
           \scalebox{0.5}{
   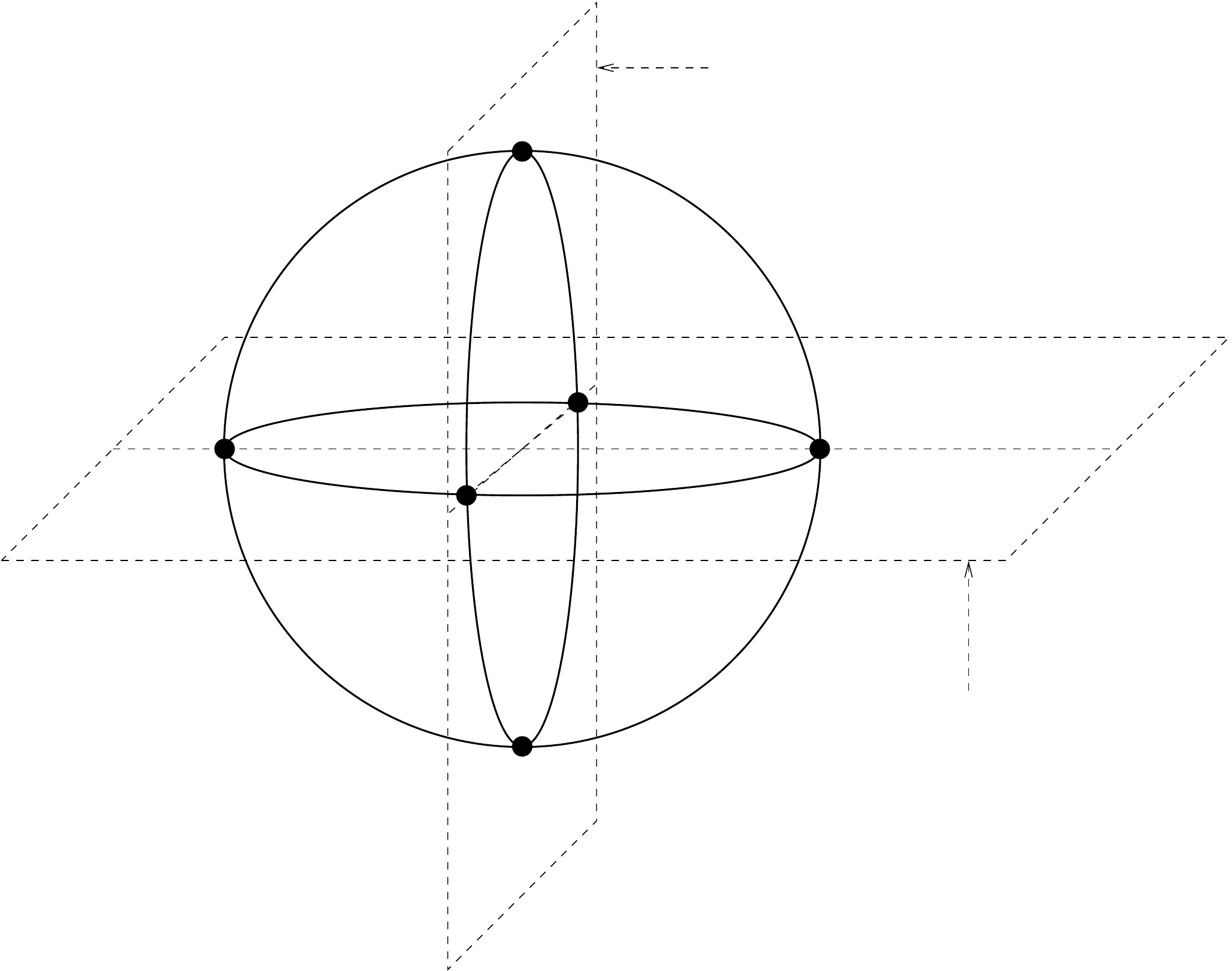
             }
           }
         \caption{The cell complex ${\mathcal K}'(D,v)$.}  
         \label{fig-eg5}
       \end{figure}

There is a unique smallest regular cell complex, 
${\mathcal K}'(D,v)$,  
that refines all these cell decompositions
whose cells are the following.
Let $L \subset \R^{k+1}$ be  a linear subspace of dimension 
$j, 0 \leq j \leq k+1$, which is an intersection
of linear subspaces $L_1,\ldots,L_{p+1}$ where 
$L_i \in {\mathcal F}(\sigma_i,v), 1 \leq i \leq p+1 \leq \ell+1$. 
The elements of the flags
${\mathcal F}(\sigma_1,v),
\ldots, {\mathcal F}(\sigma_{p+1},v)$ of dimension $j+1$
partition $L$ into polyhedral cones of various dimensions. The 
intersections of these cones with $\Sphere^{k}$ 
over all such subspaces $L \subset \R^{k+1}$ are the cells of 
${\mathcal K}'(D,v)$.
Figure \ref{fig-eg5} illustrates the refinement described above in case
of two flags in $\R^3$.
We denote by ${\mathcal K}(D,v)$ the sub-complex of 
${\mathcal K}'(D,v)$ consisting of only those cells included in
$L(\sigma_1,h(v(\sigma_1))) \cap \Sphere^k$.

We now triangulate $h(D)$ using 
the algorithm implicit in Theorem  \ref{the:triangulation} (Triangulation)
so that the combinatorial type of the arrangement of flags
\[
{\mathcal F}(\sigma_1,v),
\ldots, {\mathcal F}(\sigma_{p+1},v)
\] 
and hence the cell decomposition ${\mathcal K}'(D,v)$
stays invariant over the image,
$h_D(\theta)$,  of  each simplex, $\theta$, of this triangulation.
Notice that the combinatorial type of the cell 
decomposition ${\mathcal K}'(D,v)$ is determined by the signs of the inner 
products
$\la f_j(\sigma_i,v), f_{j'}(\sigma_{i'},v) \ra$ where
$0 \leq j,j' \leq \ell, 1 \leq i,i' \leq p+1$.

We compute 
a family of polynomials
${\mathcal A}_{D}\subset \R[Z_1,\ldots,Z_\ell]$ whose signs 
determine the vanishing or non-vanishing of the inner products
$\la f_j(\sigma_i,v), f_{j'}(\sigma_{i'},v) \ra, 0 \leq j,j' \leq k, 
1 \leq i,i' \leq p+1$. 
It is then clear that the combinatorial type of the cell 
decomposition ${\mathcal K}'(D,v)$
will stay invariant as $\omega$ varies over each connected component of 
any realizable sign condition on 
${\mathcal A}_{D}\subset \R[Z_1,\ldots,Z_\ell]$.

Given the complexity bounds on the rational functions 
defining the orthonormal bases
$\{e_0(\sigma,\omega),\ldots,e_\ell(\sigma,\omega)\}$
$\omega\in h(\sigma)$, stated above
that the
number and degrees of the polynomials in the family ${\mathcal A}_D$
are bounded by  $k^{2^{O(s)}}$.
We then
use the algorithm implicit in Theorem  \ref{the:triangulation} (Triangulation)
with ${\mathcal A}_{D}$ as input, to obtain the required triangulation.

The closures of the sets
\[
\{(x,\omega) \;\mid\; x \in c \in  {\mathcal K}(D,h^{-1}(\omega)), \; 
\omega  \in h(h_D(\theta))\}
\] 
form a regular cell complex which we denote by
${\mathcal K}(D)$.

The following proposition gives an upper bound on the size of the
complex ${\mathcal K}(D)$. We use the notation introduced in the previous
paragraph.
\begin{proposition}
\label{prop:complexity}
For each $\omega \in h(D)$, the number of cells in 
${\mathcal K}(D,h^{-1}(\omega))$
is bounded by $k^{O(\ell)}$. Moreover, the number of cells in the complex
${\mathcal K}(D)$ is bounded by $k^{2^{O(\ell)}}$.
\end{proposition}

Note that there is a homeomorphism
$i_{D,\sigma_i}: |{\mathcal K}(\sigma_i,\tau)| \cap \phi_1^{-1}(h(D))
\rightarrow |{\mathcal K}(D)|$
which takes each cell of 
$|{\mathcal K}(\sigma_i,\tau)| \cap \phi_1^{-1}(h(D))$
to a union of cells in ${\mathcal K}(D)$. 
We use these homeomorphisms
to glue the cell complexes ${\mathcal K}(\sigma_i,\tau)$
together to form the cell complex ${\mathcal K}(B_{{\mathcal Q}})$.
More precisely

\begin{definition}
\label{def:defofK(B)}
${\mathcal K}(B_{{\mathcal Q}})$ is the union of all the complexes
${\mathcal K}(D)$ constructed above, where we use the maps
$i_{D,\sigma_i}$ to make the obvious identifications. 
\end{definition}

We have that
\begin{proposition}
\label{prop:iso2}
$|{\mathcal K}(B_{{\mathcal Q}})|$ 
is homotopy equivalent to $B_{{\mathcal Q}}$.
\end{proposition}

We also have

\begin{proposition}
\label{prop:complexity2}
The number of cells in the
cell complex ${\mathcal K}(B_{{\mathcal Q}})$ is 
bounded by $k^{2^{O(\ell)}}$.
\end{proposition}

\begin{proposition}\cite{Basu8}
\label{prop:iso}
For $0 \leq i \leq k-1$, the induced homomorphisms
$$
\psi_{\mathcal Q}^*: \HH^i(\Ch^{\bullet}({\mathcal H}(T_{\mathcal Q}))) 
\rightarrow \HH^i({\mathcal M}^\bullet_{\mathcal Q})
$$
are isomorphisms.
\end{proposition}

Now let
${\mathcal B} \subset {\mathcal A} \subset {\mathcal P}$ with
$\#{\mathcal A} = \#{\mathcal B} + 1  < k$. 

The simplicial complex $\Delta_{\mathcal B}$ is a subcomplex of 
$\Delta_{\mathcal A}$ and hence,
${C}_{\mathcal B}^{\bullet,\bullet}$ is a subcomplex of 
${C}_{\mathcal A}^{\bullet,\bullet}$ and thus there exists a natural
homomorphism (induced by restriction)
$$
\phi_{{\mathcal A},{\mathcal B}}: {C}_{\mathcal A}^{\bullet,\bullet}
\rightarrow 
{C}_{\mathcal B}^{\bullet,\bullet}
$$ 
and let
$$
\phi_{{\mathcal A},{\mathcal B}}: {\rm Tot}^\bullet({C}_{\mathcal A}^{\bullet,\bullet}) = {\mathcal M}_{\mathcal A}^{\bullet}
\rightarrow
{\mathcal M}_{\mathcal B}^{\bullet}
= {\rm Tot}^\bullet({C}_{\mathcal B}^{\bullet,\bullet}),
$$
be the induced homomorphism between the corresponding associated 
total complexes.

The complexes ${\mathcal M}_{\mathcal A}^{\bullet},
{\mathcal M}_{\mathcal B}^{\bullet}$, and the homomorphisms, 
$
\phi_{{\mathcal A},{\mathcal B}},
\psi_{\mathcal A}, \psi_{\mathcal B}
$ 
satisfy
 
\begin{proposition} \cite{Basu8}
\label{prop:commutative}
The diagram
\begin{equation}
\begin{diagram}
\node{{\mathcal M}^{\bullet}_{\mathcal A}}
\arrow{e,t}{\phi_{{\mathcal A},{\mathcal B}}}
\node{{\mathcal M}^{\bullet}_{\mathcal B}} \\
\node{\Ch^\bullet({\mathcal H}(T_{\mathcal A}))}
\arrow{e,t}{r}\arrow{n,l}{\psi_{\mathcal A}} 
\node{\Ch^\bullet({\mathcal H}(T_{\mathcal B}))}
\arrow{n,r}{\psi_{\mathcal B}}
\end{diagram}
\end{equation}
is commutative,
where $r$ is the restriction homomorphism.
\end{proposition}

We denote by
$$ 
\check{\phi}_{{\mathcal B},{\mathcal A}}:
\check{{\mathcal M}}_{\mathcal B}^{\bullet}
\rightarrow
\check{{\mathcal M}}_{\mathcal A}^{\bullet}
$$
the homomorphism dual to $\phi_{{\mathcal A},{\mathcal B}}$.
We denote by ${\mathcal D}^{\bullet,\bullet}_{\mathcal P}$ 
the double complex defined by:

$$
\displaylines{
{\mathcal D}_{\mathcal P}^{p,q} = 
\bigoplus_{{\mathcal Q} \subset {\mathcal P},\#{\mathcal Q} = p+1}
{\check{\mathcal M}}^{q}_{\mathcal Q}.
}
$$

The vertical differentials
$$
\displaylines{
d: {\mathcal D}_{\mathcal P}^{p,q} \rightarrow {\mathcal D}_{\mathcal P}^{p,q-1}
}
$$
are induced component-wise from the
differentials of the individual complexes 
${\check{\mathcal M}}^{\bullet}_{\mathcal Q}$.
The horizontal differentials
$$
\displaylines{
\delta: {\mathcal D}_{\mathcal P}^{p,q} \rightarrow {\mathcal D}_{\mathcal P}^{p+1,q}
}
$$
are defined as follows:
for $a \in {\mathcal D}_{\mathcal P}^{p,q} = 
\oplus_{\#{\mathcal Q} = p+1}{\check{\mathcal M}}^q_{\mathcal Q}$
for each subset 
\[
{\mathcal Q} = \{P_{i_0},\ldots,P_{i_{p+1}}\} \subset {\mathcal P}
\]
with $i_0 < \cdots < i_{p+1}$ 
the ${\mathcal Q}$-th component of 
$\delta a \in {\mathcal D}_{\mathcal P}^{p+1,q}$ is given by
$$
(\delta a)_{{\mathcal Q}} = \sum_{0 \leq j \leq p+1}
                         \check{\phi}_{{\mathcal Q}_j,{\mathcal Q}}(a_{{\mathcal Q}_j})
$$

where ${\mathcal Q}_j = {\mathcal Q} \setminus \{P_{i_j}\}$.
{\small
$$
\begin{array}{ccccccccc}
& & \vdots  && \vdots  && \vdots  && \cr
& &
\Big\downarrow\vcenter{\rlap{$d$}} & &
\Big\downarrow\vcenter{\rlap{$d$}} & &
\Big\downarrow\vcenter{\rlap{$d$}} & & \cr
0 & \longrightarrow & \oplus_{\#{\mathcal Q} = 1}{\check{\mathcal M}}^{3}_{\mathcal Q} & 
\stackrel{\delta}{\longrightarrow}  & \oplus_{\#{\mathcal Q} = 2}{\check{\mathcal M}}^3_{\mathcal Q}&
\stackrel{\delta}{\longrightarrow} & \oplus_{\#{\mathcal Q} = 3}{\check{\mathcal M}}^3_{\mathcal Q}&
 \longrightarrow & \cdots
\cr
& &
\Big\downarrow\vcenter{\rlap{$d$}} & &
\Big\downarrow\vcenter{\rlap{$d$}} & &
\Big\downarrow\vcenter{\rlap{$d$}} & &\cr
0 & \longrightarrow & \oplus_{\#{\mathcal Q} = 1}{\check{\mathcal M}}^2_{\mathcal Q} &
\stackrel{\delta}{\longrightarrow} & \oplus_{\#{\mathcal Q} = 2}{\check{\mathcal M}}^2_{\mathcal Q} &
\stackrel{\delta}{\longrightarrow} & \oplus_{\#{\mathcal Q} = 3}{\check{\mathcal M}}^2_{\mathcal Q} &
 \longrightarrow &\cdots 
\cr
& &
\Big\downarrow\vcenter{\rlap{$d$}} & &
\Big\downarrow\vcenter{\rlap{$d$}} & &
\Big\downarrow\vcenter{\rlap{$d$}} & &\cr
0 & \longrightarrow & \oplus_{\#{\mathcal Q} = 1}{\check{\mathcal M}}^1_{\mathcal Q}&
\stackrel{\delta}{\longrightarrow} & \oplus_{\#{\mathcal Q} = 2}{\check{\mathcal M}}^1_{\mathcal Q}&
\stackrel{\delta}{\longrightarrow} & \oplus_{\#{\mathcal Q} = 3}{\check{\mathcal M}}^1_{\mathcal Q} &
 \longrightarrow & \cdots
\cr
& &
\Big\downarrow\vcenter{\rlap{$d$}} & &
\Big\downarrow\vcenter{\rlap{$d$}} & &
\Big\downarrow\vcenter{\rlap{$d$}} & & \cr
0 & \longrightarrow & \oplus_{\#{\mathcal Q} = 1}{\check{\mathcal M}}^0_{\mathcal Q} &
\stackrel{\delta}{\longrightarrow} & \oplus_{\#{\mathcal Q} = 2}{\check{\mathcal M}}^0_{\mathcal Q} &
\stackrel{\delta}{\longrightarrow} & \oplus_{\#{\mathcal Q} = 3}{\check{\mathcal M}}^0_{\mathcal Q} &
 \longrightarrow & \cdots
\cr
& &
\Big\downarrow\vcenter{\rlap{$d$}} & &
\Big\downarrow\vcenter{\rlap{$d$}} & &
\Big\downarrow\vcenter{\rlap{$d$}} & & \cr
& & 0 && 0 && 0 & &\cr 
\end{array}
$$
}

We have the following theorem.

\begin{theorem} \cite{Basu8}
\label{the:main}
For $0 \leq i \leq k,$
$$
H^i(S) \cong 
H^i({\rm Tot}^{\bullet}({\mathcal D}^{\bullet,\bullet}_{\mathcal P})).
$$
\end{theorem}

Finally, using 
Theorem \ref{the:main} we have

\begin{theorem} \cite{Basu8}
\label{the:quadratic_main1}
There exists an algorithm which given a set of $s$ polynomials,
${\mathcal P} = \{P_1,\ldots,P_s\} \subset \R[X_1,\ldots,X_k],$
with ${\rm deg}(P_i) \leq 2, 1 \leq i \leq s,$
computes $b_{k-1}(S), \ldots, b_{k-\ell}(S),$ 
where $S$ is the set defined by $P_1 \leq 0,\ldots,P_s \leq 0$.
The complexity of the algorithm  is
\begin{equation}
\label{eqn:complexity}
\sum_{i=0}^{\ell+2} {s \choose i} k^{2^{O(\min(\ell,s))}}.
\end{equation}
If the coefficients of the polynomials in
${\mathcal P}$ are integers  of bit-sizes bounded by
 $\tau$, then the bit-sizes of the integers
appearing in the intermediate computations and the output
are bounded by $\tau (sk)^{2^{O(\min(\ell,s))}}.$
\end{theorem}

For certain applications we need the following more detailed version
of Theorem \ref{the:quadratic_main1}.

\begin{theorem} \cite{Basu8}
\label{the:quadratic_main2}
There exists an algorithm  which takes as input 
a family of polynomials
$\{P_1,\ldots,P_s\}\subset  \R[X_1\ldots,X_k],$ with
${\rm deg}(P_i) \leq 2$
and a number $\ell \leq k$,
and outputs a complex ${\mathcal D}^{\bullet,\bullet}_\ell $. 
The complex $\Tot^{\bullet}({\mathcal D}^{\bullet,\bullet}_\ell)$ is
quasi-isomorphic to $\Ch^\ell_{\bullet}(S)$, the truncated singular
chain complex of $S$,
where
$$S = \bigcap_{P \in {\mathcal P}}
               \{x \in \R^{k}\; \mid \; P(x) \leq 0 \}.
$$

Moreover, given a subset ${\mathcal P}' \subset {\mathcal P}$
with 
$$
S' = \bigcap_{P \in {\mathcal P}'}
               \{x \in \R^{k}\; \mid \; P(x) \leq 0 \}.
$$
the algorithm outputs both complexes ${\mathcal D}^{\bullet,\bullet}_\ell$ and
${\mathcal D}'^{\bullet,\bullet}_\ell$ (corresponding to the sets
$S$ and $S'$ respectively) along with the matrices defining 
a homomorphism $\Phi_{{\mathcal P},{\mathcal P}'}$ 
such that 
$\Phi_{{\mathcal P},{\mathcal P}'}^*: \HH^*(\Tot^\bullet({\mathcal D}^{\bullet,\bullet}_\ell)) \cong \HH_*(S) \rightarrow  \HH_*(S') \cong
\HH^*(\Tot^\bullet({\mathcal D'}^{\bullet,\bullet}_\ell))
$
is the homomorphism induced by the inclusion map $i: S \hookrightarrow S'$.
The complexity of the algorithm is 
$ 
\sum_{i=0}^{\ell+2} {s \choose i} k^{2^{O(\min(\ell,s))}}.
$
\end{theorem}

\subsection{Projections of Sets Defined by Quadratic Inequalities}
\label{subsec:projquad}
There are two main ingredients in the algorithm for computing
Betti numbers of projections of sets defined by quadratic inequalities.
The first is the use of descent spectral sequence described in 
Section \ref{subsec:descent}.
Notice that the individual terms 
occurring in the double complex in Section \ref{subsubsec:ddd}
correspond to the chain groups of the fibered
products of the original set. {\em A crucial observation 
here is that the fibered product of a set defined by few quadratic 
inequalities is again a set of the same type.}
However, since there
is no known algorithm for efficiently triangulating semi-algebraic sets
(even those defined by few quadratic inequalities) we cannot directly use the
spectral sequence to actually compute the Betti numbers of the projections.
In order to do that we need an additional ingredient.
This second main ingredient is the polynomial time algorithm 
in Theorem \ref{the:quadratic_main2}
for computing a complex whose cohomology groups are
isomorphic to those of a given semi-algebraic set defined by a
constant number of quadratic inequalities. Using this algorithm we are
able to construct a certain double complex, whose associated total
complex is quasi-isomorphic to (implying having isomorphic homology
groups) a suitable truncation of the one obtained from the cohomological
descent spectral sequence mentioned above. This complex is of much
smaller size and can be computed in polynomial time and is enough for
computing the first $q$ Betti numbers of the projection in polynomial
time for any constant $q$.

We have the following theorem.

\begin{theorem}\cite{BZ}
\label{the:projquad_main}
There exists an algorithm that takes as input
a basic semi-algebraic set $S \subset \re^{k+m}$ defined by 
\[
P_1 \geq  0, \ldots, P_\ell \geq 0,
\] 
with $P_i \in \re[X_1,\ldots,X_k,Y_1,\ldots,Y_m]$,
$\deg(P_i) \leq 2, \; 1 \leq i \leq \ell$ 
and outputs 
\[
b_{0}(\pi(S)),\ldots, b_q(\pi(S)),
\]
where
$\pi:\re^{k+m} \rightarrow \re^m$ be the projection onto the last
$m$ coordinates. 
The complexity of the algorithm is bounded by $(k + m)^{2^{O((q+1)\ell)}}$.
\end{theorem}

\section{Betti Numbers of Arrangements}
\label{sec:arrangements}
In this section  we describe an algorithm for computing the Betti
numbers of the union of a collection, ${\mathcal S}$,  of subsets 
of $\R^k$, where each set is
assumed to be a closed and bounded semi-algebraic set of 
{\em constant description complexity}. It is customary to call
the collection ${\mathcal S}$ an {\em arrangement} and 
we will refer to this problem as the problem of 
computing the Betti numbers of the arrangement ${\mathcal S}$.
A semi-algebraic
set in $\R^k$ is said to have constant description complexity
if it can be described by a first order formula of size
bounded by some constant (see also \cite{Basu9} for a
more general mathematical framework). The key point 
which distinguishes the results
in this section from those in the previous sections is that unlike
before, here we are interested only in the {\em
combinatorial part} of complexity estimates -- i.e.
the part of the complexity that depends on the number of sets
in the input. Since the input sets are of 
constant description complexity, the {\em algebraic part} 
of the complexity -- i.e. the part that depends on the degrees and 
number of polynomials defining each set -- is bounded by a constant. 
This point of view, which is now standard in discrete and computational
geometry (see \cite{Agarwal,Matousek}), 
presents new challenges from the point of view of
designing efficient algorithms for computing Betti numbers of 
arrangements of sets of constant description complexity.

Notice that, unlike before, in this setting it is not important 
to obtain a good (say single exponential in $k$) bound
on the the algebraic part of the complexity, since it is
bounded by some constant regardless of the exact nature of
the bound. Thus, we have much greater flexibility
in designing algorithms, since we can utilize triangulation
algorithms (cf. Theorem \ref{the:triangulation})
which have doubly exponential complexity as 
long as the number of sets in the input to each such call 
is bounded by a constant. 

On the other hand
the algorithms described in Section \ref{sec:bettifew}, while having
single exponential complexity, are no longer the best possible 
in this setting, since the combinatorial complexities of these 
algorithms are very far from being optimal. The goal is to use the 
flexibility afforded in the algebraic part to design algorithm having 
much tighter combinatorial complexity.

A version of the main result of this section (Algorithm 
\ref{alg:computingbettiofarrangements} below) appears in \cite{Basu5} where
a spectral sequence argument is used.
We present here a different (and simpler) algorithm
which avoids spectral sequences but instead uses the 
more geometric notion of homotopy colimits  
(cf. Definition \ref{def:hocolimit}).
The new algorithm has the same complexity as the previous one.
 
\subsection{Computing Betti Numbers via Global Triangulations}
As we have seen in Section \ref{sec:topbackground}, one approach towards
computing the Betti numbers of the arrangement is to obtain
a triangulation of the whole arrangement using the algorithm
implicit in Theorem \ref{the:triangulation}.
Thus, in order to compute the Betti numbers of an arrangement of 
$n$ closed and bounded semi-algebraic sets of constant description
complexity in $\R^k$
it suffices to first triangulate the
arrangement and then compute the Betti numbers of the corresponding simplicial
complex. However, using the complexity estimate in
 Theorem \ref{the:triangulation} the complexity of computing such
a triangulation is $O(n^{2^k})$.
However, since the Betti numbers of such an arrangement 
is bounded by $O(n^{k})$ (cf. Theorem \ref{the:B99}), 
it is reasonable to ask for an algorithm
whose  complexity is bounded by $O(n^k)$. 
More efficient ways of decomposing  arrangements into topological
balls have been proposed. In \cite{CEGS91} the authors provide a 
decomposition into $O^*(n^{2k-3})$ cells (see \cite{Koltun} for
an improvement of this result in the case $k=4$). 
However, this decomposition does not produce a cell complex and is 
therefore  not directly useful in computing the 
Betti numbers of the arrangement.

\subsection{Local Method}
We have seen in Section \ref{sec:topbackground} that 
in certain  simple situations it is possible to compute the 
Betti numbers of an arrangement without having to
compute a triangulation.
For instance,
if the arrangement has the Leray property (cf. Definition \ref{def:leray})
Theorem \ref{the:nerve}
provides an efficient way of computing
the Betti numbers of the union. 
The dimension of the $p$-th term of the nerve complex
$\LL^{p}({\mathcal C})$ (see Eqn. (\ref{eqn:defofLL})) is this case 
is bounded by ${n \choose p+1} = O(n^{p+1})$ corresponding
to all possible $(p+1)$-ary intersections amongst the $n$ given  sets.
The truncated complex,
$\LL^{p}_{\ell+1}({\mathcal C})$,
can be computed by testing for non-emptiness of each of the possible
$\sum_{1 \leq j \leq \ell+2} {n \choose j} = O(n^{\ell+2})$ at most
$(\ell + 2)$-ary intersections among the $n$ given sets. The 
first $\ell$ Betti numbers of the arrangements
can then be computed from $\LL^{p}_{\ell+1}({\mathcal C})$
using algorithms from linear algebra.
This technique would work, for instance, if one is interested in
computing the Betti numbers of a union of balls in $\R^k$.
However, this method is no longer useful if the sets in the
arrangement do not satisfy the Leray property.

For non-Leray arrangements, some new ideas are needed. 
Before introducing them we first need some new notation. 
For the rest of this section we fix a family
\begin{equation}
\label{eqn:defofS}
{\mathcal S} = \{S_1,\ldots, S_n\}
\end{equation}
of closed and bounded semi-algebraic subsets of $\R^k$. 
For $I \subset [n]$ we denote by
\begin{align}
S^I & = \bigcup_{i \in I} S_i \\
S_I & = \bigcap_{i \in I} S_i .
\end{align}

The main new idea is that in order to compute the first $\ell$ Betti numbers
of a non-Leray arrangement
${\mathcal S}$ it suffices to compute triangulations, $h^I$,
of the sets $S^I$ with $\#I \leq \ell+2$.
These triangulations should have a certain compatibility property namely --
the triangulation of $S_I$ obtained by restricting $h^I$
should be a refinement of the triangulations
of $S_J$ obtained by restricting $h^J$  for all $J \subset I$.

More formally, we define

\begin{definition}[Adaptive Triangulations]
\label{def:recursive triangulation}
An $\ell$-adaptive triangulation, $h_{\ell}({\mathcal S})$, 
of ${\mathcal S}$ is a 
collection
$
\displaystyle{
\{h^I\}_{I \subset [n], \#I \leq \ell+2}
}
$
of semi-algebraic triangulations 
\begin{equation}
h^I: K^I \rightarrow S^I
\end{equation}
having the following properties.
\begin{enumerate}
\item
For each $I \subset [n]$ with $\#I \leq \ell+2$
the triangulation $h^I$ respects the sets $S_i, i \in I$. In particular,
$h^I$ induces  a triangulation of $S_I$, which we denote by 
$h_I: K_I \rightarrow S_I$, where $K_I$ is a subcomplex of
$K^I$.
\item
For each $J \subset I \subset [n]$ with $\#I \leq \ell+2$,
the triangulation $h_I$ is a refinement of the triangulation $h_J|_{S_I}$.
\end{enumerate}
\end{definition}

We now show how to obtain from a given $\ell$-adaptive triangulation,
a cell complex whose first $\ell$ cohomology groups are isomorphic to those
of $S^{[n]}$. 
We will use the notion of homotopy colimits introduced in Section 
\ref{subsec:hocolimit}.

Given an $\ell$-adaptive triangulation, $h_{\ell}({\mathcal S})$, 
we associate to it a cell complex, ${\mathcal K}_\ell({\mathcal S})$
(best thought of as an infinitesimally thickened version of
$\hocolimit_{\leq \ell}({\mathcal S})$),
whose associated topological space is homotopy equivalent
to $|\hocolimit_{\leq \ell}({\mathcal S})|$.

\begin{definition}[The cell complex ${\mathcal K}_\ell({\mathcal S})$]
\label{def:defofK}
Let ${\mathcal C}$ denote the cell complex 
${\mathcal C}({\rm sk}_{\ell}(\Delta_{[n]}))$ defined previously 
(see Definition \ref{def:defofC} 
replacing $\Delta_{\mathcal Q}$ by $\Delta_{[n]}$).
Let $D$ be a cell of ${\mathcal C}({\rm sk}_{\ell}(\Delta_{[n]}))$.
Then, $D \subset |\Delta_I|$ for a unique simplex $\Delta_I \in 
\Delta_{[n]}$ with $\#I \leq \ell+2$
and (following notation introduced before in Definition \ref{def:defofC})
$$
\displaylines{
D =  D_{\Delta_{I_1},\Delta_{I}} \cap \cdots \cap D_{\Delta_{I_p},\Delta_I}
\cap D_{\Delta_I},
}
$$
with $I_1 \subset \I_2 \subset \cdots \subset I_p \subset
I_{p+1} =I$
and $p \leq \ell+1$. 
We denote
\begin{equation}
{\mathcal K}(D) = \{ D \times \overline{h_{I}(|\sigma|)}\;\mid\; \sigma \in 
K^{I}, \mbox{ with } h_{I}(|\sigma|) \subset  S^{I_1}\},
\end{equation}
and
\begin{equation}
{\mathcal K}_\ell({\mathcal S})
=
\bigcup_{D \in {\mathcal C}({\rm sk}_{\ell}(\Delta_{[n]}))}
{\mathcal K}(D).
\end{equation}
\end{definition}

Notice that $|{\mathcal K}_\ell({\mathcal S})|$
is a closed and bounded semi-algebraic  set
defined over $\R' = \R\la\eps_0,\ldots,\eps_\ell\ra$,
and it contains the semi-algebraic set
$\E(|\hocolimit_{\leq \ell}({\mathcal S})|,\R')$.
Furthermore, we have

\begin{proposition}
\label{prop:hocolimit2}
The semi-algebraic set 
\[
|{\mathcal K}_\ell({\mathcal S})|
\]
is homotopy equivalent to 
\[
\E(|\hocolimit_{\leq \ell}({\mathcal S})|,\R').
\]
\end{proposition}
\begin{proof}
From the definition of the complex ${\mathcal K}_\ell({\mathcal S})$ it follows
easily that 
\[
\lim_{\eps_0}
|{\mathcal K}_\ell({\mathcal S})| = |\hocolimit_{\leq \ell}({\mathcal S})|.
\]
It now follows (see \cite[Lemma 16.17]{BPRbook06}) that  
$
|{\mathcal K}_\ell({\mathcal S})|
$
is homotopy equivalent to 
$
\E(|\hocolimit_{\leq \ell}({\mathcal S})|,\R').
$
\end{proof}

By Theorem \ref{the:homologyoftruncated},
in order to compute the
first $\ell$ Betti numbers of $S^{[n]}$, it suffices to
compute the first $\ell$ Betti numbers of 
$|\hocolimit_{\leq \ell}({\mathcal S})|$.
Moreover, by virtue of Proposition \ref{prop:hocolimit2}, 
and Proposition \ref{6:prop:homotopicsa} (homotopy invariance
of the cohomology groups) we have that in order
to compute the Betti numbers of 
$|\hocolimit_{\leq \ell}({\mathcal S})|$
it suffices to compute the Betti numbers of the set 
$|{\mathcal K}_\ell({\mathcal S})|$. 
This is the main idea behind the following algorithm.

\subsection{Algorithm for Computing the Betti Numbers of Arrangements}
We can now describe our algorithm for computing the first $\ell$ Betti
numbers of the set $S^{[n]}$.

\begin{algorithm}
\label{alg:computingbettiofarrangements}
\item[]
\item[{\sc Input}] 
A family ${\mathcal S} = \{S_1,\ldots,S_n\}$ 
of closed and bounded  semi-algebraic sets of $\R^k$ of constant 
description complexity.
\item[{\sc Output}] 
$b_0(S^{[n]}),\ldots, b_\ell(S^{[n]})$.
\item[{\sc Procedure}]
\item[Step 1] Using the algorithm implicit in Theorem \ref{the:triangulation} 
compute an  $\ell$-adaptive triangulation, $h_{\ell}({\mathcal S})$.
\item[Step 2]
Compute the matrices corresponding to the differentials in
the co-chain complex of the 
the cell complex ${\mathcal K}_\ell({\mathcal S})$.
\item[Step 3]
Compute using standard algorithms from linear algebra for computing 
dimensions of images and kernels of linear maps  the dimensions
of the cohomology groups of the complex
$\Ch^{\bullet}({\mathcal K}_\ell({\mathcal S}))$.
\item[Step 4]
For $0 \leq i \leq \ell$ output 
\[
b_i(S^{[n]}) = \dim \HH^i(\Ch^{\bullet}({\mathcal K}_\ell({\mathcal S}))).
\]
\end{algorithm}

\noindent\textsc{Proof of Correctness:}
The correctness of the algorithm is a consequence of 
Theorem \ref{the:homologyoftruncated} and Proposition \ref{prop:hocolimit2}.
\qed

\noindent\textsc{Complexity Analysis:}
There are clearly at most 
$\displaystyle{
\sum_{i}^{\ell+2} {n \choose i} = O(n^{\ell+2})
}
$ 
calls to  the triangulation algorithm. 
Each such call takes constant time under the assumption that the input
sets have constant description complexity. Thus, the total number
of algebraic operations (involving the coefficients of the
input polynomials) is bounded by $O(n^{\ell+2})$.
Additionally, one has to perform linear algebra on matrices of 
size bounded by $O(n^{\ell+2})$. 
\qed

\section{Open Problems}
\label{sec:open}
We list here some interesting open problems some of which could possibly
be tackled in the near future. \\

\paragraph{\em Computing Betti Numbers in Single Exponential Time ?}

Suppose $S \subset \R^k$ is a semi-algebraic set defined in terms
of $s$ polynomials, of degrees bounded by $d$. One of the most fundamental
open questions in algorithmic semi-algebraic geometry, is whether
there exists a single exponential (in $k$) time algorithm for computing
the Betti numbers of $S$. 
The best we can do so far is summarized in Theorem \ref{the:bettifew}
which gives the existence of single exponential time algorithms for
computing the first $\ell$ Betti numbers of $S$ for any constant
$\ell$. 
A big challenge is to extend these ideas to design an algorithm
for computing all the Betti numbers of $S$. \\

\paragraph{\em Are the Middle Betti Numbers Harder to Compute ?}

From the algorithm design perspective it seems that computing the 
lowest (as well as the highest) Betti numbers of semi-algebraic sets, as
well the Euler-Poincar\'e characteristic of semi-algebraic sets,
are easier than computing the ``middle'' Betti numbers. Is there a
complexity-theoretic hardness result that would justify this fact? 
In certain mathematical contexts 
(for instance, the topology of smooth projective
complex varieties) the middle Betti numbers
contain all the information. Is there a complexity-theoretic analogue of 
this phenomenon that would justify our experience that certain Betti
numbers are harder to compute than the others? \\

\paragraph{\em More Efficient Algorithms for 
Computing the Number of Connected Components in the Quadratic
Case ?}

For semi-algebraic sets in $\R^k$ defined by $\ell$ quadratic inequalities,
there are algorithms for deciding emptiness, as well as computing
sample points in every connected component whose complexity is
bounded by $k^{O(\ell)}$ \cite{Barvinok93,GP}. 
We also have an algorithm \cite{Basu6}
for computing the Euler-Poincar\'e characteristic of such sets whose
complexity is  $k^{O(\ell)}$. 
However, the best known algorithm for computing
the number of connected components of such sets 
has complexity $k^{2^{O(\ell)}}$ (as a special case of the 
algorithm for computing all the Betti numbers given in Theorem
\ref{the:quadratic}). This raises the question whether there
exists a more efficient algorithm with complexity $k^{O(\ell)}$ or even
$k^{O(\ell^2)}$ for counting the number of connected components of such sets.
Roadmap type constructions
used for counting connected components in the case of 
general semi-algebraic sets cannot be directly employed in this context,
because such algorithms will have complexity exponential in $k$. \\

\paragraph{\em More Efficient Algorithms for 
Computing the Number of Connected Components for General Semi-algebraic Sets ?}
A very interesting open question is whether the exponent $O(k^2)$ 
in the complexity of roadmap algorithms 
(cf. Theorem \ref{16:the:saconnecting}) can be improved to $O(k)$, so that the
complexity of testing connectivity becomes asymptotically the same as that
of testing emptiness of a semi-algebraic set (cf. Theorem 
\ref{13:the:samplealg}).

Such an improvement
would go a long way in making this algorithm practically useful.
It would also be of interest for studying 
metric properties of semi-algebraic sets because of the following.
Applying Crofton's formula from integral geometry (see for example
\cite{Santalo}) one immediately
obtains as a corollary of Theorem \ref{16:the:saconnecting}
(using the same notation as in the theorem) 
an upper bound of $s^{k'+1}d^{O(k^2)}$
on the length of a semi-algebraic connecting
path connecting two points in any connected component of $S$
(assuming that $S$ is contained in the unit ball centered at the origin).
An improvement in the complexity of algorithms for constructing
connecting paths (such as the roadmap algorithm) 
would also improve the bound on the length of connecting paths.
Recent results due to D'Acunto and  Kurdyka \cite{DK} show that 
it is possible to construct  semi-algebraic paths
of length $d^{O(k)}$ between two points of $S$
(assuming that $S$ is a connected
component of a real algebraic set contained in the  unit ball
defined by polynomials of degree $d$). 
However, the semi-algebraic complexity
of such paths cannot be bounded in terms of the parameters $d$ and $k$.
The improvement in the complexity suggested above, apart from
its algorithmic significance,  would also be an
effective version of the results in \cite{DK}.

\section*{Acknowledgment}
The author thanks Richard Pollack for his 
careful reading and the anonymous referees for many
helpful comments which helped to substantially improve the article.

\end{document}

%% file: cellcomplexofsphere2.pdftex_t
\begin{picture}(0,0)%
\includegraphics{cellcomplexofsphere2.pdf}%
\end{picture}%
\setlength{\unitlength}{3947sp}%
\begingroup\makeatletter\ifx\SetFigFont\undefined%
\gdef\SetFigFont#1#2#3#4#5{%
  \reset@font\fontsize{#1}{#2pt}%
  \fontfamily{#3}\fontseries{#4}\fontshape{#5}%
  \selectfont}%
\fi\endgroup%
\begin{picture}(7111,6214)(1651,-7094)
\put(8101,-3961){\makebox(0,0)[lb]{\smash{{\SetFigFont{12}{14.4}{\familydefault}{\mddefault}{\updefault}{\color[rgb]{0,0,0}$c_2^+$}%
}}}}
\put(4951,-1036){\makebox(0,0)[lb]{\smash{{\SetFigFont{12}{14.4}{\familydefault}{\mddefault}{\updefault}{\color[rgb]{0,0,0}$c_0^+$}%
}}}}
\put(4876,-7036){\makebox(0,0)[lb]{\smash{{\SetFigFont{12}{14.4}{\familydefault}{\mddefault}{\updefault}{\color[rgb]{0,0,0}$c_0^-$}%
}}}}
\put(4801,-4936){\makebox(0,0)[lb]{\smash{{\SetFigFont{12}{14.4}{\familydefault}{\mddefault}{\updefault}{\color[rgb]{0,0,0}$c_1^-$}%
}}}}
\put(4801,-3061){\makebox(0,0)[lb]{\smash{{\SetFigFont{12}{14.4}{\familydefault}{\mddefault}{\updefault}{\color[rgb]{0,0,0}$c_1^+$}%
}}}}
\put(1651,-3886){\makebox(0,0)[lb]{\smash{{\SetFigFont{12}{14.4}{\familydefault}{\mddefault}{\updefault}{\color[rgb]{0,0,0}$c_2^-$}%
}}}}
\end{picture}%

%% file: bettioneexample2.pdftex_t
\begin{picture}(0,0)%
\includegraphics{bettioneexample2.pdf}%
\end{picture}%
\setlength{\unitlength}{3947sp}%
\begingroup\makeatletter\ifx\SetFigFont\undefined%
\gdef\SetFigFont#1#2#3#4#5{%
  \reset@font\fontsize{#1}{#2pt}%
  \fontfamily{#3}\fontseries{#4}\fontshape{#5}%
  \selectfont}%
\fi\endgroup%
\begin{picture}(5351,4516)(2551,-5319)
\put(2551,-2911){\makebox(0,0)[lb]{\smash{{\SetFigFont{12}{14.4}{\familydefault}{\mddefault}{\updefault}{\color[rgb]{0,0,0}$S_0$}%
}}}}
\put(5251,-2911){\makebox(0,0)[lb]{\smash{{\SetFigFont{12}{14.4}{\familydefault}{\mddefault}{\updefault}{\color[rgb]{0,0,0}$S_1$}%
}}}}
\put(7426,-2836){\makebox(0,0)[lb]{\smash{{\SetFigFont{12}{14.4}{\familydefault}{\mddefault}{\updefault}{\color[rgb]{0,0,0}$S_2$}%
}}}}
\end{picture}%

%% file: example2.pdftex_t
\begin{picture}(0,0)%
\includegraphics{example2.pdf}%
\end{picture}%
\setlength{\unitlength}{3947sp}%
\begingroup\makeatletter\ifx\SetFigFont\undefined%
\gdef\SetFigFont#1#2#3#4#5{%
  \reset@font\fontsize{#1}{#2pt}%
  \fontfamily{#3}\fontseries{#4}\fontshape{#5}%
  \selectfont}%
\fi\endgroup%
\begin{picture}(6926,6214)(1651,-7094)
\put(8101,-3961){\makebox(0,0)[lb]{\smash{{\SetFigFont{12}{14.4}{\familydefault}{\mddefault}{\updefault}{\color[rgb]{0,0,0}$P_2$}%
}}}}
\put(4951,-1036){\makebox(0,0)[lb]{\smash{{\SetFigFont{12}{14.4}{\familydefault}{\mddefault}{\updefault}{\color[rgb]{0,0,0}$H_1$}%
}}}}
\put(4876,-7036){\makebox(0,0)[lb]{\smash{{\SetFigFont{12}{14.4}{\familydefault}{\mddefault}{\updefault}{\color[rgb]{0,0,0}$H_2$}%
}}}}
\put(4801,-3061){\makebox(0,0)[lb]{\smash{{\SetFigFont{12}{14.4}{\familydefault}{\mddefault}{\updefault}{\color[rgb]{0,0,0}$C_1$}%
}}}}
\put(4801,-4936){\makebox(0,0)[lb]{\smash{{\SetFigFont{12}{14.4}{\familydefault}{\mddefault}{\updefault}{\color[rgb]{0,0,0}$C_2$}%
}}}}
\put(1651,-3886){\makebox(0,0)[lb]{\smash{{\SetFigFont{12}{14.4}{\familydefault}{\mddefault}{\updefault}{\color[rgb]{0,0,0}$P_1$}%
}}}}
\end{picture}%

%% file: exp3.pdftex_t
\begin{picture}(0,0)%
\includegraphics{exp3.pdf}%
\end{picture}%
\setlength{\unitlength}{3947sp}%
\begingroup\makeatletter\ifx\SetFigFont\undefined%
\gdef\SetFigFont#1#2#3#4#5{%
  \reset@font\fontsize{#1}{#2pt}%
  \fontfamily{#3}\fontseries{#4}\fontshape{#5}%
  \selectfont}%
\fi\endgroup%
\begin{picture}(6166,4964)(1718,-5244)
\put(4651,-436){\makebox(0,0)[b]{\smash{{\SetFigFont{12}{14.4}{\familydefault}{\mddefault}{\updefault}{\color[rgb]{0,0,0}$\sigma_1$}%
}}}}
\put(6226,-1486){\makebox(0,0)[b]{\smash{{\SetFigFont{12}{14.4}{\familydefault}{\mddefault}{\updefault}{\color[rgb]{0,0,0}$\sigma_2$}%
}}}}
\put(5401,-2761){\makebox(0,0)[lb]{\smash{{\SetFigFont{12}{14.4}{\familydefault}{\mddefault}{\updefault}{\color[rgb]{0,0,0}$\tau$}%
}}}}
\end{picture}%

%% file: exp4.pdftex_t
\begin{picture}(0,0)%
\includegraphics{exp4.pdf}%
\end{picture}%
\setlength{\unitlength}{3947sp}%
\begingroup\makeatletter\ifx\SetFigFont\undefined%
\gdef\SetFigFont#1#2#3#4#5{%
  \reset@font\fontsize{#1}{#2pt}%
  \fontfamily{#3}\fontseries{#4}\fontshape{#5}%
  \selectfont}%
\fi\endgroup%
\begin{picture}(7425,5639)(751,-5244)
\put(6301,-4786){\makebox(0,0)[b]{\smash{{\SetFigFont{12}{14.4}{\familydefault}{\mddefault}{\updefault}{\color[rgb]{0,0,0}$D_\tau \cap D_{\sigma_2,\tau}$}%
}}}}
\put(5551,-2836){\makebox(0,0)[b]{\smash{{\SetFigFont{12}{14.4}{\familydefault}{\mddefault}{\updefault}{\color[rgb]{0,0,0}$D_\tau$}%
}}}}
\put(8176,-1336){\makebox(0,0)[b]{\smash{{\SetFigFont{12}{14.4}{\familydefault}{\mddefault}{\updefault}{\color[rgb]{0,0,0}$D_{\sigma_2,\tau}$}%
}}}}
\put(7351,-736){\makebox(0,0)[b]{\smash{{\SetFigFont{12}{14.4}{\familydefault}{\mddefault}{\updefault}{\color[rgb]{0,0,0}$D_{\sigma_2}$}%
}}}}
\put(6526,-61){\makebox(0,0)[b]{\smash{{\SetFigFont{12}{14.4}{\familydefault}{\mddefault}{\updefault}{\color[rgb]{0,0,0}$D_{\sigma_1,\sigma_2}\cap D_{\sigma_2}$}%
}}}}
\put(5251,239){\makebox(0,0)[b]{\smash{{\SetFigFont{12}{14.4}{\familydefault}{\mddefault}{\updefault}{\color[rgb]{0,0,0}$D_{\sigma_1,\sigma_2}$}%
}}}}
\put(751,-1561){\makebox(0,0)[b]{\smash{{\SetFigFont{12}{14.4}{\familydefault}{\mddefault}{\updefault}{\color[rgb]{0,0,0}$D_{\tau} \cap D_{\sigma_1,\tau} \cap D_{\sigma_2,\tau}$}%
}}}}
\put(2176,-361){\makebox(0,0)[b]{\smash{{\SetFigFont{12}{14.4}{\familydefault}{\mddefault}{\updefault}{\color[rgb]{0,0,0}$D_{\sigma_1,\tau}$}%
}}}}
\end{picture}%

%% file: exp5.pdftex_t
\begin{picture}(0,0)%
\includegraphics{exp5.pdf}%
\end{picture}%
\setlength{\unitlength}{3947sp}%
\begingroup\makeatletter\ifx\SetFigFont\undefined%
\gdef\SetFigFont#1#2#3#4#5{%
  \reset@font\fontsize{#1}{#2pt}%
  \fontfamily{#3}\fontseries{#4}\fontshape{#5}%
  \selectfont}%
\fi\endgroup%
\begin{picture}(9924,7824)(1189,-6973)
\put(6976,239){\makebox(0,0)[lb]{\smash{{\SetFigFont{12}{14.4}{\familydefault}{\mddefault}{\updefault}{\color[rgb]{0,0,0}${\mathcal F}(\sigma_1,v)$}%
}}}}
\put(8101,-5011){\makebox(0,0)[lb]{\smash{{\SetFigFont{12}{14.4}{\familydefault}{\mddefault}{\updefault}{\color[rgb]{0,0,0}${\mathcal F}(\sigma_2,v)$}%
}}}}
\end{picture}%

%% file: ams_final_submission.bbl
\begin{thebibliography}{1}
\bibitem{Agarwal}
{\sc P.K.\ Agarwal, M.\ Sharir}
\newblock{\em Arrangements and their applications,}
\newblock{Chapter in {Handbook of Computational Geometry,} 
{\sc J.R. Sack, J. Urrutia} (Ed.),  North-Holland, 49-120, 2000.}

\bibitem{Agrachev}
{\sc A.A.\ Agrachev}
\newblock {Topology of quadratic maps and Hessians of smooth maps},
\newblock {Algebra, Topology, Geometry, Vol 26 (Russian),85-124, 162,
Itogi Nauki i Tekhniki, Akad. Nauk SSSR, Vsesoyuz. Inst. Nauchn.i 
Tekhn. Inform., Moscow, 1988.
Translated in {\em J. Soviet Mathematics}. 49 (1990), no. 3, 990-1013.}

\bibitem{Barvinok93}
{\sc A.I.\ Barvinok}
\newblock {Feasibility Testing for Systems of Real Quadratic Equations},
\newblock {\em Discrete and Computational Geometry}, 10:1-13 (1993).

\bibitem{Barvinok97}
{\sc A. I.\ Barvinok}
\newblock{On the Betti numbers of semi-algebraic sets defined by few quadratic
inequalities},
\newblock{\em Mathematische Zeitschrift}, 225, 231-244 (1997).

\bibitem{Basu1}
{\sc S.\ Basu} 
\newblock {On Bounding the Betti Numbers and Computing the Euler
Characteristics of Semi-algebraic Sets},
\newblock {\em Discrete and Computational Geometry}, 22:1-18, 1999.

\bibitem{Basu2}
{\sc S.\ Basu} 
\newblock{New Results on Quantifier Elimination Over 
Real Closed Fields and Applications to Constraint Databases},
\newblock {\em Journal of the ACM}, July 1999, Vol 46, No 4. 537-555.

\bibitem{Basu4}
{\sc S.\ Basu} 
\newblock{On different bounds on different Betti numbers of 
semi-algebraic sets,}
\newblock {\em Discrete and Computational Geometry}, 30:1, 65-85, 2003.

\bibitem{Basu5}
{\sc S.\ Basu} 
\newblock{Computing Betti Numbers of Arrangements  via Spectral Sequences},
\newblock {\em Journal of Computer and System Sciences}, 67 (2003) 244-262.

\bibitem{Basu6}
{\sc S.\ Basu}
\newblock {Efficient algorithm for computing the 
Euler-Poincar\'e characteristic of semi-algebraic sets defined 
by few quadratic inequalities},
\newblock {\em Computational Complexity}, 15 (2006), 236-251.

\bibitem{Basu7}
{\sc S.\ Basu}
\newblock {Computing the first few Betti numbers of semi-algebraic sets in single exponential time},
\newblock {\em Journal of Symbolic Computation},
Volume 41, Issue 10, October 2006, 1125-1154.

\bibitem{Basu8}
{\sc S.\ Basu}
 \newblock {Computing the top few Betti numbers of semi-algebraic sets defined 
by quadratic inequalities in polynomial time},
\newblock {\em Foundations of Computational Mathematics} (in press),
available at [arXiv:math.AG/0603262].

\bibitem{Basu9}
{\sc S.\ Basu}
\newblock{Combinatorial complexity in o-minimal geometry},
Available at [arXiv:math.CO/0612050].
(An extended abstract appears in the 
Proceedings of the ACM Symposium on the Theory of Computing, 2007).

\bibitem{Basu10}
{\sc S.\ Basu}
 \newblock {On the number of topological types occurring in a 
parametrized family of arrangements},
\newblock preprint, available at [arXiv:0704.0295].

\bibitem{BP'R07}
{\sc S.\ Basu, D.\ Pasechnik, M.-F. Roy}  
\newblock{Betti numbers of semi-algebraic sets defined by partly 
quadratic systems of polynomials},
\newblock preprint, available at [arXiv:0707.4333].

\bibitem{BPR3}
{\sc S.\ Basu, R.\ Pollack, M.-F. \ Roy} 
\newblock{On Computing a Set of Points meeting every
Semi-algebraically Connected Component of  a Family of
Polynomials on a Variety},
\newblock{\em Journal of Complexity}, March 1997, Vol 13, Number 1, 28-37.

\bibitem{BPR4}
{\sc S.\ Basu, R.\ Pollack, M.-F. \ Roy} 
\newblock{ On the 
combinatorial and algebraic complexity of Quantifier Elimination}, 
\newblock{\em Journal of the ACM}, Vol 43, Number 6, 1002-1046, 1996.

\bibitem{BPR5}
{\sc S.\ Basu, R.\ Pollack, M.-F. \ Roy} 
\newblock{Constructing roadmaps of semi-algebraic sets on a
variety},
\newblock{\em Journal of the American Mathematical Society}
13 (2000), 55-82. 

\bibitem{BPR6}
{\sc S.\ Basu, R.\ Pollack, M.-F. \ Roy} 
\newblock{Computing the Euler-Poincar\'e Characteristic of Sign Conditions},
\newblock {\em Computational Complexity}, 14 (2005) 53-71.

\bibitem{BPR7}
{\sc S.\ Basu, R.\ Pollack, M.-F. \ Roy} 
\newblock{Computing the Dimension of  a Semi-Algebraic Set},
\newblock{\em Zap. Nauchn. Semin. POMI 316}, 42-54 (2004).  

\bibitem{BPR9}
{\sc S.\ Basu, R.\ Pollack, M.-F. \ Roy}
 \newblock {Computing the first Betti number and the connected components
of semi-algebraic sets,}
\newblock to appear in 
{\em Foundations of Computational Mathematics}, available at
[arXiv:math.AG/0603248].

\bibitem{BPR10}
{\sc S.\ Basu, R.\ Pollack, M.-F. \ Roy}
\newblock{Betti Number Bounds, Applications and Algorithms,}
\newblock {\em Current Trends in Combinatorial and Computational Geometry: 
Papers from the Special Program at MSRI},
MSRI Publications Volume 52, 
Cambridge University Press 2005, 87-97.

\bibitem{BPRbook06}
{\sc S.\ Basu, R.\ Pollack, M.-F. \ Roy}
\newblock{Algorithms in Real Algebraic Geometry}
Series: Algorithms and Computation in Mathematics, Vol 10, Second Edition.
Springer-Verlag (2006).

\bibitem{BK05}
{\sc S.\ Basu, M.\ Kettner}
\newblock{Computing the Betti numbers of arrangements in practice,}
\newblock Proceedings of the 8-th International Workshop
on Computer Algebra in Scientific Computing (CASC), LNCS 3718, 13-31, 2005.
 
\bibitem{BZ}
{\sc S.\ Basu, T.\ Zell}
\newblock{On projections of semi-algebraic sets defined by few 
quadratic inequalities,}
\newblock to appear in {\em Discrete and Computational Geometry},
available at [arXiv:math.AG/0602398].


\bibitem{BCR}
{\sc J.\ Bochnak, M.\ Coste, M.-F.\ Roy} 
\newblock G\'eom\'etrie alg\'ebrique r\'eelle. 
Springer-Verlag (1987). 

\bibitem{Bjorner}
{\sc A.\ Bj\"{o}rner} 
\newblock Topological methods, in Handbook of Combinatorics, vol II, 1819-1872,
R. Graham, M. Grotschel, and L. Lovasz Eds., North-Holland/Elsevier (1995).

\bibitem{BC}
{\sc P.\ Burgisser, F.\ Cucker}
\newblock {Counting Complexity Classes for Numeric Computations II: 
Algebraic and  Semi-algebraic Sets},
\newblock  {\em Journal of Complexity}, 22(2):147-191 (2006).

\bibitem{BC2}
{\sc P.\ Burgisser, F.\ Cucker}
\newblock {Variations by complexity theorists on three themes of
Euler, B\'ezout, Betti, and Poincar\'e},
\newblock In: Complexity of computations and proofs, Jan Krajicek (ed.),
{\em Quaderni di Matematica} 13, pp. 73-152, 2005.

\bibitem{BCSS}
\newblock {\sc L.\ Blum, F.\ Cucker, M.\ Shub, S.\ Smale, }
\newblock {Complexity and Real Computation},
\newblock Springer-Verlag, 1997.
 
\bibitem{Canny}
{\sc J.\ Canny}
\newblock  {Computing road maps in general semi-algebraic sets},
\newblock {\em Computer Journal}, 36: 504-514, (1993).

\bibitem{CGV}
{\sc J. Canny, D. Grigor'ev, N. Vorobjov}
\newblock{Finding connected components
of a semi-algebraic set in subexponential time},
\newblock {\em Appl. Algebra Eng. Commun. Comput.}, 2, No.4, 217-238 (1992).

\bibitem{CEGS91}
{\sc B.\ Chazelle, H.\ Edelsbrunner, L.J.\ Guibas, M.\ Sharir}
\newblock{A single-exponential stratification scheme
for real semi-algebraic varieties and its applications},
\newblock{\em Theoretical Computer Science}, 84, 77-105, 1991.

\bibitem{CEGSW}
{\sc K.\ Clarkson, H.\ Edelsbrunner, L.J.\ Guibas, M.\ Sharir, E.\ Welzl}
\newblock{Combinatorial complexity bounds for arrangements of curves and 
spheres},
\newblock{\em Discrete and Computational Geometry},  5:99 - 160, 1990.

\bibitem{Collins}
{\sc G. E.\ Collins}
\newblock {Quantifier elimination for real closed fields by 
cylindrical algebraic decomposition},
\newblock Springer Lecture Notes in Computer Science 33, 515-532.

\bibitem{DK}
{\sc D.\ D'Acunto, K.\ Kurdyka}
\newblock {Bounds for gradient trajectories and geodesic diameters of real
algebraic sets},
\newblock preprint.

\bibitem{Deligne1}
{\sc P.\ Deligne}
\newblock {La conjecture de Weil (I)},
\newblock Publications Math. IHES, 43, 1974.

\bibitem{Deligne2}
{\sc P.\ Deligne}
\newblock {La conjecture de Weil (II)},
\newblock Publications Math. IHES, 52, 1980.

\bibitem{Dwork}
{\sc B.\ Dwork}
\newblock{On the Rationality of the Zeta Function of an Algebraic Variety},
\newblock{\em American Journal of Mathematics}, Vol. 82, No. 3, 631-648, 1960.

\bibitem{GR92}
{\sc L. Gournay, J. J. Risler}
\newblock{Construction of roadmaps of semi-algebraic sets},
\newblock{\em Appl. Algebra Eng. Commun. Comput.} 4, No.4, 239-252 (1993).

\bibitem{GV05}
{\sc A.\ Gabrielov, N.\ Vorobjov}
\newblock {Betti numbers of semi-algebraic sets defined by 
quantifier-free formulae,}
\newblock {\em Discrete and Computational Geometry}, 33:395-401, 2005.

\bibitem{GVZ}
{\sc A.\ Gabrielov, N.\ Vorobjov, T.\ Zell}
\newblock {Betti Numbers of Semi-algebraic and Sub-Pfaffian Sets},
\newblock {\em J. London Math. Soc.} (2) 69 (2004) 27-43.

\bibitem{Gri88}
{\sc D.\ Grigor'ev}
\newblock {The Complexity of deciding Tarski algebra}, 
\newblock {\em Journal of Symbolic Computation} 5  65-108 (1988).

\bibitem{GV}
{\sc D. Grigor'ev, N.\ Vorobjov}
\newblock {Solving Systems of Polynomial Inequalities
in Subexponential Time},
\newblock {\em Journal of  Symbolic Computation}, 5  37--64 (1988).

\bibitem{GV92}
{\sc D. Grigor'ev, N. Vorobjov}
\newblock{Counting connected components of a semi-algebraic
set in subexponential time},
\newblock {\em Computational Complexity} 2, No.2, 133-186 (1992).

\bibitem{GP}
{\sc D.\ Grigor'ev, D.V.\ Pasechnik}
\newblock {Polynomial time computing over quadratic maps I. 
Sampling in real algebraic sets},
\newblock {\em Computational Complexity}, 14:20-52 (2005).

\bibitem{Hardt} 
{\sc R. M. Hardt}
\newblock{Semi-algebraic Local Triviality in Semi-algebraic Mappings},
\newblock {\em Am. J. Math.} { 102}, 291-302 (1980).

\bibitem{HRS94}
{\sc J.\ Heintz, M.-F.\ Roy, P. Solern\`o}
\newblock{Description of the
Connected Components of a Semi-algebraic Set in Single Exponential Time},
\newblock{\em Discrete and Computational Geometry} 11:121-140 (1994).

\bibitem{Koltun}
{\sc V.\ Koltun}
\newblock {Almost Tight Upper Bounds for Vertical Decompositions in Four 
Dimensions},
\newblock{\em Journal of the ACM}, Vol. 51, 699-730, 2004.

\bibitem{Loj2}
{\sc S. {\L}ojasiewicz}
\newblock{Triangulation of semi-analytic sets.}
{\em Ann. Scuola Norm. Sup. Pisa, Sci. Fis. Mat.} (3) 18,
449-474 (1964).

\bibitem{Markov}
{\sc A.\ Markov}
\newblock{Insolubility of the problem of homeomorphy},
\newblock{Proceedings of the International Congress of Mathematicians} (1960),
Cambridge University Press, 300-306.

\bibitem{Matousek}
{\sc J.\ Matousek}
\newblock{Lectures on Discrete Geometry},
\newblock Springer-Verlag (2002).

\bibitem{Mcleary}
{\sc J. \ McCleary}
\newblock A User's Guide to Spectral Sequences, Second Edition 
\newblock Cambridge Studies in Advanced Mathematics, 2001.

\bibitem{Milnor}
{\sc J. \ Milnor}
\newblock {On the Betti numbers of real varieties},
\newblock {\em Proc. Amer. Math. Soc.} 15, 275-280, (1964).


\bibitem{Morita}
{\sc S. \ Morita}
\newblock {Geometry of Characteristic Classes},
\newblock {Translations of Mathematical Monographs}, Vol 199, 
American Mathematical Society (1999).

\bibitem{OT}
{\sc T.\ Oaku, N.\ Takayama}
\newblock {An algorithm for de Rham cohomology groups of the complement of an 
affine variety via D-module computation}, 
\newblock {\em Journal of Pure and Applied Algebra} 139:201-233, 1999.

\bibitem{OP}
{\sc O. A.\ Oleinik, I. B.\ Petrovskii}
\newblock {On the topology of real algebraic surfaces},
\newblock {\em Izv. Akad. Nauk SSSR} 13, 389-402, (1949).

\bibitem{R92}
{\sc J.\  Renegar}
\newblock {On the computational complexity and 
geometry of  the  first order theory of the reals}, 
\newblock {\em Journal  of Symbolic Computation}, 255-352 (1992).

\bibitem{SS}
{\sc J.\ Schwartz, M.\ Sharir}
\newblock {On the `piano movers' problem II. General
techniques for computing topological properties of real algebraic manifolds}, 
\newblock {\em Adv. Appl. Math.} 4, 298-351 (1983).

\bibitem{Seidenberg54}
{\sc A.\ Seidenberg}
\newblock {A new decision method for elementary algebra},
\newblock {\em Annals of Mathematics}, 60:365-374, (1954).

\bibitem{Tarski51}
{\sc A.\ Tarski} 
\newblock  {A Decision method for elementary algebra and 
geometry}, 
\newblock University of California Press (1951).

\bibitem{Rotman}
{\sc J.\ J. Rotman}
\newblock {An Introduction to Algebraic Topology},
\newblock Springer Verlag, 1988.

\bibitem{Santalo}
{\sc L.\ Santalo}
\newblock {Integral Geometry and Geometric Probability},
\newblock Cambridge University Press, 2nd edition (2002).

\bibitem{Smale}
{\sc S.\ Smale}
\newblock {A Vietoris mapping theorem for homotopy},
\newblock {\em Proc. Amer. Math. Soc.} 8:3, 604-610 (1957).

\bibitem{Spanier}
{\sc E.\ H. Spanier}
\newblock {Algebraic Topology},
\newblock Springer, 3rd Edition (1994).

\bibitem{Thom}
{\sc R.\ Thom}
\newblock {Sur l'homologie des varietes algebriques reelles}, 
\newblock Differential and Combinatorial Topology,
\newblock Ed. S.S. Cairns, Princeton Univ. Press, 255-265, (1965).

\bibitem{Walther1}
{\sc U. \ Walther} 
\newblock{Algorithmic Determination of the Rational Cohomology of Complex
Varieties via Differential Forms},
\newblock{\em Contemporary Mathematics} (286), 2001.

\bibitem{Walther2}
{\sc U. \ Walther} 
\newblock{D-modules and cohomology of varieties},
\newblock{Computations in algebraic geometry using Macaulay 2},
{\sc D.\ Eisenbud, D.R.\ Grayson, M.\ Stillman, B.\ Sturmfels}  
eds., Springer, 2001.
 
\bibitem{Weil}
{\sc A. \ Weil} 
\newblock{Number of solutions of  equations over finite fields},
\newblock{\em Bulletin of the American Mathematical Society}, 55:497-508, 1949.

\bibitem{Whitehead} 
{\sc G.W.\ Whitehead}
\newblock{Elements of  Homotopy Theory},
\newblock Graduate Texts in Mathematics, Springer-Verlag (1978).
\end{thebibliography}
